\documentclass[12pt]{article}
\usepackage{amsmath,amssymb,amscd,amsthm}
\usepackage{mathtools}
\mathtoolsset{showonlyrefs}

\newcounter{myitem}[section]
\renewcommand{\themyitem}{\thethm.\arabic{myitem}}

\newtheorem{thm}{Theorem}[section]
\newtheorem{prop}[thm]{Proposition}
\newtheorem{lem}[thm]{Lemma}
\newtheorem{cor}[thm]{Corollary}

\theoremstyle{definition}
\newtheorem{defn}[thm]{Definition}
\newtheorem{rmk}[thm]{Remark}
\newtheorem{para}[thm]{}

\newenvironment{mylist}{
\setcounter{myitem}{\value{equation}}
\begin{list}{}{%
 \setlength{\topsep}{5pt}%
 \setlength{\parsep}{5pt}%
 \setlength{\itemsep}{0pt}%
 \setlength{\leftmargin}{50pt}%
 \setlength{\labelwidth}{50pt}%
}}
{\end{list}}

\newcommand{\itemno}{
\refstepcounter{myitem}
\item[{\rm (\themyitem)}]
\setcounter{equation}{\value{myitem}}}

\numberwithin{equation}{thm}

\newcommand{\RomI}{\uppercase\expandafter{\romannumeral 1}}
\newcommand{\RomII}{\uppercase\expandafter{\romannumeral 2}}
\newcommand{\RomIII}{\uppercase\expandafter{\romannumeral 3}}

\newcommand{\bC}{\mathbb C}
\newcommand{\bN}{\mathbb N}
\newcommand{\bQ}{\mathbb Q}
\newcommand{\bZ}{\mathbb Z}

\newcommand{\cA}{\mathcal A}
\newcommand{\cC}{\mathcal C}
\newcommand{\cL}{\mathcal L}
\newcommand{\cO}{\mathcal O}
\newcommand{\cS}{\mathcal S}

\newcommand{\ca}{complex analytic }
\newcommand{\cmh}{cohomological mixed Hodge }
\newcommand{\coh}{{\rm H}}
\newcommand{\del}{\partial}
\newcommand{\gp}{^{\rm gp}}
\newcommand{\si}{\sigma}
\newcommand{\Slash}[1]{{\ooalign{\hfil/\hfil\cr\cr$#1$}}}
\newcommand{\uln}{\underline}
\newcommand{\ula}{\underline{\lambda}}
\newcommand{\umu}{\underline{\mu}}
\newcommand{\unu}{\underline{\nu}}
\newcommand{\usi}{\underline{\sigma}}
\newcommand{\uta}{\underline{\tau}}

\DeclareMathOperator{\C}{\mathcal C}
\DeclareMathOperator{\dlog}{dlog}
\DeclareMathOperator{\dprod}{\prod^{\circ}}
\DeclareMathOperator{\gr}{Gr}
\DeclareMathOperator{\id}{id}
\DeclareMathOperator{\image}{Image}
\DeclareMathOperator{\kernel}{Ker}
\DeclareMathOperator{\kos}{Kos}
\DeclareMathOperator*{\limind}{\varinjlim}
\DeclareMathOperator{\lp}{\langle}
\DeclareMathOperator{\pr}{pr}
\DeclareMathOperator{\prim}{prim}
\DeclareMathOperator{\red}{red}
\DeclareMathOperator{\res}{Res}
\DeclareMathOperator{\rp}{\rangle}
\DeclareMathOperator{\tr}{Tr}

\begin{document}

\title{Polarizations on limiting mixed Hodge structures}
\author{Taro Fujisawa \\
Tokyo Denki University \\
e-mail: fujisawa@mail.dendai.ac.jp}
\date{}
\footnotetext[0]
{\hspace{-18pt}2010 {\itshape Mathematics Subject Classification}.
Primary 14C30; Secondary 14D07, 32G20, 32S35. }

\maketitle

\begin{abstract}
We construct polarizations
of mixed Hodge structures
on the relative log de Rham cohomology groups
of a projective log deformation.
To this end,
we study the behavior of weight and Hodge filtrations
under the cup product
and construct a trace morphism
for a projective log deformation.
\end{abstract}

\section*{Introduction}

\begin{para}
In \cite{SteenbrinkLE}
Steenbrik introduced the notion of the log deformation
and constructed mixed Hodge structures
on the relative log  de Rham cohomology groups
of a projective log deformation.
In this article,
we construct natural polarizations on these mixed Hodge structures
in the sense of Cattani-Kaplan-Schmid \cite[Definition (2.26)]{CKS}.

A typical example of log deformations
is the singular fiber of a semistable reduction over the unit disc.
For the case of a projective semistable reduction
over the unit disc,
the mixed Hodge structure
on the relative log de Rham cohomology groups of the singular fiber
is considered as the limits of Hodge structures
on the cohomology groups of general fibers,
and called the limiting mixed Hodge structures.
These mixed Hodge structures
were constructed by two different methods,
the transcendental method in \cite{Schmid}
and the algebro-geometric method in \cite{SteenbrinkLHS}.
In fact, Schmid's nilpotent orbit theorem and $SL_2$-orbit theorem
imply that a variation of polarized Hodge structures on the punctured disc
degenerates to a {\slshape polarized} mixed Hodge structure
(\cite[(6.16) Theorem]{Schmid}).
For the case of a projective semistable reduction over the unit disc,
the Hodge structures on the cohomology groups of fibers
induce variations of polarized Hodge structures on the punctured disc.
By applying Schmid's result above
to these variations of polarized Hodge structures,
we obtained the limiting mixed Hodge structures.
Here we note that these mixed Hodge structures are canonically polarized
by their construction.
On the other hand,
Steenbrink constructed mixed Hodge structures
on the relative log de Rham cohomology groups
of the singular fiber of a projective semistable reduction over the unit disc
by algebro-geometric methods.
The coincidence between Steenbrink's mixed Hodge structures
and Schmid's mixed Hodge structures
was proved in \cite[4.2.5 Remarque]{MorihikoSaitoMHP}
and in \cite[(A.1)]{UsuiMTTS} independently.
The motivation of this article
is to construct the polarizations on the limiting mixed Hodge structures
by algebro-geometric methods
in Steenbrink's approach.
Once we obtain polarizations in Theorem \ref{main theorem} below,
the remaining task is to prove
that our polarizations
coincide with the ones given in \cite[(6.16) Theorem]{Schmid}
for the case of a projective semistable reduction.

The main result of this article concerns the question
whether the mixed Hodge structures
on the relative log de Rham cohomology groups
of a projective log deformation
yield nilpotent orbits.
In fact, Kashiwara-Kawai \cite[Proposition 1.2.2]{Kashiwara-KawaiPL}
and Cattani-Kaplan \cite[Theorem (3.13)]{CattaniKaplanDVHS}
show that a polarized mixed Hodge structure
yields a nilpotent orbit
and vice versa.
Therefore the main result of this article
implies that the relative log de Rham cohomology groups
of a projective log deformation
give us nilpotent orbits.
Thus, it is expected that a projective log deformation
yields polarized log Hodge structures
on the standard log point
(see Kato-Usui \cite[2.5]{KatoUsuiAMS})
as a by-product of our main result.
This question is treated
in the forthcoming article \cite{Fujisawa-NakayamaGLHS}.
\end{para}

\begin{para}
Let $Y \longrightarrow \ast$
be a projective log deformation of pure dimension $n$.
In order to put mixed Hodge structures
on the log de Rham cohomology groups $\coh^q(Y,\omega_{Y/\ast})$,
we replace $\omega_{Y/\ast}$ by the weak \cmh complex $K$
defined in \cite[(5.4)]{FujisawaMHSLSD}.
Here we note that the complex $K$ carries
a multiplicative structure
which is compatible with the wedge product on $\omega_{Y/\ast}$.
Then our aim is, more precisely,
to construct polarizations on $\coh^q(Y,K)$ for all $q$.
To this end,
we follow a way similar to the case of compact K\"ahler complex manifolds.
First, we construct a cup product on $\coh^{\ast}(Y,K)$
by using the multiplicative structure on $K$.
Second, we define a trace morphism $\coh^{2n}(Y,K) \longrightarrow \bC$.
Third, we study the property of the cup product with
the class of an ample invertible sheaf in $\coh^2(Y,K)$.
Finally, we prove a kind of positivity for the bilinear form
as a conclusion.
The key ingredient for our argument
is the comparison morphism $\varphi:A \longrightarrow K$,
where $A$ denotes the \cmh complex constructed by Steenbrink
in \cite[Section 5]{SteenbrinkLE} (cf. \cite[Section 4]{SteenbrinkLHS}).
By the fact that $\varphi$ induces isomorphisms of mixed Hodge structures
between $\coh^q(Y,A)$ and $\coh^q(Y,K)$ for all $q$,
we can apply the results
by Guill\'en-Navarro Aznar \cite[(5.1)Th\'eor\`eme]{Guillen-NavarroAznarCI},
or by Morihiko Saito \cite[4.2.5 Remarque]{MorihikoSaitoMHP}
to prove the positivity.

This article is organized as follows:
In Section \ref{section for preliminaries},
we fix the notation and the sign convention used in this article.
Section \ref{Cech complex}
treats the \v{C}ech complex
of a co-cubical complex.
We give the definition of a product morphism
for the \v{C}ech complexes of two co-cubical complexes.
In Section \ref{residue morphism},
we study the residue morphisms
for the log de Rham complex
and for the Koszul complex
of a log deformation.
Section \ref{gysin morphism}
is devoted to the study of the Gysin morphism
for a log deformation.
In Section \ref{comparison iso from A to K},
we first recall the definition of the complex $K$ in \cite{FujisawaMHSLSD}
for the case of a log deformation.
We slightly modify the definition and the notation in \cite{FujisawaMHSLSD}.
Then we recall results of \cite{FujisawaMHSLSD}
in Theorem \ref{theorem for K}.
Next, we study several properties of the Gysin morphism of the complex $K$
for the later use.
Furthermore, we recall the definition of the complex $A$
in Steenbrink \cite{SteenbrinkLE}
and in Fujisawa-Nakayama \cite{Fujisawa-Nakayama}.
Here we also modify the definition of $A$ slightly.
Theorem \ref{Steenbrink's results} restates
the results of \cite{SteenbrinkLE} and of \cite{Fujisawa-Nakayama}.
Then we construct the comparison morphism $\varphi$ from $A$ to $K$
mentioned above.
We prove that the morphism $\varphi$ induces
isomorphisms between $\coh^q(Y,A)$ and $\coh^q(Y,K)$ for all $q$
in this section.
In Section \ref{product},
the multiplicative structures on the complex $K$
and on other related complexes are studied.
The multiplicative structure on $K$ induces
the cup product on $\coh^{\ast}(Y,K)$ mentioned above.
We prove that 
the cup product on $\coh^{\ast}(Y,K)$
satisfies the expected properties for the weight filtration $W$
and the Hodge filtration $F$.
In Section \ref{section for trace morphism},
the trace morphism
$\tr: \coh^{2n}(Y, K_{\bC}) \longrightarrow \bC$ mentioned above
is defined by using the $E_2$-degeneracy of the spectral sequence
$E_r^{p,q}(K_{\bC},W)$.
The cup product on $\coh^{\ast}(Y,K)$ and the trace morphism
induce the bilinear form $Q_K$ on $\coh^{\ast}(Y,K)$.
Combining all these together
in Section \ref{main results},
we prove the main results of this article
by applying the results on bigraded polarized Hodge-Lefschetz modules
in \cite{Guillen-NavarroAznarCI}.
\end{para}

\begin{para}
The results of this article have been already announced
with few proofs in \cite{FujisawaAPLMHS}.
There we restrict ourselves to the case of a semistable reduction
over the unit disc for simplicity.
In this paper we will give the complete proofs for the results.
Moreover, we modify some definitions slightly
and correct several mistakes in \cite{FujisawaAPLMHS}.
\end{para}

\section{Preliminaries}
\label{section for preliminaries}

In this section,
we collect several definitions
which will be used in this article constantly.
We follow \cite[1.3]{Conrad} and \cite[Notation]{NakkajimaSWSS}
for sign conventions.
We recall some of them for the later use.

\begin{para}
The cardinality of a finite set $A$ is denoted by $|A|$.
\end{para}

\begin{para}
For the shift of an increasing filtration $W$,
we use the notation
by Deligne \cite[D\'efinition (1,1,2), (1,1,3)]{DeligneII}.
Namely, we set
\begin{equation*}
W[k]_m=W_{m-k}
\end{equation*}
for every $k,m$.
This notation is different from that
used by Cattani-Kaplan-Schmid \cite[p.475]{CKS}.
For the shift of a decreasing filtration $F$,
we follow the standard notation,
that is,
\begin{equation*}
F[k]^p=F^{p+k}
\end{equation*}
for every $k,p$.
\end{para}

\begin{para}
Let $f: K \longrightarrow L$ be a morphism of complexes.
The complex $(C(f),d)$, called the mapping cone of $f$,
is defined by
\begin{align*}
&C(f)^p=K^{p+1} \oplus L^p \\
&d(x, y)=(-dx , f(x)+dy) \quad x \in K^{p+1}, y \in L^p
\end{align*}
as in \cite{Hartshorne}.
Two morphisms of complexes
\begin{align*}
&\alpha(f): L \longrightarrow C(f) \\
&\beta(f): C(f) \longrightarrow K[1]
\end{align*}
are defined by
\begin{align*}
&\alpha(f)(y) = (0,y) \qquad y \in L^p \\
&\beta(f)(x,y)=-x \qquad x \in K^{p+1}, \ y \in L^p
\end{align*}
for every integer $p$.
(See e.g. \cite[(1.3.3)]{Conrad}, \cite[Notation (4)]{NakkajimaSWSS}.)

For every integer $m$,
we set a morphism
\begin{equation*}
\zeta_m : C(f)[m]^p \longrightarrow C(f[m])^p
\end{equation*}
by $\zeta_m(x,y)=((-1)^mx,y)$
for an element $(x,y) \in C(f)[m]^p=K^{p+m+1} \oplus L^{p+m}$.
It is easy to see that
this defines an isomorphism of complexes
$\zeta_m: C(f)[m] \longrightarrow C(f[m])$.
Then the diagram
\begin{equation}
\label{commutative diagram for the C(f)[m] and C(f[m]):eq}
\begin{CD}
L[m] @>{\alpha(f)[m]}>> C(f)[m] @>{\beta(f)[m]}>> K[m+1] \\
@| @V{\zeta_m}VV @VV{(-1)^m \id}V \\
L[m] @>>{\alpha(f[m])}> C(f[m]) @>>{\beta(f[m])}> K[m+1]
\end{CD}
\end{equation}
is commutative.

Let
\begin{equation}
\label{short exact sequence given:eq}
\begin{CD}
0 @>>> K @>{f}>> L @>{g}>> M @>>> 0
\end{CD}
\end{equation}
be an exact sequence of complexes,
that is,
\begin{equation*}
\begin{CD}
0 @>>> K^p @>{f}>> L^p @>{g}>> M^p @>>> 0
\end{CD}
\end{equation*}
is exact for every $p$.
We define a morphism of complexes
\begin{equation*}
\delta(f,g):
C(f) \longrightarrow M
\end{equation*}
by sending $(x,y) \in C(f)^p=K^{p+1} \oplus L^p$
to $\delta(f,g)(x,y)=g(y) \in M^p$.
It is well known that
this morphism $\delta(f,g)$ is a quasi-isomorphism.
Therefore the diagram
\begin{equation*}
\begin{CD}
M @<{\delta(f,g)}<< C(f) @>{\beta(f)}>> K[1]
\end{CD}
\end{equation*}
gives us a morphism
\begin{equation}
\label{morphism in the derived category from a short exact sequence:eq}
\gamma(f,g):
M \longrightarrow K[1]
\end{equation}
in the derived category.
We can easily check that the morphism
\begin{equation*}
\coh^p(\gamma(f,g)):
\coh^p(M) \longrightarrow \coh^{p+1}(K)
\end{equation*}
induced by the morphism
\eqref{morphism in the derived category from a short exact sequence:eq}
coincides with the classical connecting homomorphism
induced by the short exact sequence
\eqref{short exact sequence given:eq}.

Because we have a commutative diagram
\begin{equation*}
\begin{CD}
M[m] @<<{\delta(f,g)[m]}< C(f)[m] @>>{\beta(f)[m]}> K[m+1] \\
@| @VV{\zeta_m}V @VV{(-1)^m\id}V \\
M[m] @<{\delta(f[m],g[m])}<< C(f[m]) @>{\beta(f[m])}>> K[m+1]
\end{CD}
\end{equation*}
by \eqref{commutative diagram for the C(f)[m] and C(f[m]):eq},
we have the equality
\begin{equation}
\label{relation for gamma and the shift:eq}
\gamma(f,g)[m]=(-1)^m\gamma(f[m],g[m])
\end{equation}
for every $m$.
\end{para}

\begin{para}
For two integers $a,b$,
we identify two complexes
$K[a] \otimes L[b]$ and $(K \otimes L)[a+b]$ as follows
(see \cite[(1.3.6)]{Conrad}).
The morphism
\begin{equation}
\label{tensor and the sift on complexes:eq}
K[a] \otimes L[b] \longrightarrow (K \otimes L)[a+b]
\end{equation}
is given by
\begin{equation*}
x \otimes y
 \mapsto (-1)^{pb} x \otimes y
\end{equation*}
on the component
$K[a]^p \otimes L[b]^q =K^{p+a} \otimes L^{q+b}$.
This gives us the identification expected.
For a morphism of complexes
\begin{equation*}
f:K_1 \otimes K_2 \longrightarrow K_3
\end{equation*}
the morphism of complexes
\begin{equation}
\label{morphism f[a,b]:eq}
f[a,b]:
K_1[a] \otimes K_2[b] \longrightarrow K_3[a+b]
\end{equation}
is the composite of the identification
\eqref{tensor and the sift on complexes:eq}
and the morphism
\begin{equation*}
f[a+b]: (K_1 \otimes K_2)[a+b]
\longrightarrow
K_3[a+b]
\end{equation*}
for every $a,b$.
\end{para}

\begin{para}
\label{gysin morphism for a filtered complex:eq}
Let $K$ be a complex equipped with an increasing filtration $W$.
For an integer $m$,
the exact sequence
\begin{equation*}
\begin{CD}
0 @>>> \gr_{m-1}^WK @>>> W_mK/W_{m-2}K @>>> \gr_m^WK @>>> 0
\end{CD}
\end{equation*}
induces the morphism
\begin{equation}
\label{gysin morphism:eq}
\gr_m^WK \longrightarrow \gr_{m-1}^WK[1]
\end{equation}
in the derived category
as \eqref{morphism in the derived category from a short exact sequence:eq}.
It is called the Gysin morphism of the filtered complex $(K,W)$
and denoted by $\gamma_m(K,W)$.
By \eqref{relation for gamma and the shift:eq},
we have
\begin{equation}
\label{gamma(K[l],W) and gamma(K,W)[l]:eq}
\gamma_m(K[l], W)=(-1)^l\gamma_m(K,W)[l]
\end{equation}
for every $l$.

The morphism
\begin{equation*}
\coh^{p+q}(\gamma_{-p}(K,W)):
\coh^{p+q}(\gr_{-p}^WK) \longrightarrow
\coh^{p+q}(\gr_{-p-1}^WK[1])=\coh^{p+q+1}(\gr_{-p-1}^WK)
\end{equation*}
coincides with the morphism
\begin{equation*}
d_1:E_1^{p,q}(K,W) \longrightarrow E_1^{p+1,q}(K,W)
\end{equation*}
of the $E_1$-terms of the spectral sequences
associated to the filtered complex $(K,W)$
under the identification
$E_1^{p,q}(K,W) \simeq \coh^{p+q}(\gr_{-p}^WK)$.
\end{para}

\begin{para}
Let $(K,d)$ be a complex equipped with an increasing filtration $W$
and $f:K \longrightarrow K[1]$ a morphism of complexes
satisfying the conditions $f^2=0$
and $f(W_mK) \subset W_{m-1}K[1]$ for every $m$.
Since we can easily check $(d+f)^2=0$,
we obtain a complex $(K,d+f)$
which is denoted by $K'$ for a while.
The same $W$ defines an increasing filtration on $K'$.
We have the identity
\begin{equation*}
\gr_m^WK=\gr_m^WK'
\end{equation*}
as complexes for every $m$.
Moreover the morphism $f$ induces
a morphism of complexes
\begin{equation*}
\gr_m^W(f): \gr_m^WK \longrightarrow \gr_{m-1}^WK[1]
\end{equation*}
for every $m$.
The following Proposition is easy to check.
\end{para}

\begin{prop}
\label{gamma for twisted filtered complex}
In the situation above,
we have
\begin{align*}
\gamma_m(K',W)=\gamma_m(K,W)&+\gr_m^W(f) \\
&: \gr_m^WK'=\gr_m^WK
\longrightarrow
\gr_{m-1}^WK[1]=\gr_{m-1}^WK'[1]
\end{align*}
for every integer $m$.
\end{prop}

\begin{para}
Let $K_1, K_2, K_3$ be complexes.
Assume that a morphism of complexes
\begin{equation*}
\varphi: K_1 \otimes K_2 \longrightarrow K_3
\end{equation*}
is given.
Then $\varphi$ induces the morphism
\begin{equation*}
\coh^p(K_1) \otimes \coh^q(K_2) \longrightarrow \coh^{p+q}(K_3)
\end{equation*}
for every $p, q$.
This morphism is denoted by
$\coh^{p,q}(\varphi)$ in this article.

For the case where $K_1, K_2, K_3$ carry increasing filtrations $W$,
if the morphism $\varphi$ satisfies the condition
\begin{equation*}
\varphi(W_aK_1 \otimes W_bK_2) \subset W_{a+b}K_3
\end{equation*}
for all integers $a,b$,
then the morphism $\varphi$ induces the morphisms
\begin{equation*}
\gr_{a,b}^W\varphi:
\gr_a^WK_1 \otimes \gr_b^WK_2 \longrightarrow \gr_{a+b}^WK_3
\end{equation*}
for all $a$ and $b$.
\end{para}

\section{\v{C}ech complexes of co-cubical complexes}
\label{Cech complex}

\begin{para}
Let $\Lambda$ be a non-empty set.
For a positive integer $n$,
$\Lambda^n$ denotes the $n$-times product set of $\Lambda$.
We set
$\prod \Lambda=\coprod_{n > 0}\Lambda^n$.
We consider $\Lambda$ as a subset of $\prod \Lambda$.
We use a symbol $\lambda$ for an element of $\prod\Lambda$.
For an element $\lambda \in \Lambda^{n+1}$,
we set $d(\lambda)=n$.

An element $\lambda \in \Lambda^{k+1}$
is denoted by
\begin{equation*}
\lambda=(\lambda(0),\lambda(1), \dots , \lambda(k)) \in \Lambda^{k+1}
\end{equation*}
more explicitly,
and the subset
\begin{equation*}
\{\lambda(0), \lambda(1), \dots ,\lambda(k)\} \subset \Lambda
\end{equation*}
is denoted by $\ula$.
We note that $|\ula| \le d(\lambda)+1$
and that the equality holds if and only if all $\lambda(i)$'s are distinct.
We set
\begin{equation*}
\Lambda^{k+1,\circ}
=\{\lambda \in \Lambda^{k+1} ; |\ula|=k+1\} \subset \Lambda^{k+1}
\end{equation*}
for $k\ge 0$
and $\dprod\Lambda=\coprod_{k \ge 0}\Lambda^{k+1,\circ} \subset \prod\Lambda$.

For an element $\lambda \in \Lambda^{k+1}$,
we set
\begin{equation}
\label{definition of lambdai}
\lambda_i
=(\lambda(0),\lambda(1),\dots,\lambda(i-1),\lambda(i+1), \dots, \lambda(k))
\in \Lambda^k
\end{equation}
for $i=0,1, \dots, k$.
If $\lambda \in \Lambda^{k+1,\circ}$,
then $\lambda_i \in \Lambda^{k,\circ}$ for all $i$.

We define a map
\begin{equation}
\label{morphism h:eq}
h_i:\Lambda^{k+1} \longrightarrow \Lambda^{i+1}
\end{equation}
for $0 \le i \le k$ by
\begin{equation*}
h_i(\lambda)(j)=\lambda(j) \quad 0 \le j \le i
\end{equation*}
for $\lambda \in \Lambda^{k+1}$.
Similarly, a map
\begin{equation}
\label{morphism t}
t_i:\Lambda^{k+1} \longrightarrow \Lambda^{k-i+1}
\end{equation}
for $0 \le i \le k$ by
\begin{equation*}
t_i(\lambda)(j)=\lambda(j+i) \quad 0 \le j \le k-i
\end{equation*}
for $\lambda \in \Lambda^{k+1}$.

We trivially have
\begin{align*}
&h_i(\Lambda^{k+1,\circ}) \subset \Lambda^{i+1,\circ} \\
&t_i(\Lambda^{k+1,\circ}) \subset \Lambda^{k-i+1,\circ}
\end{align*}
for all $i,k$.
\end{para}

\begin{para}
For a finite subset $\ula$ of $\Lambda$,
the free $\bZ$-module of rank $|\ula|$
generated by $\{e_{\lambda}\}_{\lambda \in \ula}$
is denoted by $\bZ^{\ula}$,
that is, we have
\begin{equation*}
\bZ^{\ula}=\bigoplus_{\lambda \in \ula}\bZ e_{\lambda}
\end{equation*}
by definition.
By setting
\begin{equation*}
\varepsilon(\ula)
=\bigwedge^{|\ula|}\bZ^{\ula},
\end{equation*}
we obtain a free $\bZ$-module $\varepsilon(\ula)$ of rank $1$.
We note $\varepsilon(\emptyset)=\bZ$ by definition.
There exists the canonical isomorphism
\begin{equation}
\label{theta for Lambda:eq}
\vartheta(\ula):
\varepsilon(\ula) \otimes \varepsilon(\ula)
\longrightarrow \bZ
\end{equation}
which sends
$e_{\lambda_1} \wedge e_{\lambda_2} \wedge \dots \wedge e_{\lambda_k}
\otimes
e_{\lambda_1} \wedge e_{\lambda_2} \wedge \dots \wedge e_{\lambda_k}$
to $1$.
For two finite subsets
$\ula, \umu$ of $\Lambda$ with $\ula \cap \umu=\emptyset$,
we define a morphism
\begin{equation}
\label{morphism chi:eq}
\chi(\ula,\umu):
\varepsilon(\ula) \otimes \varepsilon(\umu)
\longrightarrow
\varepsilon(\ula \cup \umu)
\end{equation}
by $\chi(\ula,\umu)(v \otimes w)=v \wedge w$.

For $\lambda \in \Lambda^{k+1,\circ}$,
we set
\begin{equation*}
e_{\lambda}
=e_{\lambda(0)} \wedge e_{\lambda(1)} \wedge \dots \wedge e_{\lambda(k)}
\in \varepsilon(\ula) ,
\end{equation*}
which is a base of $\varepsilon(\ula)$ over $\bZ$.
For a subset $\umu$ of $\Lambda$
with $\ula \cap \umu=\emptyset$,
we define an isomorphism
\begin{equation*}
e_{\lambda}\wedge:\varepsilon(\umu) \longrightarrow \varepsilon(\ula \cup \umu)
\end{equation*}
by sending $v \in \varepsilon(\umu)$
to $e_{\lambda} \wedge v \in \varepsilon(\ula \cup \umu)$.
In particular,
we obtain an isomorphism
\begin{equation}
\label{isomorphism e wedge:eq}
e_{\lambda} \wedge: \bZ \longrightarrow \varepsilon(\ula)
\end{equation}
for the case of $\umu=\emptyset$.
\end{para}

\begin{para}
\label{definition of S(Lambda):eq}
The set of all subsets of $\Lambda$
is denoted by $S(\Lambda)$.
Moreover $S_n(\Lambda)$ denotes
the set of all subsets $\ula \subset \Lambda$
with $|\ula|=n$
for $n \ge 0$.
For two subsets $\ula, \umu$ with $\ula \subset \umu$,
the inclusion $\ula \hookrightarrow \umu$
is denoted by $\iota_{\umu,\ula}$.

The set $S(\Lambda)$ admits an order
by the inclusion of subsets.
We denote by $\cS(\Lambda)$
the category
associated to the ordered set $S(\Lambda)$
as in \cite[p.14]{Kashiwara-SchapiraCS}.
The subset of $S(\Lambda)$
consisting of all the non-empty subsets of $\Lambda$
is denoted by $S^{+}(\Lambda)$
and the category associated to the ordered set $S^{+}(\Lambda)$
by $\cS^{+}(\Lambda)$.
\end{para}

\begin{para}
Let $\cC$ be a category.
A cubical object in $\cC$ indexed by the category
$\cS(\Lambda)$ (resp. $\cS^{+}(\Lambda)$)
is a contravariant functor
from the category $\cS(\Lambda)$
(resp. $\cS^{+}(\Lambda)$) to $\cC$.
On the other hand,
a co-cubical object in $\cC$
indexed by the category $\cS(\Lambda)$
(resp. $\cS^{+}(\Lambda)$)
is a covariant functor
from the category $\cS(\Lambda)$
(resp. $\cS^{+}(\Lambda)$) to $\cC$.
A morphism of (co-)cubical objects
is a morphism of functors as usual.
We use terminology such as
(co-)cubical module, (co-)cubical complex, and so on,
as in the obvious meaning.
\end{para}

\begin{para}
\label{complexes from co-cubical objects:eq}
Now we fix a commutative $\bQ$-algebra $\kappa$.
Let $\cA$ be the category of $\kappa$-modules,
or the category of the $\kappa$-sheaves on a topological space.

Let $\Lambda$ be a non-empty set.
For a co-cubical complex $K$ in $\cA$
indexed by the category $\cS^{+}(\Lambda)$,
we set
\begin{equation*}
\C(K)^{k,l}=
\prod_{\lambda \in \Lambda^{k+1,\circ}}K(\ula)^l
\end{equation*}
for integers $k,l$.
An element $f \in \C(K)^{k,l}$
is a collection
\begin{equation*}
f=(f_{\lambda})_{\lambda \in \Lambda^{k+1,\circ}}
\qquad f_{\lambda} \in K(\ula)^l
\end{equation*}
by definition.
The morphism
$\delta=\delta_K: \C(K)^{k,l} \longrightarrow \C(K)^{k+1,l}$
is defined by
\begin{equation}
\label{Cech type morphism delta:eq}
\delta(f)_{\lambda}=
\sum_{i=0}^{k+1}(-1)^iK(\iota_{\ula,\uln{\lambda_i}})(f_{\lambda_i})
\end{equation}
for $f \in \C(K)^{k,l}$
and for $\lambda \in \Lambda^{k+2,\circ}$ as usual.
On the other hand,
we define a morphism
$\del:\C(K)^{k,l} \longrightarrow \C(K)^{k,l+1}$ by
\begin{equation*}
\del(f)_{\lambda}=df_{\lambda}
\end{equation*}
for $\lambda \in \Lambda^{k+1,\circ}$.
By setting
\begin{align*}
&\C(K)^p=\bigoplus_{k+l=p}\C(K)^{k,l} \\
&d=\delta+(-1)^k\del
\end{align*}
we obtain a complex $(\C(K), d)$ in $\cA$,
which is simply denoted by $\C(K)$ for short.
Here we follow the sign convention in
\cite[p.24]{NakkajimaWFSF}.
We call it the \v{C}ech complex of a co-cubical complex $K$.
The construction above is functorial in the usual sense.
\end{para}

\begin{para}
\label{filtrations on C(K):eq}
Let $(K,W)$ be an increasingly filtered co-cubical complex in $\cA$
indexed by $\cS^{+}(\Lambda)$.
Then increasing filtrations $W$ and $\delta W$
on the complex $\C(K)$ are defined by
\begin{align*}
&W_m\C(K)^{k,l}
=\prod_{\lambda \in \Lambda^{k+1,\circ}}
W_mK(\ula)^l \\
&W_m\C(K)^p
=\bigoplus_{k+l=p}W_m\C(K)^{k,l} \\
&(\delta W)_m\C(K)^{k,l}
=\prod_{\lambda \in \Lambda^{k+1,\circ}}
W_{m+k}K(\ula)^l \\
&(\delta W)_m\C(K)^p
=\bigoplus_{k+l=p}
(\delta W)_m\C(K)^{k,l}
\end{align*}
for every $m$.
We easily see the equality
\begin{equation}
\label{gr for delta W:eq}
\gr_m^{\delta W}\C(K)
=\bigoplus_{k \ge 0}\prod_{\lambda \in \Lambda^{k+1,\circ}}
\gr_{m+k}^WK(\ula)[-k]
\end{equation}
for every $m$.

For a decreasingly filtered co-cubical complex $(K,F)$ in $\cA$
indexed by $\cS^{+}(\Lambda)$,
we similarly define decreasing filtrations $F$ and $\delta F$ on $\C(K)$.
These constructions satisfy the functoriality
in the obvious meaning.
\end{para}

\begin{lem}
\label{gysin for co-cubical complex}
We have
\begin{equation}
\gamma_m(\C(K),\delta W)
=\bigoplus_{k \ge 0}
\prod_{\lambda \in \Lambda^{k+1,\circ}}
(-1)^k\gamma_{m+k}(K(\ula),W)[-k]
+\gr_{m+k}^W\delta
\end{equation}
for every $m$.
\end{lem}
\begin{proof}
Applying Proposition \ref{gamma for twisted filtered complex}
and the equality \eqref{gamma(K[l],W) and gamma(K,W)[l]:eq},
we obtain the conclusion.
\end{proof}

\begin{para}
\label{diagonal cubical complex}
Let $K,L$ be two co-cubical complexes in $\cA$
indexed by $\cS^{+}(\Lambda)$.
A co-cubical complex $K \otimes L$
is defined by
\begin{equation*}
(K \otimes L)(\ula)=K(\ula) \otimes L(\ula)
\end{equation*}
for $\ula \in S^{+}(\Lambda)$.
For $\ula, \umu \in S^{+}(\Lambda)$ with $\ula \subset \umu$,
the morphism
\begin{equation*}
(K \otimes L)(\iota_{\umu,\ula})
:(K \otimes L)(\ula) \longrightarrow (K \otimes L)(\umu)
\end{equation*}
is defined by
$(K \otimes L)(\iota_{\umu,\ula})
=K(\iota_{\umu,\ula}) \otimes L(\iota_{\umu,\ula})$.
For the case where $K$ and $L$ carry increasing filtrations $W$,
\begin{equation*}
W_m(K \otimes L)(\ula)
=\sum_{a+b=m}W_aK(\ula) \otimes W_bL(\ula)
\end{equation*}
defines a filtration $W$ on $K \otimes L$.

Now, we will define a morphism
\begin{equation}
\label{morphism tau:eq}
\tau: \C(K) \otimes \C(L) \longrightarrow \C(K \otimes L),
\end{equation}
which is a straightforward generalization
of the cup product on the singular cohomology groups
of a topological space
in terms of the \v{C}ech cohomology.
\end{para}

\begin{defn}
\label{definitin of tau}
We define a morphism
\begin{equation*}
\tau_{k,l}:
\C(K)^{k,p-k} \otimes \C(L)^{l,q-l}
\longrightarrow
\C(K \otimes L)^{k+l,p+q-k-l}
\end{equation*}
by setting
\begin{equation*}
\tau_{k,l}(f \otimes g)_{\lambda}
=K(\iota_{\ula,\uln{h_k(\lambda)}})(f_{h_k(\lambda)})
\otimes L(\iota_{\ula,\uln{t_k(\lambda)}})(g_{t_k(\lambda)})
\in K(\ula)^{p-k} \otimes L(\ula)^{q-l}
\end{equation*}
for $f \in \C(K)^{k,p-k}, g \in \C(L)^{l,q-l}$
and for $\lambda \in \Lambda^{k+l+1,\circ}$,
where $h_k$ and $t_k$ are the maps defined
in \eqref{morphism h:eq}
and in \eqref{morphism t} respectively.

By setting
\begin{equation*}
\tau=\tau_{K,L}=\sum_{k,l \ge 0}(-1)^{(p-k)l}\tau_{k,l}
\end{equation*}
we obtain a morphism
\begin{equation*}
\tau: \C(K)^p \otimes \C(L)^q \longrightarrow \C(K \otimes L)^{p+q}
\end{equation*}
for all $p,q$.
\end{defn}

The following lemmas can be checked
by easy and direct computation.

\begin{lem}
\label{the morphism tau}
The morphism $\tau$ defines a morphism of complexes
$\tau: \C(K) \otimes \C(L) \longrightarrow \C(K \otimes L)$.
\end{lem}

\begin{lem}
\label{associativity of the product}
For three co-cubical complexes $K_1,K_2,K_3$ in $\cA$
indexed by $\cS^{+}(\Lambda)$,
the diagram
\begin{equation*}
\begin{CD}
\C(K_1) \otimes \C(K_2) \otimes \C(K_3)
@>{\tau_{K_1,K_2} \otimes \id}>>
\C(K_1 \otimes K_2) \otimes \C(K_3) \\
@V{\id \otimes \tau_{K_2,K_3}}VV @VV{\tau_{K_1 \otimes K_2,K_3}}V \\
\C(K_1) \otimes \C(K_2 \otimes K_3)
@>>{\tau_{K_1,K_2 \otimes K_3}}>
\C(K_1 \otimes K_2 \otimes K_3)
\end{CD}
\end{equation*}
is commutative.
\end{lem}

\begin{lem}
\label{filtration under tau}
Let $(K,W), (L,W)$ be co-cubical filtered complexes.
Then the morphism $\tau$ above
satisfies
\begin{align*}
&\tau(W_a\C(K) \otimes W_b\C(L)) \subset W_{a+b}\C(K \otimes L) \\
&\tau((\delta W)_a\C(K) \otimes (\delta W)_b\C(L))
\subset (\delta W)_{a+b}\C(K \otimes L)
\end{align*}
for all $a,b$.
We have the same formulas for decreasing filtrations.
\end{lem}

\section{Residue morphisms}
\label{residue morphism}

In this section, we first fix the notation for log deformations.
Then we give the definition of the residue morphism
in our case,
and prove several results on it.
In the last part of this section,
we study the residue morphism for the Koszul complexes
of a log deformation.

\begin{para}
\label{setting for a log deformation}
Let $Y \longrightarrow \ast$ be a log deformation
(for the definition, see \cite[Definition 2.15]{Fujisawa-Nakayama},
\cite[Definition (3,8)]{SteenbrinkLE}).
We assume that all the irreducible components of the log deformation $Y$
are smooth as in \cite{Fujisawa-Nakayama}.
The log structure on $Y$ is denoted by $M_Y$.
The morphism of monoid sheaves
$\bN_Y \longrightarrow M_Y$
is induced by the morphism of log \ca spaces
$Y \longrightarrow \ast$.
The image of $1 \in \bN =\Gamma(\ast, \bN)$
by the morphism above
is denoted by $t \in \Gamma(Y, M_Y)$.
\end{para}

\begin{para}
We describe the irreducible decomposition of $Y$ by
$Y=\bigcup_{\lambda \in \Lambda}Y_{\lambda}$.
We set
\begin{equation*}
Y_{\ula}=\bigcap_{\lambda \in \ula}Y_{\lambda}
\end{equation*}
for $\ula \in S^{+}(\Lambda)$.
We set
$Y_{\emptyset}=Y$.
For $\ula, \umu \in S(\Lambda)$ with $\ula \subset \umu$,
we have the canonical closed immersion
\begin{equation}
\label{closed immersion a_{ula,umu}:eq}
a_{\ula,\umu}: Y_{\umu} \longrightarrow Y_{\ula}
\end{equation}
which satisfies the natural functorial property
with respect to $\ula$ and $\umu$ trivially.
The morphism $a_{\emptyset,\ula}:Y_{\ula} \longrightarrow Y$
is denoted by $a_{\ula}$ for short.
We omit the symbol $(a_{\ula})_{\ast}$ and $(a_{\ula,\umu})_{\ast}$
for complexes of sheaves on $Y$ and $Y_{\ula}$ as usual.
Then we have
\begin{equation}
\label{M/O* for Y:eq}
M_Y/\cO_Y^{\ast}
=\bigoplus_{\lambda \in \Lambda}\bN_{Y_{\lambda}}
\end{equation}
by definition.

For $\ula \in S^{+}(\Lambda)$,
the induced log structure
$a_{\ula}^*M_Y$ is simply denoted by $M_{Y_{\ula}}$.
Unless otherwise mentioned,
$Y_{\ula}$ is considered as a log complex manifold
with the log structure $M_{Y_{\ula}}$.
The log de Rham complex of $Y_{\ula}$
is denoted by $\omega_{Y_{\ula}}$.
The closed immersion
$a_{\ula,\umu}$ in \eqref{closed immersion a_{ula,umu}:eq}
is a morphism of log \ca spaces
for $\ula, \umu \in S(\Lambda)$
with $\ula \subset \umu$.
Thus the data $\{Y_{\ula}\}_{\ula \in S^{+}(\Lambda)}$
form a cubical log complex manifold
indexed by the category $\cS^{+}(\Lambda)$,
denoted by $Y_{\bullet}$.

We have the morphism
\begin{equation*}
a_{\ula}^{-1}M_Y \longrightarrow M_{Y_{\ula}}
\end{equation*}
of monoid sheaves.
The image of $t \in \Gamma(Y, M_Y)$
by the morphism above
is denoted by the same letter $t$ in $\Gamma(Y_{\ula}, M_{Y_{\ula}})$.
We have
\begin{equation}
\label{M/O* for Y_{ula}:eq}
\begin{split}
M_{Y_{\ula}}/\cO_{Y_{\ula}}^{\ast}
 &=a_{\ula}^{-1} (M_Y/\cO_Y^{\ast}) \\
 &=a_{\ula}^{-1}
   (\bigoplus_{\mu \in \Lambda}\bN_{Y_{\mu}})
 =\bigoplus_{\mu \in \Lambda}
    \bN_{Y_{\ula} \cap Y_{\mu}}
 =\bigoplus_{\mu \in \Lambda}
    \bN_{Y_{\ula \cup \{\mu\}}} ,
\end{split}
\end{equation}
by \eqref{M/O* for Y:eq}.
We have the canonical projection
\begin{equation*}
\red_{Y_{\ula}}:
M_{Y_{\ula}}
\longrightarrow
M_{Y_{\ula}}/\cO_{Y_{\ula}}^{\ast}
= \bigoplus_{\mu \in \Lambda}
\bN_{Y_{\ula \cup \{\mu\}}}
\end{equation*}
for $\ula \in S(\Lambda)$.

For $\usi \in S(\Lambda)$,
the monoid subsheaf
\begin{equation*}
M_{Y_{\ula}}^{\usi}
=\red_{Y_{\ula}}^{-1}
(\bigoplus_{\mu \in \usi}
\bN_{Y_{\ula \cup \{\mu\}}})
\end{equation*}
of $M_{Y_{\ula}}$ equipped with the restriction of the structure morphism
$M_{Y_{\ula}} \longrightarrow \cO_{Y_{\ula}}$
defines a log structure on $Y_{\ula}$.
The \ca space $Y_{\ula}$
equipped with the log structure $M_{Y_{\ula}}^{\usi}$
is denoted by $Y_{\ula}^{\usi}$.
According to this definition,
$M_{Y_{\ula}}$ coincides with $M_{Y_{\ula}}^{\Lambda}$,
and $Y_{\ula}$ with $Y_{\ula}^{\Lambda}$.
For an element $\lambda \in \Lambda$,
we use the notation $M_{Y_{\ula}}^{\lambda}, Y_{\ula}^{\lambda}$
instead of $M_{Y_{\ula}}^{\{\lambda\}}, Y_{\ula}^{\{\lambda\}}$
for simplicity.
The log de Rham complex of $Y_{\ula}^{\usi}$
is denoted by $\omega_{Y_{\ula}^{\usi}}$,
which is a subcomplex of $\omega_{Y_{\ula}}$
in the trivial way.
For $\usi, \uta \in S(\Lambda)$ with $\usi \subset \uta$,
we have
$M_{Y_{\ula}}^{\usi} \subset M_{Y_{\ula}}^{\uta}$,
and the inclusion
$\omega_{Y_{\ula}^{\usi}} \subset \omega_{Y_{\ula}^{\uta}}$
as subcomplexes of $\omega_{Y_{\ula}}$.
\end{para}

\begin{para}
\label{increasing filtration W(usi)}
For $\usi \in S(\Lambda)$,
we define an increasing filtration $W(\usi)$
on $\omega_{Y_{\ula}}$ by
\begin{equation*}
W(\usi)_m\omega_{Y_{\ula}}^p
=\image(\omega_{Y_{\ula}}^m
         \otimes_{\cO_{Y_{\ula}}}
           \omega_{Y_{\ula}^{\Lambda \setminus \usi}}^{p-m}
         \longrightarrow \omega_{Y_{\ula}}^p)
\end{equation*}
for every non-negative integer $m$.
The filtration $W(\Lambda)$ is denoted by $W$ for short.
The morphism
\begin{equation*}
\bigwedge^m\dlog:
\bigwedge^mM_{Y_{\ula}}\gp \longrightarrow \omega_{Y_{\ula}}^m
\end{equation*}
induces a morphism
\begin{equation*}
\bigwedge^mM_{Y_{\ula}}\gp
\otimes_{\bZ} \omega_{Y_{\ula}^{\Lambda \setminus \usi}}^{p-m}
\longrightarrow W(\usi)_m\omega_{Y_{\ula}}^p
\end{equation*}
for every $p$.
By composing the morphism above and the projection
\begin{equation*}
W(\usi)_m\omega_{Y_{\ula}}^p
\longrightarrow
\gr_m^{W(\usi)}\omega_{Y_{\ula}}^p ,
\end{equation*}
the morphism
\begin{equation*}
\bigwedge^mM_{Y_{\ula}}\gp
\otimes_{\bZ} \omega_{Y_{\ula}^{\Lambda \setminus \usi}}^{p-m}
\longrightarrow \gr_m^{W(\usi)}\omega_{Y_{\ula}}^p
\end{equation*}
is obtained.
We can easily see that
the morphism above factors through the surjection
\begin{equation*}
\bigwedge^mM_{Y_{\ula}}\gp
\otimes_{\bZ} \omega_{Y_{\ula}^{\Lambda \setminus \usi}}^{p-m}
\longrightarrow
\bigwedge^m(M_{Y_{\ula}}\gp/(M_{Y_{\ula}}^{\Lambda \setminus \usi})\gp)
\otimes_{\bZ} \omega_{Y_{\ula \cup \usi}^{\Lambda \setminus \usi}}^{p-m}
\end{equation*}
by the definition of $W(\usi)$.
If $m=|\usi|$, we obtain a morphism
\begin{equation}
\label{isomorphism to gr of W(usi):eq}
\varepsilon(\usi)
\otimes_{\bZ} \omega_{Y_{\ula \cup \usi}^{\Lambda \setminus \usi}}^{p-m}
\longrightarrow
\gr_m^{W(\usi)}\omega_{Y_{\ula}}^p
\end{equation}
by using
$\bigwedge^m(M_{Y_{\ula}}\gp/(M_{Y_{\ula}}^{\Lambda \setminus \usi})\gp)
=\varepsilon(\usi)$.
\end{para}

\begin{para}
\label{local description of a log deformation:eq}
We first describe the local case for the later use.
So we assume the following:
\begin{mylist}
\itemno
\label{local setting for Y:eq}
$Y=\{x_1x_2 \cdots x_r=0\}$
in the polydisc $\Delta^n$
with the coordinate functions $x_1, x_2, \dots, x_n$.
\itemno
$\Lambda=\{1, 2, \dots, r\}$ and
$Y_\lambda=\{x_{\lambda}=0\}$ for $\lambda \in \Lambda$.
\itemno
\label{local setting for t:eq}
$t=x_1x_2 \cdots x_r=\prod_{i=1}^rx_i$.
\end{mylist}
Let $\usi=\{\si_1, \si_2, \dots, \si_m\}$
be an element of $S_m(\Lambda)$.
For a local section $\omega$
of $\omega_{Y_{\ula \cup \usi}^{\Lambda \setminus \usi}}^{p-m}$,
the morphism
\eqref{isomorphism to gr of W(usi):eq}
sends
\begin{equation*}
e_{\si_1} \wedge e_{\si_2} \wedge \dots \wedge e_{\si_m} \otimes \omega
\end{equation*}
to the local section
\begin{equation*}
\dlog x_{\si_1} \wedge \dlog x_{\si_2} \wedge \dots \wedge \dlog x_{\si_m}
 \wedge \tilde{\omega}
\end{equation*}
where $\tilde{\omega}$ is a local section of
$\omega_{Y_{\ula}^{\Lambda \setminus \usi}}^{p-m}$
whose restriction to $Y_{\ula \cup \usi}^{\Lambda \setminus \usi}$
coincides with $\omega$.
\end{para}

\begin{prop}
The morphism
\eqref{isomorphism to gr of W(usi):eq}
induces an isomorphism of complexes
\begin{equation}
\label{isomorphism of complexes to gr of W(usi):eq}
\varepsilon(\usi)
\otimes_{\bZ} \omega_{Y_{\ula \cup \usi}^{\Lambda \setminus \usi}}[-m]
\longrightarrow
\gr_m^{W(\usi)}\omega_{Y_{\ula}}
\end{equation}
for every $\usi \in S_m(\Lambda)$.
\end{prop}
\begin{proof}
Same as \cite[Proposition 3.6]{DeligneED}.
\end{proof}

\begin{defn}
\label{definition of residue on log de Rham complex}
For the case of $|\usi|=m$,
the morphism
\begin{equation*}
\res_{Y_{\ula}}^{\usi}:
\omega_{Y_{\ula}}
\longrightarrow
\varepsilon(\usi) \otimes_{\bZ} \omega_{Y_{\ula \cup \usi}}[-m]
\end{equation*}
is defined as the composite of the three morphisms,
the projection
\begin{equation*}
\omega_{Y_{\ula}}=W(\usi)_m\omega_{Y_{\ula}}
\longrightarrow \gr_m^{W(\usi)}\omega_{Y_{\ula}} ,
\end{equation*}
the inverse of the isomorphism
\eqref{isomorphism of complexes to gr of W(usi):eq},
and the morphism
\begin{equation*}
\varepsilon(\usi) \otimes_{\bZ}
\omega_{Y_{\ula \cup \usi}^{\Lambda \setminus \usi}}[-m]
\longrightarrow
\varepsilon(\usi) \otimes_{\bZ} \omega_{Y_{\ula \cup \usi}}[-m]
\end{equation*}
induced from the inclusion
$\omega_{Y_{\ula \cup \usi}^{\Lambda \setminus \usi}}
\subset \omega_{Y_{\ula \cup \usi}}$.
Note that $\res_{Y_{\ula}}^{\emptyset}=\id$.

Moreover we set
\begin{equation*}
\res_{Y_{\ula}}^m
=\sum_{\usi \in S_m(\Lambda)}
\res_{Y_{\ula}}^{\usi}:
\omega_{Y_{\ula}}
\longrightarrow
\bigoplus_{\usi \in S_m(\Lambda)}
\varepsilon(\usi) \otimes_{\bZ}
\omega_{Y_{\ula \cup \usi}}[-m]
\end{equation*}
for every non-negative integer $m$.
Here we remark that the definition of the residue morphisms above
is different from that by Deligne in \cite[(3.1.5.2)]{DeligneII}.
\end{defn}

\begin{lem}
\label{lemma on gr of omega}
In the situation above,
the morphism
$\res_{Y_{\ula}}^m$
induces an isomorphism
\begin{equation}
\label{gr of W for omega:eq}
\gr_m^W\omega_{Y_{\ula}}
\overset{\simeq}{\longrightarrow}
\bigoplus_{\usi \in S_m(\Lambda)}
\varepsilon(\usi) \otimes_{\bZ} \Omega_{Y_{\ula \cup \usi}}[-m]
\end{equation}
for every $m$.
\end{lem}
\begin{proof}
Same as \cite[Proposition 3.6]{DeligneED}.
\end{proof}

\begin{para}
A global section $\dlog t$ of $\omega_{Y_{\ula}}^1$
is obtained from $t \in \Gamma(Y_{\ula},M_{Y_{\ula}})$.
A morphism of complexes
\begin{equation}
\label{morphism dlog t wedge:eq}
\dlog t \wedge: \omega_{Y_{\ula}} \longrightarrow \omega_{Y_{\ula}}[1]
\end{equation}
is defined by sending a local section $\omega \in \omega_{Y_{\ula}}^p$
to $\dlog t \wedge \omega \in \omega_{Y_{\ula}}^{p+1}$.
We can easily see the property
$(\dlog t \wedge)(W_m\omega_{Y_{\ula}}) \subset W_{m+1}\omega_{Y_{\ula}}[1]$
for every $m$.
Therefore the morphism $\dlog t \wedge$ above
induces a morphism of complexes
\begin{equation}
\label{morphism dlog t on grWomega:eq}
\dlog t \wedge:
\gr_m^W\omega_{Y_{\ula}}
\longrightarrow
\gr_{m+1}^W\omega_{Y_{\ula}}[1]
\end{equation}
for every $m$.
\end{para}

\begin{lem}
\label{lemma on dlog t and residues}
We have
\begin{equation*}
\begin{split}
\res_{Y_{\ula}}^{\usi}[1] &(\dlog t \wedge) \\
&=(-1)^m(\id \otimes (\dlog t \wedge)[-m])\res_{Y_{\ula}}^{\usi}
+\sum_{\lambda \in \usi}((e_{\lambda} \wedge)
\otimes
a_{\ula \cup (\usi \setminus \{\lambda\}),\ula \cup \usi}^{\ast})
\res_{Y_{\ula}}^{\usi \setminus \{\lambda\}} \\
&\qquad
:\omega_{Y_{\ula}}
\longrightarrow
\varepsilon(\usi) \otimes_{\bZ} \omega_{Y_{\ula \cup \usi}}[1-m]
\end{split}
\end{equation*}
for $\usi \in S_m(\Lambda)$.
\end{lem}
\begin{proof}
We may assume the conditions
\eqref{local setting for Y:eq}--\eqref{local setting for t:eq}
in \ref{local description of a log deformation:eq}.
Now we take a subset
$\usi=\{\sigma_1, \sigma_2, \dots, \sigma_m\}$ of $\Lambda$.
For a local section $\omega$ of $\omega_{Y_{\ula}}$,
we write
\begin{equation*}
\begin{split}
\omega=
 \dlog x_{\sigma_1} \wedge \dlog x_{\sigma_2}
  &\wedge \dots \wedge \dlog x_{\sigma_m}
  \wedge \eta \\
 &+\sum_{i=1}^m\dlog x_{\sigma_1} \wedge \dots
  \wedge \Slash{\dlog x_{\sigma_i}} \wedge \dots
  \wedge \dlog x_{\sigma_m} \wedge \eta_i
  +\eta'
\end{split}
\end{equation*}
where $\Slash{\dlog x_{\sigma_i}}$ means to omit
$\dlog x_{\sigma_i}$
and where
$\eta, \eta_i \in W(\usi)_0\omega_{Y_{\ula}},
\eta' \in W(\usi)_{m-2}\omega_{Y_{\ula}}$.
Then we have
\begin{equation*}
\begin{split}
\dlog t \wedge \omega
 =(-1)^m&\sum_{\lambda \in \Lambda \setminus \usi}
  \dlog x_{\sigma_1} \wedge \dots \wedge \dlog x_{\sigma_m}
  \wedge \dlog x_{\lambda} \wedge \eta \\
  &+\sum_{i=1}^m\dlog x_{\sigma_i} \wedge \dlog x_{\sigma_1}
  \wedge \dots \wedge \Slash{\dlog x_{\sigma_i}} \wedge \dots
  \wedge \dlog x_{\sigma_m} \wedge \eta_i \\
 &+\sum_{\lambda \in \Lambda \setminus \usi}
   \dlog x_{\lambda} \wedge \dlog x_{\sigma_0} \wedge \dots
  \wedge \Slash{\dlog x_{\sigma_i}} \wedge\dots
  \wedge \dlog x_{\sigma_m} \wedge \eta_i \\
 &+\dlog t \wedge \eta' ,
\end{split}
\end{equation*}
and then
\begin{align*}
&\res_{Y_{\ula}}^{\usi}(\dlog t \wedge \omega)
 = (-1)^me_{\sigma_1} \wedge \dots \wedge e_{\sigma_m}
     \otimes \sum_{\lambda \in \Lambda \setminus \usi}
      \dlog x_{\lambda} \wedge \eta \\
&\qquad \qquad \qquad \qquad \qquad \qquad
   +\sum_{i=1}^me_{\sigma_i} \wedge e_{\sigma_1} \wedge \dots \wedge
                \Slash{e_{\sigma_i}}
              \wedge \dots \wedge e_{\sigma_m} \otimes \eta_i , \\
&\res_{Y_{\ula}}^{\usi}(\omega)
 =e_{\sigma_1} \wedge \dots \wedge e_{\sigma_m} \otimes \eta
\end{align*}
by definition.
On the other hand,
\begin{equation*}
\res_{Y_{\ula}}^{\usi \setminus \{\sigma_i\}}(\omega)
 =e_{\sigma_1} \wedge \dots \wedge
   \Slash{e_{\sigma_i}} \wedge \dots \wedge e_{\sigma_m} \otimes
  ((-1)^{m-i} \dlog x_{\sigma_i} \wedge \eta
   +\eta_i)
\end{equation*}
holds for $i=1,2, \dots, m$.
Therefore
\begin{align*}
\res_{Y_{\ula}}^{\usi}&(\dlog t \wedge \omega) \\
&=(-1)^me_{\sigma_1} \wedge \dots \wedge e_{\sigma_m}
\otimes \sum_{\lambda \in \Lambda \setminus \usi}
\dlog x_{\lambda} \wedge \eta \\
&\qquad
+\sum_{i=1}^m(e_{\sigma_i} \wedge \otimes \id)
\res_{Y_{\ula}}^{\usi \setminus \{\sigma_i\}}(\omega)
+(-1)^me_{\sigma_1} \wedge \dots \wedge e_{\sigma_m} \otimes
\sum_{i=1}^{m} \dlog x_{\sigma_i} \wedge \eta \\
&=(-1)^m(\id \otimes (\dlog t \wedge)[-m]) \res_{Y_{\ula}}^{\usi}(\omega) \\
&\qquad \qquad \qquad \qquad \qquad
+\sum_{\nu \in \usi}((e_{\nu} \wedge) \otimes
a_{\ula \cup (\usi \setminus \{\nu\}),\ula \cup \usi}^{\ast}
\cdot \res_{Y_{\ula}}^{\usi \setminus \{\nu\}}(\omega)
\end{align*}
is obtained.
\end{proof}

\begin{cor}
\label{dlog t on grWomega}
The morphism \eqref{morphism dlog t on grWomega:eq}
is identified with the morphism
\begin{equation*}
\begin{split}
\bigoplus_{\usi \in S_m(\Lambda)}
\sum_{\mu \in \Lambda \setminus \usi}
(e_{\mu} \wedge) &\otimes
(a_{\ula \cup \usi, \ula \cup \usi \cup \{\mu\}})^{\ast} \\
&:
\bigoplus_{\usi \in S_m(\Lambda)}
\varepsilon(\usi) \otimes_{\bZ} \Omega_{Y_{\ula \cup \usi}}[-m]
\longrightarrow
\bigoplus_{\uta \in S_{m+1}(\Lambda)}
\varepsilon(\uta) \otimes_{\bZ} \Omega_{Y_{\ula \cup \uta}}[-m]
\end{split}
\end{equation*}
under the identification \eqref{gr of W for omega:eq}.
\end{cor}

\begin{para}
We recall the definition of the Koszul complexes.
Moreover, we will define the residue morphism of the Koszul complexes
of a log deformation,
and prove several results on it.
The main reference is \cite[Section 1]{FujisawaMHSLSD}.

For $\usi \in S(\Lambda)$,
a $\bQ$-sheaf $\kos_{Y_{\ula}}(M_{Y_{\ula}}^{\usi};n)^p$
is defined as in \cite[(2.3)]{FujisawaMHSLSD} by
\begin{equation}
\label{definition of kos:eq}
\kos_{Y_{\ula}}(M_{Y_{\ula}}^{\usi};n)^p
=\Gamma_{n-p}(\cO_{Y_{\ula}})
\otimes_{\bZ} \bigwedge^p(M_{Y_{\ula}}^{\usi})\gp
\end{equation}
for non-negative integers $n,p$ with $p \le n$,
where the divided power envelope $\Gamma_{n-p}(\cO_{Y_{\ula}})$
is taken over the base field $\bQ$.
By setting
\begin{equation*}
d(f_1^{[n_1]}f_2^{[n_2]} \cdots f_k^{[n_k]} \otimes m)
=\sum_{i=1}^{k}f_1^{[n_1]} \cdots f_i^{[n_i-1]} \cdots f_k^{[n_k]}
\otimes \exp(2\pi\sqrt{-1}f_i) \wedge m
\end{equation*}
for integers $n_1, n_2, \dots, n_k$
with $n_1+n_2+\dots+n_k=n-p$,
the differential
\begin{equation*}
d: \kos_{Y_{\ula}}(M_{Y_{\ula}}^{\usi};n)^p
\longrightarrow \kos_{Y_{\ula}}(M_{Y_{\ula}}^{\usi};n)^{p+1}
\end{equation*}
is defined.
The global section $1$ of $\cO_{Y_{\ula}}$
gives us the morphism
\begin{equation}
\label{morphism from kos(q) to kos(q+1):eq}
\kos_{Y_{\ula}}(M_{Y_{\ula}}^{\usi};n)
\longrightarrow
\kos_{Y_{\ula}}(M_{Y_{\ula}}^{\usi};n+1)
\end{equation}
by sending $f \otimes m$ to $1^{[1]}f \otimes m$
for $f \in \Gamma_{n-p}(\cO_{Y_{\ula}})$
and for $m \in \bigwedge^p(M_{Y_{\ula}}^{\usi})\gp$.
Thus we obtain an inductive system
\begin{equation*}
\begin{CD}
\cdots @>>> \kos_{Y_{\ula}}(M_{Y_{\ula}}^{\usi};n)
       @>>> \kos_{Y_{\ula}}(M_{Y_{\ula}}^{\usi};n+1)
       @>>> \cdots ,
\end{CD}
\end{equation*}
and its limit
\begin{equation*}
\kos_{Y_{\ula}}(M_{Y_{\ula}}^{\usi})
=\limind_{n}\kos_{Y_{\ula}}(M_{Y_{\ula}}^{\usi};n)
\end{equation*}
as in \cite[Definition 1.8]{FujisawaMHSLSD}.
By setting
\begin{equation*}
\psi_{(Y_{\ula},M_{Y_{\ula}})}(f_1^{[n_1]} \cdots f_k^{[n_k]} \otimes m)
=(2\pi\sqrt{-1})^{-p}(n_1! \cdots n_k!)^{-1}
f_1^{n_1} \cdots f_k^{n_k}(\bigwedge^p\dlog)(m) ,
\end{equation*}
a morphism of complexes
\begin{equation*}
\psi_{(Y_{\ula},M_{Y_{\ula}}^{\usi})}:
\kos_{Y_{\ula}}(M_{Y_{\ula}}^{\usi};n)
\longrightarrow \omega_{Y_{\ula}^{\usi}}
\end{equation*}
is defined for every $n$.
These data $\psi_{(Y_{\ula},M_{Y_{\ula}}^{\usi})}$ for all $n$
induce a morphism of complexes
\begin{equation}
\label{definition of psi:eq}
\psi_{(Y_{\ula},M_{Y_{\ula}}^{\usi})}:
\kos_{Y_{\ula}}(M_{Y_{\ula}}^{\usi})
\longrightarrow \omega_{Y_{\ula}^{\usi}}
\end{equation}
as in \cite[(2.4)]{FujisawaMHSLSD}.

For the case of $\usi=\emptyset$,
the log structure $M_{Y_{\ula}}^{\emptyset}$
is nothing but the trivial log structure $\cO_{Y_{\ula}}^{\ast}$.
Then the morphism
$\bQ_{Y_{\ula}}
\longrightarrow
\Gamma_n(\cO_{Y_{\ula}})
=\kos_{Y_{\ula}}(\cO_{Y_{\ula}}^{\ast};n)^0$
which sends $f \in \bQ$ to $(n!f)1^{[n]} \in \Gamma_n(\cO_{Y_{\ula}})$
induces a morphism of complexes
$\bQ_{Y_{\ula}} \longrightarrow \kos_{Y_{\ula}}(\cO_{Y_{\ula}}^{\ast})$.
Moreover these morphisms for all $n$
are compatible with the morphisms
\eqref{morphism from kos(q) to kos(q+1):eq}.
Therefore a morphism of complexes
\begin{equation}
\label{morphism from Q to kos(O*):eq}
\bQ_{Y_{\ula}} \longrightarrow \kos_{Y_{\ula}}(\cO_{Y_{\ula}}^{\ast})
\end{equation}
is obtained.
\end{para}

\begin{lem}
\label{commutativity on Q and C for trivial log}
The morphism \eqref{morphism from Q to kos(O*):eq}
is a quasi-isomorphism,
which fits in the commutative diagram
\begin{equation*}
\begin{CD}
\bQ_{Y_{\ula}} @>>> \kos_{Y_{\ula}}(\cO_{Y_{\ula}}^{\ast}) \\
@VVV @VV{\psi_{(Y_{\ula}, \cO_{Y_{\ula}}^{\ast})}}V \\
\bC_{Y_{\ula}} @>>> \Omega_{Y_{\ula}},
\end{CD}
\end{equation*}
where the left vertical arrow
is the canonical inclusion $\bQ_{Y_{\ula}} \longrightarrow \bC_{Y_{\ula}}$
and the bottom horizontal arrow is the usual morphism
induced by the canonical inclusion
$\bC_{Y_{\ula}} \longrightarrow \cO_{Y_{\ula}}$.
\end{lem}
\begin{proof}
Easy by definition.
\end{proof}

\begin{para}
For the case of $\usi=\Lambda$,
the global section
$t$ of $M_{Y_{\ula}}$
defines a morphism of complexes
\begin{equation*}
\kos_{Y_{\ula}}(M_{Y_{\ula}};n)
\longrightarrow \kos_{Y_{\ula}}(M_{Y_{\ula}};n+1)[1]
\end{equation*}
by sending $f \otimes m$ to $f \otimes t \wedge m$.
This induces a morphism of complexes
\begin{equation}
\label{the morphism t wedge on Kos:eq}
t \wedge:
\kos_{Y_{\ula}}(M_{Y_{\ula}})
\longrightarrow \kos_{Y_{\ula}}(M_{Y_{\ula}})[1]
\end{equation}
as in \cite[(1.11)]{FujisawaMHSLSD}.
Direct computation shows that the diagram
\begin{equation}
\label{commutative diagram for t wedge and dlog t wedge:eq}
\begin{CD}
\kos_{Y_{\ula}}(M_{Y_{\ula}})
@>{t \wedge}>> \kos_{Y_{\ula}}(M_{Y_{\ula}})[1] \\
@V{\psi_{(Y_{\ula},M_{Y_{\ula}})}}VV
@VV{(2\pi\sqrt{-1})\psi_{(Y_{\ula},M_{Y_{\ula}})}}V \\
\omega_{Y_{\ula}}
@>>{\dlog t \wedge}> \omega_{Y_{\ula}}[1] \\
\end{CD}
\end{equation}
is commutative.
\end{para}

\begin{para}
For $\ula, \umu \in S(\Lambda)$ with $\ula \subset \umu$,
the inclusion
$a_{\ula,\umu}: Y_{\umu}^{\usi} \longrightarrow Y_{\ula}^{\usi}$
induces morphisms of complexes
\begin{equation}
\label{canonical morphism for the Koszul complexes:eq}
a_{\ula,\umu}^{\ast}:
a_{\ula,\umu}^{-1}\kos_{Y_{\ula}}(M_{Y_{\ula}}^{\usi})
\longrightarrow
\kos_{Y_{\umu}}(M_{Y_{\umu}}^{\usi})
\end{equation}
and
\begin{equation*}
\kos_{Y_{\ula}}(M_{Y_{\ula}}^{\usi})
\longrightarrow
\kos_{Y_{\umu}}(M_{Y_{\umu}}^{\usi})
=(a_{\ula,\umu})_{\ast}\kos_{Y_{\umu}}(M_{Y_{\umu}}^{\usi})
\end{equation*}
in the trivial way.
These two morphisms are denoted by the same letter
$a_{\ula,\umu}^{\ast}$,
by abuse of the notation.
\end{para}

\begin{para}
For $\usi \in S(\Lambda)$,
the subsheaf
$(M_{Y_{\ula}}^{\Lambda \setminus \usi})\gp$
of $M_{Y_{\ula}}\gp$ yields the filtration
$W((M_{Y_{\ula}}^{\Lambda \setminus \usi})\gp)$
on $\kos_{Y_{\ula}}(M_{Y_{\ula}})$
as defined in \cite[Definition 1.8]{FujisawaMHSLSD}.
This filtration on $\kos_{Y_{\ula}}(M_{Y_{\ula}})$
is denoted by $W(\usi)$ in this article.
The filtration $W(\Lambda)$
is denoted by $W$.
The morphism $\psi_{(Y_{\ula},M_{Y_{\ula}})}$ above
preserves the filtration $W(\usi)$
for any subset $\usi$ of $\Lambda$.
As proved in \cite[Proposition 1.10]{FujisawaMHSLSD},
we have an isomorphism of complexes
\begin{equation}
\label{isom on gr_m^W(usi)kos:eq}
\gr_m^{W(\usi)}\kos_{Y_{\ula}}(M_{Y_{\ula}})
\simeq
\varepsilon(\usi)
\otimes_{\bZ}
a_{\ula,\ula \cup \usi}^{-1}
\kos_{Y_{\ula}}(M_{Y_{\ula}}^{\Lambda \setminus \usi})[-m]
\end{equation}
for every integer $m$.

We have the inclusion
\begin{equation*}
a_{\ula,\ula \cup \usi}^{-1}
\kos_{Y_{\ula}}(M_{Y_{\ula}}^{\Lambda \setminus \usi})
\longrightarrow
a_{\ula,\ula \cup \usi}^{-1}
\kos_{Y_{\ula}}(M_{Y_{\ula}})
\end{equation*}
induced by the inclusion
$M_{Y_{\ula}}^{\Lambda \setminus \usi} \subset M_{Y_{\ula}}$.
Therefore we obtain the morphism
\begin{equation}
\label{id otimes inclusion for kos:eq}
\varepsilon(\usi) \otimes_{\bZ}
a_{\ula,\ula \cup \usi}^{-1}
\kos_{Y_{\ula}}(M_{Y_{\ula}}^{\Lambda \setminus \usi})[-m]
\longrightarrow
\varepsilon(\usi) \otimes_{\bZ}
a_{\ula,\ula \cup \usi}^{-1}
\kos_{Y_{\ula}}(M_{Y_{\ula}})[-m]
\end{equation}
by tensoring the identity and by shifting.

On the other hand,
we have the canonical morphism
\eqref{canonical morphism for the Koszul complexes:eq}
\begin{equation*}
a_{\ula,\ula \cup \usi}^{\ast}:
a_{\ula,\ula \cup \usi}^{-1}\kos_{Y_{\ula}}(M_{Y_{\ula}})
\longrightarrow
\kos_{Y_{\ula \cup \usi}}(M_{Y_{\ula \cup \usi}})
\end{equation*}
for $\ula, \usi \in S(\Lambda)$.
Thus the morphism
\begin{equation}
\label{id otimes the functoriality morphism:eq}
\begin{split}
\id \otimes a_{\ula,\ula \cup \usi}^{\ast}[-m]:
\varepsilon(\usi) \otimes_{\bZ}
a_{\ula,\ula \cup \usi}^{-1}&\kos_{Y_{\ula}}(M_{Y_{\ula}})[-m] \\
&\longrightarrow
\varepsilon(\usi) \otimes_{\bZ}
\kos_{Y_{\ula \cup \usi}}(M_{Y_{\ula \cup \usi}})[-m]
\end{split}
\end{equation}
is obtained.
\end{para}

\begin{defn}
For $\usi \in S_m(\Lambda)$,
the equality
\begin{equation*}
W(\usi)_m\kos_{Y_{\ula}}(M_{Y_{\ula}})
=\kos_{Y_{\ula}}(M_{Y_{\ula}})
\end{equation*}
can be easily seen.
Then the composite of four morphisms,
the projection
\begin{equation*}
\kos_{Y_{\ula}}(M_{Y_{\ula}})
=W(\usi)_m\kos_{Y_{\ula}}(M_{Y_{\ula}})
\longrightarrow \gr_m^{W(\usi)}\kos_{Y_{\ula}}(M_{Y_{\ula}}) ,
\end{equation*}
the isomorphism \eqref{isom on gr_m^W(usi)kos:eq},
the morphisms
\eqref{id otimes inclusion for kos:eq} and
\eqref{id otimes the functoriality morphism:eq},
is denoted by
\begin{equation*}
\res_{Y_{\ula}}^{\usi}:
\kos_{Y_{\ula}}(M_{Y_{\ula}})
\longrightarrow
\varepsilon(\usi) \otimes_{\bZ}
\kos_{Y_{\ula \cup \usi}}(M_{Y_{\ula \cup \usi}})[-m]
\end{equation*}
by abuse of the language.
Moreover we set
\begin{equation*}
\res_{Y_{\ula}}^m
=\sum_{\usi \in S_m(\Lambda)}
\res_{Y_{\ula}}^{\usi}:
\kos_{Y_{\ula}}(M_{Y_{\ula}}) \longrightarrow
\bigoplus_{\usi \in S_m(\Lambda)}
\varepsilon(\usi)
\otimes_{\bZ} \kos_{Y_{\ula \cup \usi}}(M_{Y_{\ula \cup \usi}})[-m]
\end{equation*}
as in Definition \ref{definition of residue on log de Rham complex} again.
\end{defn}

\begin{lem}
\label{lemma on gr of kos}
The morphism $\res_{Y_{\ula}}^m$ induces a quasi-isomorphism
\begin{equation}
\label{residue iso for kos:eq}
\res_{Y_{\ula}}^m :
\gr_m^W\kos_{Y_{\ula}}(M_{Y_{\ula}})
\longrightarrow
\bigoplus_{\usi \in S_m(\Lambda)}
\varepsilon(\usi) \otimes_{\bZ}
\kos_{Y_{\ula \cup \usi}}(\cO_{Y_{\ula \cup \usi}}^{\ast})[-m]
\end{equation}
for every $m$.
In particular,
we have an isomorphism
\begin{equation}
\label{residue iso in the derived category:eq}
\gr_m^W\kos_{Y_{\ula}}(M_{Y_{\ula}})
\longrightarrow
\bigoplus_{\usi \in S_m(\Lambda)}
\varepsilon(\usi) \otimes_{\bZ}
\bQ_{Y_{\ula \cup \usi}}[-m]
\end{equation}
in the derived category for every $m$.
\end{lem}
\begin{proof}
We have an isomorphism
\begin{equation*}
\gr_m^W\kos_{Y_{\ula}}(M_{Y_{\ula}})
\longrightarrow
\bigoplus_{\usi \in S_m(\Lambda)}
\varepsilon(\usi) \otimes_{\bZ}
a_{\ula,\ula \cup \usi}^{-1}
\kos_{Y_{\ula}}(\cO_{Y_{\ula}}^{\ast})[-m]
\end{equation*}
as in the case of $\omega_{Y_{\ula}}$.
We can check that the canonical morphism
\begin{equation*}
a_{\ula,\ula \cup \usi}^{-1}
\kos_{Y_{\ula}}(\cO_{Y_{\ula}}^{\ast})
\longrightarrow
\kos_{Y_{\ula \cup \usi}}(\cO_{Y_{\ula \cup \usi}}^{\ast})
\end{equation*}
is a quasi-isomorphism.
Thus the morphism
\eqref{residue iso for kos:eq}
is a quasi-isomorphism.
Combining with the quasi-isomorphism
\eqref{morphism from Q to kos(O*):eq} for $\ula \cup \usi$,
we obtain the isomorphism
\eqref{residue iso in the derived category:eq}.
See \cite[Section 1]{FujisawaMHSLSD} for the detail.
\end{proof}

\begin{lem}
\label{commutativity of residue for omega and koszul}
For a non-negative integer $m$,
the diagram
\begin{equation*}
\begin{CD}
\kos_{Y_{\ula}}(M_{Y_{\ula}})
@>{\res_{Y_{\ula}}^{\usi}}>>
\varepsilon(\usi)
\otimes_{\bZ} \kos_{Y_{\ula \cup \usi}}(M_{Y_{\ula \cup \usi}})[-m] \\
@V{\psi_{(Y_{\ula}, M_{Y_{\ula}})}}VV
@VV{\id \otimes
(2\pi\sqrt{-1})^{-m}\psi_{(Y_{\ula \cup \usi},M_{Y_{\ula \cup \usi}})[-m]}}V \\
\omega_{Y_{\ula}} @>>{\res_{Y_{\ula}}^{\usi}}>
\varepsilon(\usi) \otimes_{\bZ} \omega_{Y_{\ula \cup \usi}}[-m]
\end{CD}
\end{equation*}
is commutative for every $\usi \in S_m(\Lambda)$.
\end{lem}
\begin{proof}
Easy by definition.
\end{proof}

\begin{lem}
\label{lemma on t wedge and residues for koszul complexes}
We have
\begin{equation*}
\begin{split}
\res_{Y_{\ula}}^{\usi}[1] & (t \wedge) \\
&=(-1)^m(\id \otimes (t \wedge)[-m]) \res_{Y_{\ula}}^{\usi}
+\sum_{\nu \in \usi}((e_{\nu} \wedge)
\otimes
a_{\ula \cup (\usi \setminus \{\nu\}),\ula \cup \usi}^{\ast})
\res_{Y_{\ula}}^{\usi \setminus \{\nu\}} \\
&\qquad
:\kos_{Y_{\ula}}(M_{Y_{\ula}})
\longrightarrow
\varepsilon(\usi)
\otimes_{\bZ} \kos_{Y_{\ula \cup \usi}}(M_{Y_{\ula \cup \usi}})[1-m]
\end{split}
\end{equation*}
for $\usi \in S_m(\Lambda)$.
\end{lem}
\begin{proof}
Similar to the case of the log de Rham complex $\omega_{Y_{\ula}}$
in Lemma \ref{lemma on dlog t and residues}.
\end{proof}

\section{Gysin morphisms}
\label{gysin morphism}

In this section, we fix the notation on the so called ``Gysin map''.
Because the signs of objects in the cohomology groups
are crucial for our computation,
we start with the well-known objects
and fix the signs explicitly.

\begin{para}
\label{general defiition for Gysin morphisms:eq}
Let $X$ be a \ca space
equipped with a log structure $M_X$.
For the log de Rham complex $\omega_X$
of a log \ca space $X$,
we set an increasing filtration $W$ by
\begin{equation*}
W_m\omega_X^p
=\image(\omega_X^m \otimes_{\cO_X} \Omega_X^{p-m} \longrightarrow \omega_X^p)
\end{equation*}
as in \ref{increasing filtration W(usi)}.
Then we have
a morphism
\begin{equation*}
\gamma_m(\omega_X,W):
\gr_m^W\omega_X \longrightarrow \gr_{m-1}^W\omega_X[1]
\end{equation*}
in the derived category
as in \ref{gysin morphism for a filtered complex:eq}.
We use the symbol
$\gamma_m(X, M_X)$
instead of $\gamma_m(\omega_X,W)$.
We sometimes drop the subscript $m$,
if there is no danger of confusion.
\end{para}

\begin{para}
\label{smooth divisor case:eq}
First, we recall the simplest example.
Let $X$ be a complex manifold
and $D$ a smooth hypersurface in $X$.
The log structure $M_X(D)$ associated to the divisor $D$
is equipped to $X$.
In this case, the log de Rham complex $\omega_X$
is nothing but $\Omega_X(\log D)$,
and the increasing filtration $W$ coincides with
the usual weight filtration on $\Omega_X(\log D)$
in \cite{DeligneED}.
Then we have
$\gr_0^W\Omega_X(\log D)=W_0\Omega_X(\log D)=\Omega_X$
by definition.
Moreover we have $W_1\Omega_X(\log D)=\Omega_X(\log D)$
because $D$ is smooth.
We have the residue isomorphism of complexes
\begin{equation*}
\label{residue for irreducible D:eq}
\res_X^D:
\gr_1^W\Omega_X(\log D) \overset{\simeq}{\longrightarrow} \Omega_D[-1] ,
\end{equation*}
by which we identify $\gr_1^W\Omega_X(\log D)$
and $\Omega_D[-1]$.
Thus we obtain the morphism
\begin{equation*}
\gamma(X,D)=\gamma(X,M_X(D)):
\Omega_D[-1] \longrightarrow \Omega_X[1]
\end{equation*}
in the derived category.
\end{para}

\begin{prop}
\label{proposition for gysin and integration}
In addition to the situation above,
we assume that $X$ is compact.
Then we have the equality
\begin{equation*}
\int_X \coh^{p+1}(X,\gamma(X,D))(a) \cup b
=-(2\pi\sqrt{-1})\int_D a \cup (b|_D)
\end{equation*}
for any $a \in \coh^p(D,\Omega_D)$
and $b \in \coh^{2\dim X-2-p}(X,\Omega_X)$.
\end{prop}
\begin{proof}
See Griffiths-Schmid \cite[\S2 (b)]{GriffithsSchmid}.
\end{proof}

\begin{para}
Let $Y \longrightarrow \ast$ be a log deformation.
As defined in \ref{general defiition for Gysin morphisms:eq},
we have the morphism
\begin{equation*}
\gamma_m(Y_{\ula},M_{Y_{\ula}}^{\usi}):
\gr_m^W\omega_{Y_{\ula}^{\usi}}
\longrightarrow \gr_{m-1}^W\omega_{Y_{\ula}^{\usi}}[1]
\end{equation*}
in the derived category
for $\ula, \usi \in S(\Lambda)$.
In particular,
the morphism
\begin{equation*}
\gamma(Y_{\ula},M_{Y_{\ula}}^{\lambda}):
\gr_1^W\omega_{Y_{\ula}^{\lambda}}
\longrightarrow \gr_{0}^W\omega_{Y_{\ula}^{\lambda}}[1]
=\Omega_{Y_{\ula}}[1]
\end{equation*}
is obtained for $\lambda \in \Lambda$.
By the identification
\begin{equation*}
\Omega_{Y_{\ula \cup \{\lambda\}}}[-1]
\simeq
\varepsilon(\lambda) \otimes_{\bZ} \Omega_{Y_{\ula \cup \{\lambda\}}}[-1]
\simeq \gr_1^W\omega_{Y_{\ula}^{\lambda}}
\end{equation*}
we obtain a morphism
\begin{equation*}
\gamma_{Y^{\lambda}_{\ula}}:
\Omega_{Y_{\ula \cup \{\lambda\}}}[-1]
\longrightarrow
\Omega_{Y_{\ula}}[1]
\end{equation*}
in the derived category.
We have
\begin{equation}
\label{gysin for irreducible D:eq}
\gamma_{Y^{\lambda}_{\ula}}
=\gamma(Y_{\ula},Y_{\ula \cup \{\lambda\}})
\end{equation}
for $\lambda \notin \ula$.

The following proposition is very similar
to \cite[Proposition 4.3, Proposition 4.5]{NakkajimaSWSS}.
However, we restate it for the completeness,
because our definition of the residue isomorphism
are different from Nakkajima's.
Here we only give a sketch of the proof
because it is almost the same
as Proposition 4.5 in \cite{NakkajimaSWSS}.
\end{para}

\begin{prop}
\label{gamma for a log deformation}
For a positive integer $m$,
the morphism
\begin{equation*}
\gamma_m(Y_{\ula},M_{Y_{\ula}}):
\gr_m^W\omega_{Y_{\ula}}
\longrightarrow
\gr_{m-1}^W\omega_{Y_{\ula}}[1]
\end{equation*}
is identified with
\begin{equation*}
\begin{split}
\bigoplus_{\usi \in S_m(\Lambda)}
\sum_{\nu \in \usi}
(e_{\nu} \wedge)^{-1} &\otimes
\gamma_{Y^{\nu}_{\ula \cup(\usi \setminus \{\nu\})}}[1-m] \\
&:
\bigoplus_{\usi \in S_m(\Lambda)}
\varepsilon(\usi) \otimes_{\bZ} \Omega_{Y_{\ula \cup \usi}}[-m]
\longrightarrow
\bigoplus_{\uta \in S_{m-1}(\Lambda)}
\varepsilon(\uta) \otimes_{\bZ} \Omega_{Y_{\ula \cup \uta}}[2-m]
\end{split}
\end{equation*}
under the isomorphism
\eqref{gr of W for omega:eq},
where $e_{\nu} \wedge$ denotes the isomorphism
$\varepsilon(\usi \setminus \{\nu\}) \longrightarrow \varepsilon(\usi)$
in \eqref{isomorphism e wedge:eq}.
\end{prop}
\begin{proof}
The canonical inclusion
$\omega_{Y^{\usi}_{\ula}} \hookrightarrow \omega_{Y_{\ula}}$
induces the commutative diagram
\begin{equation*}
\begin{CD}
0 @>>> \gr_{m-1}^W\omega_{Y^{\usi}_{\ula}}
  @>>> W_m\omega_{Y^{\usi}_{\ula}}/W_{m-2}\omega_{Y^{\usi}_{\ula}}
  @>>> \gr_m^W\omega_{Y^{\usi}_{\ula}} @>>> 0 \\
@. @VVV @VVV @VVV \\
0 @>>> \gr_{m-1}^W\omega_{Y_{\ula}}
  @>>> W_m\omega_{Y_{\ula}}/W_{m-2}\omega_{Y_{\ula}}
  @>>> \gr_m^W\omega_{Y_{\ula}} @>>> 0
\end{CD}
\end{equation*}
with exact rows
for $\usi \in S_m(\Lambda)$.
Thus the restriction of $\gamma_m(Y_{\ula},M_{Y_{\ula}})$
on the direct summand
$\varepsilon(\usi) \otimes_{\bZ} \Omega_{Y_{\ula \cup \usi}}[-m]$
under the identification
\eqref{gr of W for omega:eq}
coincides with
$\gamma_m(Y_{\ula},M^{\usi}_{Y_{\ula}})$
under the identification
$\gr_m^W\omega_{Y^{\usi}_{\ula}}
\simeq \varepsilon(\usi) \otimes_{\bZ} \Omega_{Y_{\ula \cup \usi}}[-m]$.
For $\uta \in S_{m-1}(\usi)$,
$\res^{\uta}_{Y_{\ula}}$ induces the commutative diagram
{\small \begin{equation*}
\begin{CD}
0 @>>> \gr_{m-1}^W\omega_{Y^{\usi}_{\ula}}
  @>>> W_m\omega_{Y^{\usi}_{\ula}}/W_{m-2}\omega_{Y^{\usi}_{\ula}}
  @>>> \gr_m^W\omega_{Y^{\usi}_{\ula}} @>>> 0 \\
@. @VVV @VVV @VVV \\
0 @>>> \varepsilon(\uta) \otimes_{\bZ} \Omega_{Y_{\ula \cup \usi}}[1-m]
  @>>> \varepsilon(\uta) \otimes_{\bZ} \omega_{Y^{\nu}_{\ula \cup \uta}}[1-m]
  @>>> \varepsilon(\uta) \otimes_{\bZ} \Omega_{Y_{\ula \cup \usi}}[-m]
  @>>> 0
\end{CD}
\end{equation*}
}

\noindent
with exact rows,
where $\usi \setminus \uta=\{\nu\}$.
Then we have
\begin{equation*}
\begin{split}
\res_{Y_{\ula}}^{\uta}[1]\gamma_m(Y_{\ula},M^{\usi}_{Y_{\ula}})
=&(-1)^{m-1}\chi(\uta,\{\nu\})^{-1}
\otimes \gamma_{Y^{\nu}_{\ula \cup \uta}} \\
&:
\varepsilon(\usi) \otimes_{\bZ} \Omega_{Y_{\ula \cup \usi}}[-m]
\longrightarrow
\varepsilon(\uta) \otimes_{\bZ} \Omega_{Y_{\ula \cup \uta}}[2-m]
\end{split}
\end{equation*}
by using \eqref{gamma(K[l],W) and gamma(K,W)[l]:eq}.
Because of $\chi(\uta,\{\nu\})=(-1)^{m-1}e_{\nu} \wedge$
under the identification $\varepsilon(\nu) \simeq \bZ$,
the conclusion is obtained.
\end{proof}

\section{Comparison between $A$ and $K$}
\label{comparison iso from A to K}

In this section,
we first recall results in \cite{FujisawaMHSLSD}
and adjust them to the case of a log deformation.
The definition and the notation are slightly changed
from that in \cite{FujisawaMHSLSD}.
In addition to this change,
the method in Section \ref{Cech complex} is used
instead of the simplicial method in \cite{FujisawaMHSLSD}
because it can work without fixing the total order
on the index set.
After recalling the results in Steenbrink \cite{SteenbrinkLE},
Fujisawa-Nakayama \cite{Fujisawa-Nakayama} briefly,
we construct a morphism
from Steenbrink's \cmh complex
to the complex in \cite{FujisawaMHSLSD}.

\begin{para}
\label{C-construction for ula}
Let $Y \longrightarrow \ast$ be a log deformation.
We assume that
\begin{mylist}
\itemno
\label{finiteness assumption:eq}
$Y$ has finitely many irreducible components
\end{mylist}
in the remainder of this article.
Then the index set $\Lambda$ of the irreducible components of $Y$
is a finite set.
This assumption does not affect to our main results
because we are interested only in the case
where $Y$ is compact.
\end{para}

\begin{para}
\label{definition of C-str on Y{ula}}
Fix an element $\ula \in S(\Lambda)$.
A morphism
\begin{equation*}
\nabla:\bC[u] \otimes_{\bC} \omega^p_{Y_{\ula}}
\longrightarrow
\bC[u] \otimes_{\bC} \omega^{p+1}_{Y_{\ula}}
\end{equation*}
is defined by
\begin{equation}
\label{morphism nabla for C-str:eq}
\nabla=\id \otimes d+(2\pi\sqrt{-1})^{-1}\frac{d}{du} \otimes \dlog t \wedge
\end{equation}
for a non-negative integer $p$,
where $\dlog t \wedge$ is the morphism \eqref{morphism dlog t wedge:eq}.
We can easily see the equality $\nabla^2=0$.
Thus we obtain a complex of $\bC$-sheaves on $Y_{\ula}$,
which is denoted by $(\bC[u] \otimes_{\bC} \omega_{Y_{\ula}},\nabla)$
or simply by $\bC[u] \otimes_{\bC} \omega_{Y_{\ula}}$.
A morphism of complexes
\begin{equation}
\label{inclusion from omega to C[u] omega:eq}
\omega_{Y_{\ula}} \longrightarrow \bC[u] \otimes_{\bC} \omega_{Y_{\ula}}
\end{equation}
is induced
by the natural inclusion $\bC \longrightarrow \bC[u]$.
We consider $\omega_{Y_{\ula}}$ as a subcomplex
of $\bC[u] \otimes_{\bC} \omega_{Y_{\ula}}$
by the inclusion above.

By using the identity
\begin{equation*}
\bC[u] \otimes_{\bC} \omega_{Y_{\ula}}
=\bigoplus_{r \ge 0}\bC u^r \otimes_{\bC} \omega_{Y_{\ula}},
\end{equation*}
the weight filtration $W$
and the Hodge filtration $F$ on $\bC[u] \otimes_{\bC} \omega_{Y_{\ula}}$
are defined by
\begin{align*}
&W_m(\bC[u] \otimes_{\bC} \omega_{Y_{\ula}})=
\bigoplus_{r \ge 0}\bC u^r \otimes_{\bC} W_{m-2r}\omega_{Y_{\ula}} \\
&F^p(\bC[u] \otimes_{\bC} \omega_{Y_{\ula}})
=\bigoplus_{r \ge 0}\bC u^r \otimes_{\bC} F^{p-r}\omega_{Y_{\ula}}
\end{align*}
for every $m, p$,
where $F$ on $\omega_{Y_{\ula}}$ denotes the stupid filtration
as in \cite[(1.4.6)]{DeligneII}.
It can be easily seen that
the filtrations $W$ and $F$
are preserved by $\nabla$.
Thus we obtain a bifiltered complex
$(\bC[u] \otimes_{\bC} \omega_{Y_{\ula}},W,F)$.

By setting
\begin{equation*}
\pi_{\bC,\ula, r}(P(u) \otimes \omega)
=\frac{d^rP}{du^r}(0) \otimes \omega ,
\end{equation*}
a morphism
\begin{equation}
\pi_{\bC,\ula,r}:
\bC[u] \otimes_{\bC} \omega_{Y_{\ula}}^p \longrightarrow \omega_{Y_{\ula}}^p
\end{equation}
is obtained for every non-negative integer $r$.
Note that $\pi_{\bC,\ula,r}$ does not define a morphism of complexes.
We have
\begin{align*}
&\pi_{\bC,\ula,r}(W_m(\bC[u] \otimes_{\bC} \omega_{Y_{\ula}}^p))
\subset
W_{m-2r}\omega_{Y_{\ula}}^p \\
&\pi_{\bC,\ula,r}(F^q(\bC[u] \otimes_{\bC} \omega_{Y_{\ula}}^p))
\subset
F^{q-r}\omega_{Y_{\ula}}^p
\end{align*}
for every $m,q$.
It is easy to see that
\begin{equation*}
\gr_m^W\pi_{\bC,\ula,r}:
(\gr_m^W(\bC[u] \otimes_{\bC} \omega_{Y_{\ula}}),F)
\longrightarrow
(\gr_{m-2r}^W\omega_{Y_{\ula}},F[-r])
\end{equation*}
defines a morphism of filtered complexes,
although the morphism $\pi_{\bC,\ula,r}$ is not a morphism of complexes.
Moreover, the morphism of filtered complexes
\begin{equation}
\label{the morphism pi0:eq}
\pi_{\ula/\ast}:
(\bC[u] \otimes_{\bC} \omega_{Y_{\ula}},F)
\longrightarrow
(\omega_{Y_{\ula}/\ast},F)
\end{equation}
is given
by composing the morphism $\pi_{\bC,\ula,0}$
and the canonical projection
$\omega_{Y_{\ula}}^p \longrightarrow \omega_{Y_{\ula}/\ast}^p$.
\end{para}

\begin{para}
\label{Q-construction for ula}
We have a complex
\begin{equation*}
\bQ[u] \otimes_{\bQ} \kos_{Y_{\ula}}(M_{Y_{\ula}})
\end{equation*}
with the differential
\begin{equation}
\label{morphism mabla for Q-str:eq}
\nabla=\id \otimes d+\frac{d}{du} \otimes t \wedge ,
\end{equation}
where $t \wedge$ is the morphism defined
in \eqref{the morphism t wedge on Kos:eq}.
Moreover, an increasing filtration $W$ on
$\bQ[u] \otimes_{\bQ} \kos_{Y_{\ula}}(M_{Y_{\ula}})$
is defined by
\begin{equation*}
W_m(\bQ[u] \otimes_{\bQ} \kos_{Y_{\ula}}(M_{Y_{\ula}}))
=\bigoplus_{r \ge 0}\bQ u^r
\otimes_{\bQ} W_{m-2r}\kos_{Y_{\ula}}(M_{Y_{\ula}})
\end{equation*}
for every $m$.

We define a morphism
\begin{equation*}
\pi_{\bQ,\ula}:
\bQ[u] \otimes_{\bQ} \kos_{Y_{\ula}}(M_{Y_{\ula}})^p
\longrightarrow
\kos_{Y_{\ula}}(M_{Y_{\ula}})^p
\end{equation*}
by substituting $0$ for the variable $u$
as in the case of $\pi_{\bC,\ula,0}$.
Note that $\pi_{\bQ,\ula}$ is not a morphism of complexes.
It is clear that $\pi_{\bQ,\ula}$ preserves the filtration $W$
and induces a morphism of complexes
\begin{equation*}
\gr_m^W\pi_{\bQ,\ula}:
\gr_m^W(\bQ[u] \otimes_{\bQ} \kos_{Y_{\ula}}(M_{Y_{\ula}}))
\longrightarrow
\gr_m^W\kos_{Y_{\ula}}(M_{Y_{\ula}})
\end{equation*}
for every $m$.
\end{para}

\begin{para}
\label{morphism from Q to C for ula}
We have the morphism of complexes
\begin{equation*}
\psi_{\ula,0}=\psi_{(Y_{\ula},M_{Y_{\ula}})}:
\kos_{Y_{\ula}}(M_{Y_{\ula}})
\longrightarrow
\omega_{Y_{\ula}}
\end{equation*}
defined in \eqref{definition of psi:eq}.

Tensoring with the canonical inclusion
$\bQ[u] \longrightarrow \bC[u]$
with the morphism $\psi_{(Y_{\ula},M_{Y_{\ula}})}$,
we obtain a morphism
\begin{equation*}
\psi_{\ula}:
\bQ[u] \otimes_{\bQ} \kos_{Y_{\ula}}(M_{Y_{\ula}})
\longrightarrow
\bC[u] \otimes_{\bC} \omega_{Y_{\ula}}
\end{equation*}
for every $\ula$.
The commutative diagram
\eqref{commutative diagram for t wedge and dlog t wedge:eq}
tells us that the morphism $\psi_{\ula}$
is a morphism of complexes.
\end{para}

\begin{para}
\label{several Cech complex:eq}
The construction in
\ref{C-construction for ula}--\ref{morphism from Q to C for ula}
is functorial with respect to the morphisms
induced from the canonical inclusion $Y_{\umu} \subset Y_{\ula}$
for $\ula \subset \umu$.
Thus we obtain the corresponding co-cubical objects
over the cubical log complex manifold $Y_{\bullet}$.

Then we obtain the filtered co-cubical complexes of $\bQ$-sheaves
\begin{equation*}
(\kos_{Y_{\bullet}}(M_{Y_{\bullet}}),W), \quad
(\bQ[u] \otimes_{\bQ} \kos_{Y_{\bullet}}(M_{Y_{\bullet}}),W) ,
\end{equation*}
the (bi)filtered co-cubical complexes of $\bC$-sheaves
\begin{equation*}
(\omega_{Y_{\bullet}/\ast},F), \quad
(\omega_{Y_{\bullet}},W, F), \quad
(\bC[u] \otimes_{\bC} \omega_{Y_{\bullet}},W, F)
\end{equation*}
and the morphisms
\begin{align*}
&\psi_{\bullet,0}:
(\kos_{Y_{\bullet}}(M_{Y_{\bullet}}),W)
\longrightarrow
(\omega_{Y_{\bullet}},W) \\
&\psi_{\bullet}:
(\bQ[u] \otimes_{\bQ} \kos_{Y_{\bullet}}(M_{Y_{\bullet}}),W)
\longrightarrow (\bC[u] \otimes_{\bC} \omega_{Y_{\bullet}},W) \\
&\pi_{\bQ,\bullet}:
(\bQ[u] \otimes_{\bQ} \kos_{Y_{\bullet}}(M_{Y_{\bullet}})^p,W)
\longrightarrow
(\kos_{Y_{\bullet}}(M_{Y_{\bullet}})^p,W) \\
&\pi_{\bC,\bullet,r}:
(\bC[u] \otimes_{\bC} \omega^p_{Y{\bullet}},W,F)
\longrightarrow
(\omega^p_{Y_{\bullet}},W[2r],F[-r]) \\
&\pi_{\bullet/\ast}:
(\bC[u] \otimes_{\bC} \omega_{Y_{\bullet}},F)
\longrightarrow (\omega_{Y_{\bullet}/\ast}, F)
\end{align*}
on $Y_{\bullet}$.
So we obtain the filtered complexes of $\bQ$-sheaves
\begin{equation*}
(\C(\kos_{Y_{\bullet}}(M_{Y_{\bullet}})), \delta W), \quad
(\C(\bQ[u] \otimes_{\bQ} \kos_{Y_{\bullet}}(M_{Y_{\bullet}})),\delta W)
\end{equation*}
and the (bi)filtered complexes of $\bC$-sheaves
\begin{equation*}
(\C(\omega_{Y_{\bullet}/\ast}), F), \quad
(\C(\omega_{Y_{\bullet}}), \delta W, F),  \quad 
(\C(\bC[u] \otimes_{\bC} \omega_{Y_{\bullet}}),\delta W,F),
\end{equation*}
on $Y$ as in \ref{complexes from co-cubical objects:eq}
and \ref{filtrations on C(K):eq}.
We set
\begin{align*}
&(K_{\bQ},W)
=(\C(\bQ[u] \otimes_{\bQ} \kos_{Y_{\bullet}}(M_{Y_{\bullet}})),\delta W) \\
&(K_{\bC},W,F)
=(\C(\bC[u] \otimes_{\bC} \omega_{Y_{\bullet}}),\delta W,F)
\end{align*}
for short.
Moreover
the morphisms of filtered complexes
\begin{align*}
&\psi_0=\C(\psi_{\bullet,0}):
(\C(\kos_{Y_{\bullet}}(M_{Y_{\bullet}})), \delta W)
\longrightarrow
(\C(\omega_{Y_{\bullet}}),\delta W) \\
&\psi=\C(\psi_{\bullet}):
(K_{\bQ},W) \longrightarrow (K_{\bC},W) \\
&\pi_{/\ast}=\C(\pi_{\bullet/\ast}):
(K_{\bC},F) \longrightarrow (\C(\omega_{Y_{\bullet}/\ast}),F)
\end{align*}
are obtained.
Moreover, $\C(\omega_{Y_{\bullet}})$ is considered
as a subcomplex of $K_{\bC}$
by the inclusion \eqref{inclusion from omega to C[u] omega:eq}.
\end{para}

\begin{para}
The morphisms $\pi_{\bQ,\bullet}$ and $\pi_{\bC,\bullet,r}$
induce morphisms
\begin{align*}
&\pi_{\bQ}=\C(\pi_{\bQ,\bullet}):
(K^p_{\bQ},W)
\longrightarrow
(\C(\kos_{Y_{\bullet}}(M_{Y_{\bullet}}))^p, \delta W) \\
&\pi_{\bC,r}=\C(\pi_{\bC,\bullet,r}):
(K^p_{\bC},W,F)
\longrightarrow
(\C(\omega_{Y_{\bullet}})^p, \delta W[2r],F[-r])
\end{align*}
for every $p$,
which induce a morphism of complexes
\begin{equation*}
\gr_m^W\pi_{\bQ}:
\gr_m^WK_{\bQ}
\longrightarrow
\gr_m^{\delta W}\C(\kos_{Y_{\bullet}}(M_{Y_{\bullet}}))
\end{equation*}
and a morphism of filtered complexes
\begin{equation}
\label{morphism grWpi:eq}
\gr_m^W\pi_{\bC,r}:
(\gr_m^WK_{\bC},F)
\longrightarrow
(\gr_{m-2r}^{\delta W}\C(\omega_{Y_{\bullet}}),F[-r])
\end{equation}
for every $m,r$.
We have the commutative diagram
\begin{equation}
\label{commutaive diagram for grWpi0:eq}
\begin{CD}
\gr_m^WK_{\bQ} @>{\gr_m^W\psi}>> \gr_m^WK_{\bC} \\
@V{\gr_m^W\pi_{\bQ,0}}VV @VV{\gr_m^W\pi_{\bC,0}}V \\
\gr_m^{\delta W}\C(\kos_{Y_{\bullet}}(M_{Y_{\bullet}}))
@>>{\gr_m^{\delta W}\psi_0}> \gr_m^{\delta W}\C(\omega_{Y_{\bullet}})
\end{CD}
\end{equation}
for every $m$.
\end{para}

\begin{para}
\label{morphism from omegaY to K, somega somega/ast}
For an element $\lambda \in \Lambda$,
the morphism
\begin{equation*}
a^{\ast}_{\lambda}:
\omega_Y^p \longrightarrow
\omega^p_{Y_{\lambda}}
\end{equation*}
can be regarded as a morphism
$\omega^p_Y \longrightarrow
\C(\omega_{Y_{\bullet}})^{0,p} \subset \C(\omega_{Y_{\bullet}})^p$
for every $p$.
We set
\begin{equation*}
a_0^{\ast}=\sum_{\lambda \in \Lambda}
a_{\lambda}^{\ast}:
\omega_Y^p
\longrightarrow
\C(\omega_{Y_{\bullet}})^p,
\end{equation*}
which induces a morphism
\begin{equation*}
a_0^{\ast}: (\omega_Y,W,F) \longrightarrow (\C(\omega_{Y_{\bullet}}),\delta W,F)
\end{equation*}
of bifiltered complexes.
The composite of $a_0^{\ast}$
and the canonical inclusion
$\C(\omega_{Y_{\bullet}}) \longrightarrow K_{\bC}$
is denoted by $a^{\ast}$.
Thus a morphism of bifiltered complexes
\begin{equation*}
a^{\ast}:(\omega_Y,W,F) \longrightarrow (K_{\bC},W,F)
\end{equation*}
is obtained.
Then the equality
$a_0^{\ast}=\pi_{\bC,0}a^{\ast}$ holds by definition.
A morphism of filtered complexes
\begin{equation}
\label{simplicial resolution of omegaY/*:eq}
a^{\ast}_{/\ast}:
(\omega_{Y/\ast},F)
\longrightarrow
(\C(\omega_{Y/\ast}),F)
\end{equation}
is defined by the same way.

Morphisms of filtered complexes
\begin{equation*}
\begin{split}
&a_0^{\ast}:
(\kos_Y(M_Y),W)
\longrightarrow
(\C(\kos_{Y_{\bullet}}(M_{Y_{\bullet}})),\delta W) \\
&a^{\ast}:
(\kos_Y(M_Y),W)
\longrightarrow
(K_{\bQ},W)
\end{split}
\end{equation*}
are defined similarly.
Then the diagrams
\begin{equation}
\label{commutativity of a* for Q and C:eq}
\begin{CD}
\kos_Y(M_Y) @>{\psi_{(Y,M_Y)}}>> \omega_Y \\
@V{a_0^{\ast}}VV @VV{a_0^{\ast}}V \\
\C(\kos_{Y_{\bullet}}(M_{Y_{\bullet}}))
@>>{\psi_0}> \C(\omega_{Y_{\bullet}})
\end{CD}
\qquad \qquad
\begin{CD}
\kos_Y(M_Y) @>{\psi_{(Y,M_Y)}}>> \omega_Y \\
@V{a^{\ast}}VV @VV{a^{\ast}}V \\
K_{\bQ}
@>>{\psi}> K_{\bC}
\end{CD}
\end{equation}
are commutative.
\end{para}

\begin{defn}
We set
\begin{equation*}
(K,W,F)=((K_{\bQ},W), (K_{\bC},W,F),\psi) \\
\end{equation*}
and
\begin{equation*}
(\coh^q(Y,K),W,F)=
((\coh^q(Y,K_{\bQ}),W),(\coh^q(Y,K_{\bC}),W,F),\coh^q(Y,\psi)) ,
\end{equation*}
for an integer $q$.
\end{defn}

\begin{thm}
\label{theorem for K}
For a log deformation  $Y \longrightarrow \ast$,
the morphism $a^{\ast}_{/\ast}$ in \eqref{simplicial resolution of omegaY/*:eq}
is a filtered quasi-isomorphism
with respect to the filtrations $F$ on both sides.
Therefore the morphism $a^{\ast}_{/\ast}$
induces an isomorphism
\begin{equation*}
\coh^q(Y,a^{\ast}_{/\ast}):
\coh^q(Y,\omega_{Y/\ast})
\longrightarrow
\coh^q(Y,\C(\omega_{Y_{\bullet}/\ast}))
\end{equation*}
for every integer $q$,
under which the filtrations $F$ on both sides coincide.

If we assume the following conditions
\begin{mylist}
\itemno
\label{compactness of Y:eq}
$Y \longrightarrow \ast$ is proper, that is,
$Y$ is compact,
\itemno
\label{Kahler condition for Y:eq}
all the irreducible components $Y_{\lambda}$
are K\"ahler complex manifolds
\end{mylist}
in addition,
then we have the following\,$:$
\begin{mylist}
\itemno
The morphism $\pi_{/\ast}$ induces an isomorphism
$\coh^q(Y,K_{\bC}) \longrightarrow \coh^q(Y,\C(\omega_{Y_{\bullet}/\ast}))$
for every integer $q$,
under which the filtrations $F$ on both sides coincide.
\itemno
The data $(\coh^q(Y,K),W[q],F)$ is a mixed Hodge structure
for every integer $q$.
\itemno
The spectral sequence $E_r^{p,q}(K_{\bC},F)$
degenerates at $E_1$-terms.
\itemno
The spectral sequences
$E_r^{p,q}(K_{\bQ},W)$ and $E_r^{p,q}(K_{\bC},W)$
degenerate at $E_2$-terms.
\end{mylist}
\end{thm}
\begin{proof}
We can deduce the conclusion from \cite{FujisawaMHSLSD}
by fixing a total order on $\Lambda$.
\end{proof}

\begin{rmk}
\label{isomorphisms for somega and K}
Here, we recall several isomorphisms
for the later use.
We note that we can use $\bigoplus_{\lambda \in \Lambda^{k+1,\circ}}$
instead of $\prod_{\lambda \in \Lambda^{k+1,\circ}}$
by the assumption \eqref{finiteness assumption:eq}.

For the complex $\C(\kos_{Y_{\bullet}}(M_{Y_{\bullet}}))$
we have the isomorphism in the derived category
\begin{equation}
\label{isomorphism for grmWskos:eq}
\begin{split}
\gr_m^{\delta W}\C(\kos_{Y_{\bullet}}(M_{Y_{\bullet}}))
&=\bigoplus_{\lambda \in \dprod\Lambda}
\gr_{m+d(\lambda)}^W\kos_{Y_{\ula}}(M_{Y_{\ula}})[-d(\lambda)] \\
&\simeq
\bigoplus_{\lambda \in \dprod\Lambda}
\bigoplus_{\usi \in S_{m+d(\lambda)}(\Lambda)}
\varepsilon(\usi)
\otimes_{\bZ} \bQ_{Y_{\ula \cup \usi}}[-m-2d(\lambda)]
\end{split}
\end{equation}
by \eqref{gr for delta W:eq}
and by \eqref{residue iso in the derived category:eq}.
For $\C(\omega_{Y_{\bullet}})$,
we have the isomorphism of complexes
\begin{equation}
\label{isomorphism for grmWsomega:eq}
\begin{split}
\gr_m^{\delta W}\C(\omega_{Y_{\bullet}})
&=\bigoplus_{\lambda \in \dprod\Lambda}
\gr_{m+d(\lambda)}^W\omega_{Y_{\ula}}[-d(\lambda)] \\
&\simeq
\bigoplus_{\lambda \in \dprod\Lambda}
\bigoplus_{\usi \in S_{m+d(\lambda)}(\Lambda)}
\varepsilon(\usi) \otimes_{\bZ}
\Omega_{Y_{\ula \cup \usi}}[-m-2d(\lambda)]
\end{split}
\end{equation}
by \eqref{gr for delta W:eq}
and by the residue isomorphism \eqref{gr of W for omega:eq}.
Under the identifications \eqref{isomorphism for grmWskos:eq}
and \eqref{isomorphism for grmWsomega:eq},
the morphism
$\gr_m^W\psi_0$
coincides with the morphism
induced by the inclusion
\begin{equation*}
(2\pi\sqrt{-1})^{-m-d(\lambda)}\iota: \bQ \longrightarrow \bC
\end{equation*}
on the direct summand
$\varepsilon(\usi) \otimes_{\bZ} \bQ_{Y_{\ula \cup \usi}}[-m-2d(\lambda)]$
by Lemma \ref{commutativity on Q and C for trivial log}
and by Lemma \ref{commutativity of residue for omega and koszul}.

Similarly, we have the isomorphism in the derived category
\begin{equation}
\label{isomorphism for grmWKQ:eq}
\begin{split}
\gr_m^WK_{\bQ}
&=\bigoplus_{r \ge 0}
\bigoplus_{\lambda \in \dprod\Lambda}
\bQ u^r
\otimes_{\bQ}
\gr_{m+d(\lambda)-2r}^W\kos_{Y_{\ula}}(M_{Y_{\ula}})[-d(\lambda)] \\
&\simeq
\bigoplus_{r \ge 0}
\bigoplus_{\lambda \in \dprod\Lambda}
\bigoplus_{\usi \in S_{m+d(\lambda)-2r}(\Lambda)}
\varepsilon(\usi)
\otimes_{\bZ} \bQ_{Y_{\ula \cup \usi}}[-m-2d(\lambda)+2r]
\end{split}
\end{equation}
and the isomorphism of complexes
\begin{equation}
\label{isomorphism for grmWKC:eq}
\begin{split}
\gr_m^WK_{\bC}
&=\bigoplus_{r \ge 0}
\bigoplus_{\lambda \in \dprod\Lambda}
\bC u^r \otimes_{\bC}
\gr_{m+d(\lambda)-2r}^W\omega_{Y_{\ula}}[-d(\lambda)] \\
&\simeq
\bigoplus_{r \ge 0}
\bigoplus_{\lambda \in \dprod\Lambda}
\bigoplus_{\usi \in S_{m+d(\lambda)-2r}(\Lambda)}
\varepsilon(\usi)
\otimes_{\bZ}
\Omega_{Y_{\ula \cup \usi}}[-m-2d(\lambda)+2r]
\end{split}
\end{equation}
as above.
Lemma \ref{commutativity of residue for omega and koszul}
tells us that the morphism
$\gr_m^W\psi: \gr_m^WK_{\bQ} \longrightarrow \gr_m^WK_{\bC}$
is identified with the morphism induced by the inclusion
\begin{equation*}
(2\pi\sqrt{-1})^{2r-d(\lambda)-m}\iota:
\bQ \longrightarrow \bC
\end{equation*}
on the direct summand
$\varepsilon(\usi) \otimes_{\bZ}
\bQ_{Y_{\ula \cup \usi}}[-m-2d(\lambda)+2r]$.
\end{rmk}

\begin{para}
Now we compare the Gysin morphisms of $K_{\bC}$
and of $\C(\omega_{Y_{\bullet}})$ for the later use.
For this purpose,
we introduce a new complex.

The morphism
\begin{equation*}
\id \otimes d:
\bC[u] \otimes_{\bC} \omega_{Y_{\ula}}^p
\longrightarrow
\bC[u] \otimes_{\bC} \omega_{Y_{\ula}}^{p+1}
\end{equation*}
yields a complex $(\bC[u] \otimes_{\bC} \omega_{Y_{\ula}}, \id \otimes d)$
for every $\ula$.
Thus we obtain a co-cubical complex
$(\bC[u] \otimes_{\bC} \omega_{Y_{\bullet}}, \id \otimes d)$.
We set
$L=\C(\bC[u] \otimes_{\bC} \omega_{Y_{\bullet}},\id \otimes d)$
for a while.
The complex $L$ carries the filtration $W$ and $F$
by the same definition for $K_{\bC}$.
Then we trivially have $(\gr_m^WL,F)=(\gr_m^WK_{\bC},F)$
for every $m$.
Moreover, the morphism $\pi_{\bC,r}$
defines a morphism of bifiltered complexes
\begin{equation*}
\pi_{\bC,r}:
(L,W,F) \longrightarrow (\C(\omega_{Y_{\bullet}}),\delta W[2r], F[-r])
\end{equation*}
for every $r \ge 0$.

The morphism
\begin{equation*}
(-1)^k\frac{d}{du} \otimes \dlog t \wedge:
\bC[u] \otimes_{\bC} \omega_{Y_{\ula}}^{p-k}
\longrightarrow
\bC[u] \otimes_{\bC} \omega_{Y_{\ula}}^{p+1-k}
\end{equation*}
for $\lambda \in \Lambda^{k+1,\circ}$
induces morphisms of bifiltered complexes
\begin{equation*}
(L,W,F) \longrightarrow (L[1],W[1],F), \quad
(K_{\bC},W,F) \longrightarrow (K_{\bC}[1],W[1],F)
\end{equation*}
which are denoted by $\C(d/du \otimes \dlog t \wedge)$.
Similarly,
morphisms of bifiltered complexes
\begin{align*}
&\C(\id \otimes \dlog t \wedge):
(K_{\bC},W,F)
\longrightarrow
(K_{\bC}[1],W[-1],F[1]) \\
&\C(\dlog t \wedge):
(\C(\omega_{Y_{\bullet}}),\delta W,F)
\longrightarrow
(\C(\omega_{Y_{\bullet}})[1],\delta W[-1],F[1])
\end{align*}
are induced by the morphisms
\begin{align*}
&(-1)^k\id \otimes \dlog t \wedge:
\bC[u] \otimes_{\bC} \omega_{Y_{\ula}}^{p-k}
\longrightarrow
\bC[u] \otimes_{\bC} \omega_{Y_{\ula}}^{p+1-k} \\
&(-1)^k\dlog t \wedge:
\omega_{Y_{\ula}}^{p-k}
\longrightarrow
\omega_{Y_{\ula}}^{p+1-k}
\end{align*}
for every $\lambda \in \Lambda^{k+1,\circ}$
and for every $p$.

We have the equality
\begin{equation}
\label{equality for pi and dlog t:eq}
\pi_{\bC,r}[1] \C(\frac{d}{du} \otimes \dlog t \wedge)
=\C(\dlog t \wedge)\pi_{\bC,r+1}
\end{equation}
by definition.
\end{para}

\begin{lem}
We have
\begin{equation*}
\gamma_m(K_{\bC},W)
=\gamma_m(L,W)
+(2\pi\sqrt{-1})^{-1}\gr_m^W\C(\frac{d}{du} \otimes \dlog t \wedge)
\end{equation*}
for every $m$.
\end{lem}
\begin{proof}
By Proposition \ref{gamma for twisted filtered complex}.
\end{proof}

\begin{prop}
\label{gamma for K and for somega}
We have
\begin{equation*}
\begin{split}
&\gr_{m-1}^W\pi_{\bC,r}[1]\gamma_m(K_{\bC},W) \\
&\qquad
=\gamma_{m-2r}(\C(\omega_{Y_{\bullet}}),\delta W)
\gr_m^W\pi_{\bC,r}
+(2\pi\sqrt{-1})^{-1}
\gr_{m-2r-2}^{\delta W}\C(\dlog t \wedge)\gr_m^W\pi_{\bC,r+1}
\end{split}
\end{equation*}
for every $m,r$.
\end{prop}
\begin{proof}
We obtain
\begin{equation*}
\gr_{m-1}^W\pi_{\bC,r}[1]\gamma_m(L,W)
=\gamma_{m-2r}(\C(\omega_{Y_{\bullet}}),\delta W)\gr_m^W\pi_{\bC,r}
\end{equation*}
by the functoriality of the Gysin morphism.
Then we obtain the conclusion
by \eqref{equality for pi and dlog t:eq}
and by the lemma above.
\end{proof}

\begin{defn}
For $\lambda \in \Lambda^{k+1,\circ}$
and for $\usi \in S_{m+k}(\Lambda)$,
a morphism
\begin{equation*}
\Pi_0(\lambda,\usi):
\gr_m^{\delta W}\C(\omega_{Y_{\bullet}})
\longrightarrow
\varepsilon(\usi) \otimes_{\bZ} \Omega_{Y_{\ula \cup \usi}}[-m-2k]
\end{equation*}
is defined as the projection onto a direct summand
under the identification
\eqref{isomorphism for grmWsomega:eq}.
In particular,
we have a morphism
\begin{equation*}
\Pi_0(\lambda)=\Pi_0(\lambda,\ula):
\gr_1^{\delta W}\C(\omega_{Y_{\bullet}})
\longrightarrow
\varepsilon(\ula) \otimes_{\bZ} \Omega_{Y_{\ula}}[-1-2k]
\end{equation*}
for $\lambda \in \Lambda^{k+1,\circ}$.
We set
\begin{equation}
\label{definition of Theta0:eq}
\Theta_{\bC,0}(\lambda)
=((e_{\lambda} \wedge)^{-1} \otimes \id)\Pi_0(\lambda):
\gr_1^{\delta W}\C(\omega_{Y_{\bullet}})
\longrightarrow
\Omega_{Y_{\ula}}[-1-2k]
\end{equation}
for every $\lambda \in \Lambda^{k+1,\circ}$,
where $e_{\lambda} \wedge$ is the isomorphism
\eqref{isomorphism e wedge:eq}.
Moreover a morphism
\begin{equation*}
\Theta_{\bC}(\lambda):
\gr_1^WK_{\bC} \longrightarrow \Omega_{Y_{\ula}}[-1-2k]
\end{equation*}
is defined by
$\Theta_{\bC}(\lambda)=\Theta_{\bC,0}(\lambda) \gr_1^W\pi_{\bC,0}$,
for $\lambda \in \Lambda^{k+1,\circ}$.
\end{defn}

\begin{lem}
\label{Theta and gamma for K and somega}
In the situation above,
we have
\begin{equation*}
\begin{split}
\Theta_{\bC}(\lambda)[2] &\gr_0^W\C(\id \otimes \dlog t \wedge)[1]
\gamma_1(K_{\bC},W) \\
&=\Theta_{\bC,0}(\lambda)[2]
\gr_0^{\delta W}\C(\dlog t \wedge)[1]
\gamma_1(\C(\omega_{Y_{\bullet}}),\delta W)
\gr_1^W\pi_{\bC,0} \\
\end{split}
\end{equation*}
for $\lambda \in \Lambda^{k+1,\circ}$.
\end{lem}
\begin{proof}
Easy by Proposition \ref{gamma for K and for somega}.
\end{proof}

\begin{para}
The morphism
\begin{equation*}
\frac{d}{du} \otimes \id:
\bC[u] \otimes_{\bC} \omega_{Y_{\ula}}
\longrightarrow
\bC[u] \otimes_{\bC} \omega_{Y_{\ula}}
\end{equation*}
induces a morphism of complexes
\begin{equation*}
\frac{d}{du} \otimes \id:
K_{\bC} \longrightarrow K_{\bC} ,
\end{equation*}
which satisfies the conditions
\begin{align*}
&(\frac{d}{du} \otimes \id)(W_mK_{\bC}) \subset W_{m-2}K_{\bC} \\
&(\frac{d}{du} \otimes \id)(F^pK_{\bC}) \subset F^{p-1}K_{\bC}
\end{align*}
for every $m,p$.
Similarly, the morphism
\begin{equation*}
\frac{d}{du} \otimes \id:
\bQ[u] \otimes_{\bQ} \kos_{Y_{\ula}}(M_{Y_{\ula}})
\longrightarrow
\bQ[u] \otimes_{\bQ} \kos_{Y_{\ula}}(M_{Y_{\ula}})
\end{equation*}
induces a morphism of complexes
\begin{equation*}
\frac{d}{du} \otimes \id:
K_{\bQ} \longrightarrow K_{\bQ} ,
\end{equation*}
which satisfies the condition
\begin{equation*}
(\frac{d}{du} \otimes \id)(W_mK_{\bQ}) \subset W_{m-2}K_{\bQ}
\end{equation*}
for every $m$ as above.
Trivially,
these morphisms are compatible with
$\psi:K_{\bQ} \longrightarrow K_{\bC}$.
Thus the morphism
$d/du \otimes \id$
induces a morphism
\begin{equation}
\label{morphism N_K:eq}
\coh^q(Y,\frac{d}{du} \otimes \id):
(\coh^q(Y,K),W[q],F)
\longrightarrow
(\coh^q(Y,K),W[q+2],F[-1])
\end{equation}
for every $q$,
denoted by $N_K$ for short.
\end{para}

\begin{para}
\label{Steenbrink's A:eq}
We recall results in Steenbrink \cite{SteenbrinkLE}
and in Fujisawa-Nakayama \cite{Fujisawa-Nakayama},
which are analogues of results in Steenbrink \cite{SteenbrinkLHS}
from the viewpoint of log geometry.

We set
\begin{equation*}
A_{\bC}^p= \bigoplus_{r \ge 0} \omega_Y^{p+1}/W_r\omega_Y^{p+1}
\end{equation*}
for every $p$.
The morphism $\dlog t \wedge$ in \eqref{morphism dlog t wedge:eq}
induces a morphism
\begin{equation*}
\dlog t \wedge:
\omega_Y^{p+1}/W_r\omega_Y^{p+1}
\longrightarrow
\omega_Y^{p+2}/W_{r+1}\omega_Y^{p+2} \subset A_{\bC}^{p+1}
\end{equation*}
for every $p,r$.
By setting
\begin{equation*}
d=\bigoplus_{r \ge 0}(-d-\dlog t \wedge):
A_{\bC}^p \longrightarrow A_{\bC}^{p+1},
\end{equation*}
we obtain a complex of $\bC$-sheaves $A_{\bC}$ on $Y$.
The weight filtration $W$ on $A_{\bC}$ is given by
\begin{equation*}
W_mA_{\bC}^p=\bigoplus_{r \ge 0}W_{m+2r+1}\omega_Y^{p+1}/W_r\omega_Y^{p+1}
\end{equation*}
for every $m$,
and the Hodge filtration $F$ by
\begin{equation*}
F^nA_{\bC}^p=\bigoplus_{0 \le r \le p-n}\omega_Y^{p+1}/W_r\omega_Y^{p+1}
\end{equation*}
for every $n$.

We set
\begin{equation*}
A_{\bQ}^p
=\bigoplus_{r \ge 0}\kos_Y(M_Y)^{p+1}/W_r\kos_Y(M_Y)^{p+1}
\end{equation*}
for every $p$.
The morphism $t\wedge$ in \eqref{the morphism t wedge on Kos:eq}
induces a morphism
\begin{equation*}
t \wedge :
\kos_Y(M_Y)^{p+1}/W_r\kos_Y(M_Y)^{p+1}
\longrightarrow
\kos_Y(M_Y)^{p+2}/W_{r+1}\kos_Y(M_Y)^{p+2}
\subset A_{\bQ}^{p+1}
\end{equation*}
for every $p, r$.
By setting
\begin{equation*}
d=\bigoplus_{r \ge 0}
(-d-t \wedge):
A_{\bQ}^p \longrightarrow A_{\bQ}^{p+1},
\end{equation*}
we obtain a complex $A_{\bQ}$.
The weight filtration $W$ on $A_{\bQ}$
is defined by
\begin{equation*}
W_mA_{\bQ}^p=
\bigoplus_{r \ge 0}
W_{m+2r+1}\kos_Y(M_Y)^{p+1}/W_r\kos_Y(M_Y)^{p+1}
\end{equation*}
for every $m$.

We can easily check that
the morphism
$(2\pi\sqrt{-1})^{r+1}\psi_Y: \kos_Y(M_Y) \longrightarrow \omega_Y$
induces a morphism of filtered complexes
\begin{equation*}
\alpha=\bigoplus_{r \ge 0}(2\pi\sqrt{-1})^{r+1}\psi_Y:
(A_{\bQ}, W) \longrightarrow (A_{\bC}, W)
\end{equation*}
by using the commutative diagram
\eqref{commutative diagram for t wedge and dlog t wedge:eq}.
\end{para}

\begin{defn}
We set
\begin{equation*}
(A,W,F)=((A_{\bQ}, W), (A_{\bC}, W, F),\alpha)
\end{equation*}
and
\begin{equation*}
(\coh^q(Y,A),W,F)
=((\coh^q(Y,A_{\bQ}),W), (\coh^q(Y,A_{\bC}),W,F),\coh^q(Y,\alpha))
\end{equation*}
for every $q$.
\end{defn}

\begin{rmk}
The signs in the definition of the differentials above
are different from that
in \cite{SteenbrinkLE}, \cite{Fujisawa-Nakayama}.
Moreover, the $\bQ$-structure $A_{\bQ}$ above
is slightly different from the ones
in \cite{SteenbrinkLE}, \cite{Fujisawa-Nakayama}.
We can prove that the above $\bQ$-structure
induces the same $\bQ$-structure as the ones
in \cite{SteenbrinkLE}, \cite{Fujisawa-Nakayama}
on the cohomology groups.
However, we will not give the proof here
because we do not need this fact in this article.
What we need in this article is
Theorem \ref{Steenbrink's results} below.
\end{rmk}

\begin{para}
The composite of the morphism
\begin{equation*}
\dlog t \wedge:
\omega_{Y}^p \longrightarrow \omega_Y^{p+1}
\end{equation*}
and the projection
$\omega_Y^{p+1} \longrightarrow \omega_Y^{p+1}/W_0\omega_Y^{p+1}$
can be regarded as a morphism
\begin{equation*}
\theta:
\omega_Y^p \longrightarrow
\omega_Y^{p+1}/W_0\omega_Y^{p+1} \subset A_{\bC}^p
\end{equation*}
for every $p$.
It is easy to check that this morphism defines
a morphism of complexes,
which is denoted by
$\theta: \omega_Y \longrightarrow A_{\bC}$.
The morphism $\theta$
factors through the surjection
$\omega_Y \longrightarrow \omega_{Y/\ast}$.
Thus a morphism of complexes
$\theta_{/\ast}: \omega_{Y/\ast} \longrightarrow A_{\bC}$
is obtained.
\end{para}

\begin{thm}
\label{Steenbrink's results}
In the situation above,
the morphism
$\theta_{/\ast}: (\omega_{Y/\ast},F) \longrightarrow (A_{\bC},F)$
is a filtered quasi-isomorphism.
Therefore the morphism
\begin{equation*}
\coh^q(Y,\theta_{/\ast}):
\coh^q(Y,\omega_{Y/\ast}) \longrightarrow \coh^q(Y,A_{\bC})
\end{equation*}
is an isomorphism for every $q$,
under which the filtrations $F$ on both sides coincide.

If we assume the conditions
\eqref{compactness of Y:eq} and
\eqref{Kahler condition for Y:eq},
the data $A$
is a \cmh complex on $Y$.
In particular,
$(\coh^q(Y,A),W[q],F)$
is a mixed Hodge structure for every $q$.
\end{thm}
\begin{proof}
By Lemmas \ref{lemma on gr of omega},
\ref{lemma on gr of kos},
and \ref{commutativity of residue for omega and koszul},
the same proof as in \cite{SteenbrinkLE}, \cite{Fujisawa-Nakayama}
can work.
\end{proof}

\begin{rmk}
\label{isomorphisms for A}
We have an isomorphism in the derived category
\begin{equation}
\label{isomorphism for grWAQ:eq}
\begin{split}
\gr_m^WA_{\bQ}
&=\bigoplus_{r \ge \max(0,-m)} \gr_{m+2r+1}^W\kos_Y(M_Y)[1] \\
&\simeq
\bigoplus_{r \ge \max(0,-m)}\bigoplus_{\usi \in S_{m+2r+1}(\Lambda)}
\varepsilon(\usi) \otimes_{\bZ} \bQ_{Y_{\usi}}[-m-2r]
\end{split}
\end{equation}
and the isomorphism of complexes
\begin{equation}
\label{isomorphism for grWAC:eq}
\begin{split}
\gr_m^WA_{\bC}
&=\bigoplus_{r \ge \max(0,-m)}\gr_{m+2r+1}^W\omega_Y[1] \\
&\simeq
\bigoplus_{r \ge \max(0,-m)}\bigoplus_{\usi \in S_{m+2r+1}(\Lambda)}
\varepsilon(\usi) \otimes_{\bZ} \Omega_{Y_{\usi}}[-m-2r]
\end{split}
\end{equation}
as in Remark \ref{isomorphisms for somega and K}.
Under these identification, the morphism
\begin{equation*}
\gr_m^W\alpha: \gr_m^WA_{\bQ} \longrightarrow \gr_m^WA_{\bC}
\end{equation*}
is identified with the morphism
whose restriction on
$\varepsilon(\usi) \otimes_{\bZ} \bQ_{Y_{\usi}}[-m-2r]$
coincides with the morphism
$(2\pi\sqrt{-1})^{-m-r}\iota: \bQ \longrightarrow \bC$.
\end{rmk}

\begin{para}
As in Steenbrink \cite[(5.6)]{SteenbrinkLE}
a morphism of bifiltered complexes
\begin{equation*}
\nu_{\bC}:(A_{\bC},W,F) \longrightarrow (A_{\bC},W[2],F[-1])
\end{equation*}
is induced from the projection
$\omega_Y^{p+1}/W_r\omega_Y^{p+1}
\longrightarrow \omega_Y^{p+1}/W_{r+1}\omega_Y^{p+1}$
for $r \ge 0$.

Similarly, the projection
\begin{equation*}
\kos_Y(M_Y)^{p+1}/W_r\kos_Y(M_Y)^{p+1}
\longrightarrow
\kos_Y(M_Y)^{p+1}/W_{r+1}\kos_Y(M_Y)^{p+1}
\end{equation*}
induces a morphism of filtered complexes
$\nu_{\bQ}:(A_{\bQ},W) \longrightarrow (A_{\bQ},W[2])$
as above.
Since the diagram
\begin{equation*}
\begin{CD}
A_{\bQ} @>{\nu_{\bQ}}>> A_{\bQ} \\
@V{\alpha}VV @VV{\alpha}V \\
A_{\bC} @>>{(2\pi\sqrt{-1})\nu_{\bC}}> A_{\bC}
\end{CD}
\end{equation*}
is commutative,
the pair of the morphism
\begin{equation*}
\nu=(\nu_{\bQ},(2\pi\sqrt{-1}) \nu_{\bC}):
(A,W,F)
\longrightarrow
(A,W[2],F[-1])
\end{equation*}
induces a morphism
\begin{equation*}
N_A=\coh^q(Y,\nu):
(\coh^q(Y,A),W[q],F)
\longrightarrow
(\coh^q(Y,A),W[q+2],F[-1])
\end{equation*}
for every $q$.
\end{para}

Next, we will define morphisms of complexes
$A_{\bQ} \longrightarrow K_{\bQ}$
and $A_{\bC} \longrightarrow K_{\bC}$
which play an important role in the remainder of this article.

\begin{defn}
We set
\begin{equation*}
\res_Y^{\lambda}
=((e_{\lambda}\wedge)^{-1} \otimes \id) \res_Y^{\ula}:
\omega_Y \longrightarrow \omega_{Y_{\ula}}[-|\ula|]
\end{equation*}
for an element $\lambda \in \dprod\Lambda$.

Because the morphism
$\res_Y^{\lambda}$ sends $W_r\omega_Y$ to zero
for $r \le d(\lambda)$,
a morphism
\begin{equation*}
u^{[d(\lambda)-r]} \otimes \res_Y^{\lambda}:
\omega_Y^{p+1}/W_r\omega_Y^{p+1}
\longrightarrow
\bC[u] \otimes_{\bC} \omega_{Y_{\ula}}^{p-d(\lambda)}
\subset K_{\bC}^p
\end{equation*}
is induced,
where we set $u^{[n]}=u^n/n!$ for a non-negative integer $n$.
Then we set
\begin{equation*}
\varphi_{\bC}=\bigoplus_{r \ge 0}
\sum_{\substack{\lambda \in \dprod\Lambda \\
                d(\lambda) \ge r}}
(-1)^{d(\lambda)}(2\pi\sqrt{-1})^{d(\lambda)-r}
u^{[d(\lambda)-r]} \otimes \res_Y^{\lambda}:
A_{\bC}^p
\longrightarrow
K_{\bC}^p
\end{equation*}
for every $p$.
\end{defn}

\begin{lem}
The morphism $\varphi_{\bC}$ above
defines a morphism of complexes
preserving the filtrations $W$ and $F$.
\end{lem}
\begin{proof}
Since it is easy to check
that $\varphi_{\bC}$ preserves the filtrations $W$ and $F$,
it suffices to prove that
$\varphi_{\bC}$ is a morphism of complexes.

For any $\lambda \in \dprod\Lambda$,
the equality
\begin{equation}
\label{d and res lambda:eq}
\res^{\lambda}_Y d=(-1)^{d(\lambda)+1}d \res^{\lambda}_Y
\end{equation}
can be easily checked by definition.
Moreover, we can easily see the equality
\begin{equation}
\label{dlog t and res lambda:eq}
\res^{\lambda}_Y \dlog t\wedge
=(-1)^{d(\lambda)+1}(\dlog t\wedge) \res^{\lambda}_Y
+\sum_{i=0}^{d(\lambda)}(-1)^i
a^{\ast}_{\uln{\lambda_i},\ula}\res^{\lambda_i}_Y
\end{equation}
as morphisms from $\omega^p_Y$ to $\omega^{p-d(\lambda)}_{Y_{\ula}}$
by Lemma \ref{lemma on dlog t and residues}.

We write the differentials of the complexes $K_{\bC}$ and $A_{\bC}$
by $d_K$ and $d_A$
and the projection
\begin{equation}
K^p_{\bC}
=\bigoplus_{\lambda \in \dprod\Lambda}
\bC[u] \otimes_{\bC} \omega^{p-d(\lambda)}_{Y_{\ula}}
\longrightarrow
\bC[u] \otimes_{\bC} \omega^{p-d(\lambda)}_{Y_{\ula}}
\end{equation}
by $\pr_{\lambda}$ for a while.
In order to prove that $\varphi_{\bC}$ is a morphism of complexes,
it suffices to prove the equality
\begin{equation}
\pr_{\lambda}d_K\varphi_{\bC}=\pr_{\lambda}\varphi_{\bC}d_A:
A^p_{\bC}
\longrightarrow
\bC[u] \otimes_{\bC} \omega^{p+1-d(\lambda)}_{Y_{\ula}}
\label{equality for varphi compatible with d:eq}
\end{equation}
for any $\lambda \in \dprod\Lambda$.
By the definition of the differential of the \v{C}ech complex
in \ref{complexes from co-cubical objects:eq},
we have
\begin{equation}
\begin{split}
\pr_{\lambda}d_K
&=\sum_{i=0}^{d(\lambda)}(-1)^i
a^{\ast}_{\uln{\lambda_i},\ula}\pr_{\lambda_i}
+(-1)^{d(\lambda)}(\id \otimes d)\pr_{\lambda} \\
&\qquad \qquad \qquad \qquad
+(-1)^{d(\lambda)}((2\pi\sqrt{-1})^{-1}
\frac{d}{du} \otimes \dlog t \wedge)\pr_{\lambda},
\end{split}
\end{equation}
where $\lambda_i$ is the element $\prod\Lambda$
defined in \eqref{definition of lambdai}.
On the other hand,
we have
\begin{equation}
\begin{split}
\pr_{\lambda}\varphi_{\bC}&|_{\omega^{p+1}_Y/W_r\omega^{p+1}_Y} \\
&=
\begin{cases}
(-1)^{d(\lambda)}(2\pi\sqrt{-1})^{d(\lambda)-r}
u^{[d(\lambda)-r]} \otimes \res^{\lambda}_Y &\quad d(\lambda) \ge r \\
0 &\quad d(\lambda) < r
\end{cases}
\end{split}
\end{equation}
for a non-negative integer $r$.
Take $r \ge 0$ and $\lambda \in \dprod\Lambda$ with $d(\lambda)=k$.
If $k \ge r+1$, we have
\begin{equation}
\begin{split}
\pr_{\lambda}d_K\varphi_{\bC}|_{\omega^{p+1}_Y/W_r\omega^{p+1}_Y}
&=
\sum_{i=0}^k(-1)^{k+i+1}(2\pi\sqrt{-1})^{k-r-1}
u^{[k-r-1]}
\otimes (a^{\ast}_{\uln{\lambda_i},\ula}\res^{\lambda_i}_Y) \\
&\qquad \qquad
+(2\pi\sqrt{-1})^{k-r}
u^{[k-r]} \otimes (d \res^{\lambda}_Y) \\
&\qquad \qquad
+(2\pi\sqrt{-1})^{k-r-1}
u^{[k-r-1]} \otimes (\dlog t \wedge \res^{\lambda}_Y)
\end{split}
\end{equation}
by using $d(\lambda_i)=d(\lambda)-1=k-1$,
and
\begin{equation}
\begin{split}
\pr_{\lambda}\varphi_{\bC}d_A|_{\omega^{p+1}_Y/W_r\omega^{p+1}_Y}
&=
(-1)^k(2\pi\sqrt{-1})^{k-r}u^{[k-r]} \otimes (\res^{\lambda}_Y (-d)) \\
&\qquad \qquad
+(-1)^k(2\pi\sqrt{-1})^{k-r-1}
u^{[k-r-1]} \otimes (\res^{\lambda}_Y (-\dlog t\wedge)) \\
&=
(-1)^{k+1}(2\pi\sqrt{-1})^{k-r}u^{[k-r]} \otimes (\res^{\lambda}_Yd) \\
&\qquad \qquad
+(-1)^{k+1}(2\pi\sqrt{-1})^{k-r-1}
u^{[k-r-1]} \otimes (\res^{\lambda}_Y\dlog t\wedge)
\end{split}
\end{equation}
because
$(\dlog t\wedge)(\omega^{p+1}_Y/W_r\omega^{p+1}_Y)
\subset \omega^{p+2}_Y/W_{r+1}\omega^{p+2}_Y$.
Then we obtain
\eqref{equality for varphi compatible with d:eq}
by \eqref{d and res lambda:eq} and by \eqref{dlog t and res lambda:eq}.
If $k=r$, we have
\begin{equation}
\pr_{\lambda}d_K \varphi_{\bC}|_{\omega^{p+1}_Y/W_r\omega^{p+1}_Y}
=u^{[0]} \otimes (d \res^{\lambda}_Y)
\end{equation}
by  $d(\lambda_i)=k-1<r$ and by $(d/du)u^{[0]}=0$.
On the other hand,
\begin{equation}
\pr_{\lambda}\varphi_{\bC}d_A
=(-1)^{k+1}u^{[0]} \otimes (\res^{\lambda}_Y d)
\end{equation}
because
$(\dlog t\wedge)(\omega^{p+1}_Y/W_r\omega^{p+1}_Y)
\subset \omega^{p+2}_Y/W_{r+1}\omega^{p+2}_Y$ again.
Thus we obtain 
\eqref{equality for varphi compatible with d:eq}
by \eqref{d and res lambda:eq}.
If $k< r$, we have
\begin{equation}
\pr_{\lambda}d_K\varphi_{\bC}=\pr_{\lambda}\varphi_{\bC}d_A=0
\end{equation}
by definition.
Thus we obtain
\eqref{equality for varphi compatible with d:eq}
for any $\lambda \in \dprod\Lambda$.
\end{proof}

\begin{defn}
We set
\begin{equation*}
\res_Y^{\lambda}=((e_{\lambda} \wedge)^{-1} \otimes \id)\res_Y^{\ula}:
\kos_Y(M_Y)
\longrightarrow
\kos_{Y_{\ula}}(M_{Y_{\ula}})[-|\ula|]
\end{equation*}
for an element $\lambda \in \dprod\Lambda$, and
\begin{equation*}
\varphi_{\bQ}
=\bigoplus_{r \ge 0}
\sum_{\substack{\lambda \in \dprod\Lambda \\
                d(\lambda) \ge r}}
(-1)^{d(\lambda)}u^{[d(\lambda)-r]} \otimes \res_Y^{\lambda}:
A_{\bQ}^p
\longrightarrow
K_{\bQ}^p
\end{equation*}
for every $p$.
By Lemma \ref{lemma on t wedge and residues for koszul complexes},
$\varphi_{\bQ}$ is a morphism of complexes
as in the case of $\varphi_{\bC}$,
which also preserves the filtration $W$.
\end{defn}

\begin{lem}
The diagram
\begin{equation*}
\begin{CD}
A_{\bQ} @>{\varphi_{\bQ}}>> K_{\bQ} \\
@V{\alpha}VV @VV{\psi}V \\
A_{\bC} @>>{\varphi_{\bC}}> K_{\bC}
\end{CD}
\end{equation*}
is commutative.
\end{lem}
\begin{proof}
Lemma \ref{commutativity of residue for omega and koszul}
implies the conclusion.
\end{proof}

\begin{defn}
The pair $(\varphi_{\bQ}, \varphi_{\bC})$ is abbreviated as
\begin{equation*}
\varphi: A \longrightarrow K
\end{equation*}
and the pair $(\coh^q(Y,\varphi_{\bQ}),\coh^q(Y,\varphi_{\bC}))$ as
\begin{equation*}
\coh^q(Y,\varphi):
\coh^q(Y,A) \longrightarrow \coh^q(Y,K)
\end{equation*}
for simplicity.
\end{defn}

\begin{thm}
\label{comparison theorem for A and K}
If a log deformation $Y \longrightarrow \ast$
satisfies the conditions
\eqref{compactness of Y:eq} and
\eqref{Kahler condition for Y:eq}, then
\begin{equation*}
\coh^q(Y,\varphi): \coh^q(Y,A) \longrightarrow \coh^q(Y,K)
\end{equation*}
is an isomorphism of mixed Hodge structures
for every $q$.
\end{thm}
\begin{proof}
It is clear that $\coh^q(Y,\varphi)$
is a morphism of mixed Hodge structures.
Therefore it is sufficient to prove that
the morphism
$\coh^q(Y,\varphi_{\bC}): \coh^q(Y,A_{\bC}) \longrightarrow \coh^q(Y,K_{\bC})$
is an isomorphism.
Lemma \ref{lemma on dlog t and residues}
implies that the diagram
\begin{equation*}
\begin{CD}
\omega_{Y/\ast} @>{a^{\ast}_{/\ast}}>> \C(\omega_{Y_{\bullet}/\ast}) \\
@V{\theta_{/\ast}}VV @AA{\pi_{/\ast}}A \\
A_{\bC} @>>{\varphi_{\bC}}> K_{\bC}
\end{CD}
\end{equation*}
is commutative.
The morphisms
$\coh^q(Y,a^{\ast}_{/\ast})$, $\coh^q(Y,\pi_{/\ast})$
are isomorphisms for every $q$
by Theorem \ref{theorem for K}
and the morphism
$\coh^q(Y,\theta_{/\ast})$ is an isomorphism for every $q$
by Theorem \ref{Steenbrink's results}.
Therefore the morphism
$\coh^q(Y,\varphi_{\bC})$ is an isomorphism for every $q$.
\end{proof}

\begin{cor}
\label{comparison of E2}
In the situation above,
the morphism
$\varphi_{\bC}$
induces filtered isomorphism
\begin{equation*}
E_2^{p,q}(\varphi_{\bC}):
(E_2^{p,q}(A_{\bC},W),F)
\overset{\simeq}{\longrightarrow}
(E_2^{p,q}(K_{\bC},W),F)
\end{equation*}
for every $p,q$.
\end{cor}

\begin{prop}
\label{varphi, nu and d/du}
The diagrams
\begin{equation*}
\begin{CD}
A_{\bC} @>{\varphi_{\bC}}>> K_{\bC} \\
@V{(2\pi\sqrt{-1})\nu_{\bC}}VV @VV{\frac{d}{du} \otimes \id}V \\
A_{\bC} @>>{\varphi_{\bC}}> K_{\bC}
\end{CD}
\qquad \qquad \qquad
\begin{CD}
A_{\bQ} @>{\varphi_{\bQ}}>> K_{\bQ} \\
@V{\nu_{\bQ}}VV @VV{\frac{d}{du} \otimes \id}V \\
A_{\bQ} @>>{\varphi_{\bQ}}> K_{\bQ}
\end{CD}
\end{equation*}
are commutative.
Therefore the morphism $N_A$ and $N_K$
are identified under the isomorphism
$\coh^q(Y,\varphi)$
in Theorem $\ref{comparison theorem for A and K}$.
\end{prop}

\begin{defn}
We set
\begin{equation*}
\varphi_{\bC,0}=\pi_{\bC,0}\varphi_{\bC}:
(A^p_{\bC},W,F) \longrightarrow (\C(\omega_{Y_{\bullet}})^p,\delta W,F)
\end{equation*}
for every $p$.
It does not define a morphism of complexes.
However it induces a morphism of filtered complexes
\begin{equation*}
\gr_m^W\varphi_{\bC,0}:
(\gr_m^WA_{\bC},F) \longrightarrow (\gr_m^{\delta W}\C(\omega_{Y_{\bullet}}),F)
\end{equation*}
for every $m$.
Explicitly, we have
\begin{equation*}
\varphi_{\bC,0}
=(-1)^r\sum_{\lambda \in \Lambda^{r+1,\circ}}
((e_{\lambda}\wedge)^{-1} \otimes \id)\res_Y^{\lambda}:
\omega^{p+1}_Y/W_r\omega^{p+1}_Y
\longrightarrow
\bigoplus_{\lambda \in \prod\Lambda}\omega^{p-d(\lambda)}_{Y_{\ula}}
=\C(\omega_{Y_{\bullet}})^p
\end{equation*}
on the direct summand $\omega^{p+1}_Y/W_r\omega^{p+1}_Y$ for every $r \ge 0$.
\end{defn}

\section{Product}
\label{product}

\begin{para}
\label{product over C for ula}
Let $Y \longrightarrow \ast$ be a log deformation
satisfying the condition \eqref{finiteness assumption:eq}.
By sending
\begin{equation*}
(\bC[u] \otimes_{\bC} \omega^p_{Y_{\ula}})
\otimes_{\bC} (\bC[u] \otimes_{\bC} \omega^q_{Y_{\ula}})
\ni (P(u) \otimes \omega) \otimes (Q(u) \otimes \eta)
\end{equation*}
to
\begin{equation*}
P(u)Q(u) \otimes \omega \wedge \eta
\in \bC[u] \otimes_{\bC} \omega^{p+q}_{Y_{\ula}}
\end{equation*}
we obtain a morphism
\begin{equation*}
(\bC[u] \otimes_{\bC} \omega_{Y_{\ula}})
\otimes_{\bC} (\bC[u] \otimes_{\bC} \omega_{Y_{\ula}})
\longrightarrow
\bC[u] \otimes_{\bC} \omega_{Y_{\ula}},
\end{equation*}
denoted by the same letter $\wedge$.
In fact, it is easy to check that the morphism $\wedge$
is a morphism of complexes.
Thus the morphisms of complexes
\begin{align*}
&\wedge:
\omega_{Y_{\ula}} \otimes_{\bC} \omega_{Y_{\ula}}
\longrightarrow
\omega_{Y_{\ula}} \\
&\wedge:
\omega_{Y_{\ula}/\ast} \otimes_{\bC} \omega_{Y_{\ula}/\ast}
\longrightarrow
\omega_{Y_{\ula}/\ast} \\
&\wedge:
(\bC[u] \otimes_{\bC} \omega_{Y_{\ula}})
\otimes_{\bC} (\bC[u] \otimes_{\bC} \omega_{Y_{\ula}})
\longrightarrow
\bC[u] \otimes_{\bC} \omega_{Y_{\ula}}
\end{align*}
are obtained.
For these three morphisms,
the images of $W_a \otimes W_b$
(resp. $F^a \otimes F^b$)
are contained in $W_{a+b}$ (resp. $F^{a+b}$)
by definition.
These morphisms are compatible with the morphisms
$\pi_{\bC,\ula,0}, \pi_{\ula/\ast}$
and with the morphism induced from the inclusion
$Y_{\umu} \subset Y_{\ula}$ for $\ula \subset \umu$.
\end{para}

\begin{para}
\label{product over Q for ula}
Now we consider the case of the Koszul complex.
We take local sections
\begin{equation*}
f_1^{[i_1]}f_2^{[i_2]} \cdots f_k^{[i_k]} \otimes
x \in \kos_{Y_{\ula}}(M_{Y_{\ula}}^{\usi};n)^p, \quad
g_1^{[j_1]}g_2^{[j_2]} \cdots g_l^{[j_l]} \otimes
y \in \kos_{Y_{\ula}}(M_{Y_{\ula}}^{\usi};m)^q
\end{equation*}
respectively,
where
\begin{equation*}
f_1, f_2, \dots, f_k, g_1, g_2, \dots g_l \in \cO_{Y_{\ula}}, \quad
x \in \bigwedge^p(M_{Y_{\ula}}^{\usi})\gp, \quad
y \in \bigwedge^q(M_{Y_{\ula}}^{\usi})\gp
\end{equation*}
and $i_1, i_2, \dots i_k, j_1, j_2, \dots j_l$ are positive integers
with the conditions $i_1+i_2+ \dots +i_k=n-p$ and $j_1+j_2+ \dots +j_l=m-q$.
Then
\begin{equation*}
f_1^{[i_1]}f_2^{[i_2]} \cdots f_k^{[i_k]}
g_1^{[j_1]}g_2^{[j_2]} \cdots g_l^{[j_l]}
\otimes x \wedge y
\end{equation*}
is a local section of
\begin{equation*}
\Gamma_{n+m-p-q}(\cO_{Y_{\ula}})
\otimes \bigwedge^{p+q}(M_{Y_{\ula}}^{\usi})\gp
=\kos_{Y_{\ula}}(M_{Y_{\ula}}^{\usi};n+m)^{p+q}
\end{equation*}
by the definition \eqref{definition of kos:eq}.
We can check that this correspondence induces a morphism of complexes
\begin{equation}
\label{wedge product on kos:eq}
\wedge:
\kos_{Y_{\ula}}(M_{Y_{\ula}}^{\usi})
\otimes_{\bQ} \kos_{Y_{\ula}}(M_{Y_{\ula}}^{\usi})
\longrightarrow
\kos_{Y_{\ula}}(M_{Y_{\ula}}^{\usi}), 
\end{equation}
which sends $W_a \otimes W_b$ to $W_{a+b}$.
Similarly to the case of $\bC[u] \otimes_{\bC} \omega_{Y_{\ula}}$,
we define a morphism
\begin{equation*}
(\bQ[u] \otimes_{\bQ} \kos_{Y_{\ula}}(M_{Y_{\ula}}^{\usi}))
\otimes_{\bQ} (\bQ[u] \otimes_{\bQ} \kos_{Y_{\ula}}(M_{Y_{\ula}}^{\usi}))
\longrightarrow
\bQ[u] \otimes_{\bQ} \kos_{Y_{\ula}}(M_{Y_{\ula}}^{\usi})
\end{equation*}
by using the product of $\bQ[u]$
and the product $\wedge$ of $\kos_{Y_{\ula}}(M_{Y_{\ula}}^{\usi})$.
This morphism is also denoted by $\wedge$ by abuse of the language.

For the case of $\usi=\emptyset$,
we have the commutative diagram
\begin{equation*}
\begin{CD}
\bQ_{Y_{\ula}} \otimes_{\bQ} \bQ_{Y_{\ula}} @>>>
\kos_{Y_{\ula}}(\cO_{Y_{\ula}}^{\ast})
\otimes_{\bQ} \kos_{Y_{\ula}}(\cO_{Y_{\ula}}^{\ast}) \\
@VVV @VV{\wedge}V \\
\bQ_{Y_{\ula}} @>>> \kos_{Y_{\ula}}(\cO_{Y_{\ula}}^{\ast})
\end{CD}
\end{equation*}
for every $\ula \in S(\Lambda)$,
where the top horizontal arrow is the tensor product of the morphism
\eqref{morphism from Q to kos(O*):eq},
the bottom horizontal arrow is
the morphism \eqref{morphism from Q to kos(O*):eq} itself
and the left vertical arrow
is the canonical morphism
which sends $a \otimes b \in \bQ \otimes_{\bQ} \bQ$
to $ab \in \bQ$.
Because of this compatibility,
the canonical morphism on the left is denoted by
\begin{equation}
\label{wedge for Q:eq}
\wedge: \bQ_{Y_{\ula}} \otimes_{\bQ} \bQ_{Y_{\ula}}
\longrightarrow \bQ_{Y_{\ula}}
\end{equation}
in the remainder of this article.

For every $\ula$,
we can check the commutativity of the diagram
\begin{equation}
\label{compatibility of wedge for Q and C:eq}
\begin{CD}
\kos_{Y_{\ula}}(M_{Y_{\ula}}) \otimes_{\bQ} \kos_{Y_{\ula}}(M_{Y_{\ula}})
@>{\wedge}>> \kos_{Y_{\ula}}(M_{Y_{\ula}}) \\
@V{\psi_{(Y_{\ula},M_{Y_{\ula}})} \otimes \psi_{(Y_{\ula},M_{Y_{\ula}})}}VV
@VV{\psi_{(Y_{\ula},M_{Y_{\ula}})}}V \\
\omega_{Y_{\ula}} \otimes_{\bC} \omega_{Y_{\ula}}
@>>{\wedge}> \omega_{Y_{\ula}}
\end{CD}
\end{equation}
by direct computation.
\end{para}

\begin{para}
The morphisms $\wedge$ in \ref{product over C for ula}
and \ref{product over Q for ula}
define the morphisms
\begin{equation*}
\begin{split}
&\omega_{Y_{\bullet}}
\otimes_{\bC}
\omega_{Y_{\bullet}} \longrightarrow \omega_{Y_{\bullet}} \\
&\omega_{Y_{\bullet}/\ast} \otimes_{\bC} \omega_{Y_{\bullet}/\ast}
\longrightarrow \omega_{Y_{\bullet}/\ast} \\
&(\bC[u] \otimes_{\bC} \omega_{Y_{\bullet}})
\otimes_{\bC} (\bC[u] \otimes_{\bC} \omega_{Y_{\bullet}})
\longrightarrow \bC[u] \otimes_{\bC} \omega_{Y_{\bullet}} \\
&\kos_{Y_{\bullet}}(M_{Y_{\bullet}})
\otimes_{\bQ} \kos_{Y_{\bullet}}(M_{Y_{\bullet}})
\longrightarrow \kos_{Y_{\bullet}}(M_{Y_{\bullet}}) \\
&(\bQ[u] \otimes_{\bQ} \kos_{Y_{\bullet}}(M_{Y_{\bullet}}))
\otimes_{\bQ} (\bQ[u] \otimes_{\bQ} \kos_{Y_{\bullet}}(M_{Y_{\bullet}}))
\longrightarrow \bQ[u] \otimes_{\bQ} \kos_{Y_{\bullet}}(M_{Y_{\bullet}})
\end{split}
\end{equation*}
of co-cubical complexes over $Y_{\bullet}$,
where the left hand sides denote the co-cubical complexes defined
in \ref{diagonal cubical complex}.
Thus the morphisms of complexes
\begin{equation}
\label{pre-product morphisms:eq}
\begin{split}
&\C(\omega_{Y_{\bullet}} \otimes \omega_{Y_{\bullet}})
\longrightarrow \C(\omega_{Y_{\bullet}}) \\
&\C(\omega_{Y_{\bullet}/\ast} \otimes \omega_{Y_{\bullet}/\ast})
\longrightarrow \C(\omega_{Y_{\bullet}/\ast}) \\
&\C((\bC[u] \otimes_{\bC} \omega_{Y_{\bullet}})
\otimes_{\bC} (\bC[u] \otimes_{\bC} \omega_{Y_{\bullet}}))
\longrightarrow K_{\bC}=\C(\bC[u] \otimes_{\bC} \omega_{Y_{\bullet}}) \\
&\C(\kos_{Y_{\bullet}}(M_{Y_{\bullet}})
\otimes_{\bQ} \kos_{Y_{\bullet}}(M_{Y_{\bullet}}))
\longrightarrow \C(\kos_{Y_{\bullet}}(M_{Y_{\bullet}})) \\
&\C((\bQ[u] \otimes_{\bQ} \kos_{Y_{\bullet}}(M_{Y_{\bullet}}))
\otimes_{\bQ} (\bQ[u] \otimes_{\bQ} \kos_{Y_{\bullet}}(M_{Y_{\bullet}}))) \\
&\qquad \qquad \qquad \qquad \qquad \qquad
\longrightarrow
K_{\bQ}=\C(\bQ[u] \otimes_{\bQ} \kos_{Y_{\bullet}}(M_{Y_{\bullet}}))
\end{split}
\end{equation}
are induced.
\end{para}

\begin{defn}
We define the morphisms of complexes
\begin{equation*}
\begin{split}
&\Phi_{\bC,0}:
\C(\omega_{Y_{\bullet}}) \otimes_{\bC} \C(\omega_{Y_{\bullet}})
\longrightarrow \C(\omega_{Y_{\bullet}}) \\
&\Phi_{/\ast}:
\C(\omega_{Y_{\bullet}/\ast}) \otimes_{\bC} \C(\omega_{Y_{\bullet}/\ast})
\longrightarrow \C(\omega_{Y_{\bullet}/\ast}) \\
&\Phi_{\bC}:
K_{\bC} \otimes_{\bC} K_{\bC} \longrightarrow K_{\bC} \\
&\Phi_{\bQ,0}:
\C(\kos_{Y_{\bullet}}(M_{Y_{\bullet}}))
\otimes_{\bQ} \C(\kos_{Y_{\bullet}}(M_{Y_{\bullet}}))
\longrightarrow \C(\kos_{Y_{\bullet}}(M_{Y_{\bullet}})) \\
&\Phi_{\bQ}: K_{\bQ} \otimes_{\bQ} K_{\bQ} \longrightarrow K_{\bQ}
\end{split}
\end{equation*}
by composing the morphisms in \eqref{pre-product morphisms:eq}
with the morphisms $\tau$ in Definition \ref{definitin of tau}.
\end{defn}

\begin{para}
We have the commutative diagrams
\begin{equation*}
\begin{CD}
\C(\kos_{Y_{\bullet}}(M_{Y_{\bullet}}))
\otimes_{\bQ} \C(\kos_{Y_{\bullet}}(M_{Y_{\bullet}}))
@>{\Phi_{\bQ,0}}>> \C(\kos_{Y_{\bullet}}(M_{Y_{\bullet}})) \\
@V{\psi_0 \otimes \psi_0}VV @VV{\psi_0}V \\
\C(\omega_{Y_{\bullet}}) \otimes_{\bC} \C(\omega_{Y_{\bullet}})
@>>{\Phi_{\bC,0}}> \C(\omega_{Y_{\bullet}})
\end{CD}
\end{equation*}
and
\begin{equation}
\label{compatibility of PhiQ and PhiC:eq}
\begin{CD}
K_{\bQ} \otimes_{\bQ} K_{\bQ}
@>{\Phi_{\bQ}}>> K_{\bQ} \\
@V{\psi \otimes \psi}VV @VV{\psi}V \\
K_{\bC} \otimes_{\bC} K_{\bC}
@>>{\Phi_{\bC}}> K_{\bC}
\end{CD}
\end{equation}
from the commutativity of the diagram
\eqref{compatibility of wedge for Q and C:eq}.
Moreover, we also have the commutative diagram
\begin{equation*}
\begin{CD}
K_{\bC} \otimes_{\bC} K_{\bC} @>{\Phi_{\bC}}>> K_{\bC} \\
@V{\pi_{/\ast} \otimes \pi_{/\ast}}VV @VV{\pi_{/\ast}}V \\
\C(\omega_{Y_{\bullet}/\ast}) \otimes_{\bC} \C(\omega_{Y_{\bullet}/\ast})
@>>{\Phi_{/\ast}}> \C(\omega_{Y_{\bullet}/\ast})
\end{CD}
\end{equation*}
from the compatibility of $\wedge$
with the morphism $\pi_{\ula/\ast}$.
\end{para}

\begin{lem}
The diagram
\begin{equation*}
\begin{CD}
\omega_{Y/\ast} \otimes_{\bC} \omega_{Y/\ast} @>{\wedge}>> \omega_{Y/\ast} \\
@V{a_{/\ast}^{\ast} \otimes a_{/\ast}^{\ast}}VV @VV{a_{/\ast}^{\ast}}V \\
\C(\omega_{Y_{\bullet}/\ast}) \otimes_{\bC} \C(\omega_{Y_{\bullet}/\ast})
@>>{\Phi_{/\ast}}>
\C(\omega_{Y_{\bullet}/\ast})
\end{CD}
\end{equation*}
is commutative.
\end{lem}

\begin{lem}
\label{product and NK}
The equalities
\begin{equation*}
\begin{split}
&\Phi_{\bC}
((\frac{d}{du} \otimes \id) \otimes \id+\id \otimes (\frac{d}{du} \otimes \id))
=(\frac{d}{du} \otimes \id) \Phi_{\bC} \\
&\Phi_{\bQ}
((\frac{d}{du} \otimes \id)
\otimes \id+\id \otimes (\frac{d}{du} \otimes \id))
=(\frac{d}{du} \otimes \id) \Phi_{\bQ}
\end{split}
\end{equation*}
hold.
\end{lem}

\begin{para}
From Corollary \ref{filtration under tau},
we obtain
\begin{equation}
\label{filtrations under products:eq}
\begin{split}
&\Phi_{\bC,0}
(\delta W_a\C(\omega_{Y_{\bullet}}) \otimes \delta W_b\C(\omega_{Y_{\bullet}}))
\subset \delta W_{a+b}\C(\omega_{Y_{\bullet}}) \\
&\Phi_{\bC,0}
(F^a\C(\omega_{Y_{\bullet}}) \otimes F^b\C(\omega_{Y_{\bullet}}))
\subset F^{a+b}\C(\omega_{Y_{\bullet}}) \\
&\Phi_{/\ast}
(F^a\C(\omega_{Y_{\bullet}/\ast}) \otimes F^b\C(\omega_{Y_{\bullet}/\ast}))
\subset F^{a+b}\C(\omega_{Y_{\bullet}/\ast}) \\
&\Phi_{\bC}(W_aK_{\bC} \otimes W_bK_{\bC}) \subset W_{a+b}K_{\bC} \\
&\Phi_{\bC}(F^aK_{\bC} \otimes F^bK_{\bC}) \subset W_{a+b}K_{\bC} \\
&\Phi_{\bQ,0}
(\delta W_a\C(\kos_{Y_{\bullet}}(M_{Y_{\bullet}}))
\otimes \delta W_b\C(\kos_{Y_{\bullet}}(M_{Y_{\bullet}})))
\subset \delta W_{a+b}\C(\kos_{Y_{\bullet}}(M_{Y_{\bullet}})) \\
&\Phi_{\bQ}(W_aK_{\bQ} \otimes W_bK_{\bQ}) \subset W_{a+b}K_{\bQ}
\end{split}
\end{equation}
for every $a,b$.
Therefore we obtain morphisms
\begin{align*}
&\gr_{a,b}^W\Phi_{\bC}:
\gr_a^WK_{\bC} \otimes \gr_b^WK_{\bC}
\longrightarrow \gr_{a+b}^WK_{\bC} \\
&\gr_{a,b}^{\delta W}\Phi_{\bC,0}:
\gr_a^{\delta W}\C(\omega_{Y_{\bullet}})
\otimes \gr_b^{\delta W}\C(\omega_{Y_{\bullet}})
\longrightarrow
\gr_{a+b}^{\delta W}\C(\omega_{Y_{\bullet}})
\end{align*}
and so on, for every $a,b$.
\end{para}

\begin{defn}
We set
\begin{equation*}
\Psi_{\bC}=\Phi_{\bC} (\varphi_{\bC} \otimes \varphi_{\bC}):
A_{\bC} \otimes_{\bC} A_{\bC} \longrightarrow K_{\bC},
\end{equation*}
which is a morphism of complexes.
Moreover, we set
\begin{equation*}
\Psi_{\bC,0}=\Phi_{\bC,0} (\varphi_{\bC,0} \otimes \varphi_{\bC,0}):
A_{\bC}^p \otimes_{\bC} A_{\bC}^q \longrightarrow \C(\omega_{Y_{\bullet}})^{p+q}
\end{equation*}
for every $p,q$.
Although $\Psi_{\bC,0}$ dose not define a morphism of complexes,
it induces a morphism of complexes
\begin{equation*}
\gr_{a,b}^W\Psi_{\bC,0}:
\gr_a^WA_{\bC} \otimes_{\bC} \gr_b^WA_{\bC}
\longrightarrow \gr_{a+b}^{\delta W}\C(\omega_{Y_{\bullet}})
\end{equation*}
for every $a,b$ as before.
\end{defn}

\begin{para}
We can easily see the equality
\begin{equation}
\label{Psi and Psi0:eq}
\Psi_{\bC,0}=\pi_{\bC,0}\Psi_{\bC}
\end{equation}
by $\pi_{\bC,0}\Phi_{\bC}=\Phi_{\bC,0}(\pi_{\bC,0} \otimes \pi_{\bC,0})$.
\end{para}

\begin{prop}
The equality
\begin{equation*}
\Psi_{\bC} (\nu_{\bC} \otimes \id+\id \otimes \nu_{\bC})
=(\frac{d}{du} \otimes \id) \Psi_{\bC}
\end{equation*}
holds
\end{prop}
\begin{proof}
Proposition \ref{varphi, nu and d/du}
and Lemma \ref{product and NK}
yield the conclusion.
\end{proof}

\begin{para}
For the later use,
we describe the morphism
\begin{equation}
\label{the morphism Theta sdlog t Psi:eq}
\Theta_{\bC,0}(\lambda)[1]
\gr_0^{\delta W}\C(\dlog t \wedge)
\gr_{m,-m}^W\Psi_{\bC,0}:
\gr_m^WA_{\bC} \otimes_{\bC} \gr_{-m}^WA_{\bC}
\longrightarrow
\Omega_{Y_{\ula}}[-2k]
\end{equation}
for $\lambda \in \Lambda^{k+1,\circ}$
and for a non-negative integer $m$,
where $\Theta_{\bC,0}(\lambda)$
is the morphism \eqref{definition of Theta0:eq}.
\end{para}

\begin{lem}
\label{lemma on grWPsi}
Under the identification
\eqref{isomorphism for grWAC:eq},
the restriction
of the morphism \eqref{the morphism Theta sdlog t Psi:eq}
on the direct summand
\begin{equation*}
(\varepsilon(\usi) \otimes_{\bZ} \Omega_{Y_{\usi}}[-m-2r])
\otimes_{\bC}
(\varepsilon(\uta) \otimes_{\bZ} \Omega_{Y_{\uta}}[m-2s])
\end{equation*}
is the zero morphism except for the case of
$\usi=\uta=\ula, s=r+m$.

For the case of $\usi=\uta=\ula, s=r+m$,
the restriction
of the morphism \eqref{the morphism Theta sdlog t Psi:eq}
on the direct summand
\begin{equation*}
(\varepsilon(\ula) \otimes_{\bZ} \Omega_{Y_{\ula}}[-m-2r])
\otimes_{\bC}
(\varepsilon(\ula) \otimes_{\bZ} \Omega_{Y_{\ula}}[-m-2r])
\end{equation*}
coincides with the composite of the exchange isomorphism
\begin{equation*}
\begin{split}
(\varepsilon(\ula) \otimes_{\bZ} \Omega_{Y_{\ula}}[-m-2r])
\otimes_{\bC}
(&\varepsilon(\ula) \otimes_{\bZ} \Omega_{Y_{\ula}}[-m-2r]) \\
&\simeq
\varepsilon(\ula) \otimes_{\bZ} \varepsilon(\ula)
\otimes_{\bZ} \Omega_{Y_{\ula}}[-m-2r] \otimes_{\bC} \Omega_{Y_{\ula}}[-m-2r]
\end{split}
\end{equation*}
and the morphism
\begin{equation*}
\begin{split}
(-1)^r\vartheta(\ula) \otimes &\wedge[-m-2r,-m-2r] \\
&:\varepsilon(\ula) \otimes_{\bZ} \varepsilon(\ula)
\otimes_{\bZ}
\Omega_{Y_{\ula}}[-m-2r] \otimes_{\bC} \Omega_{Y_{\ula}}[-m-2r]
\longrightarrow
\Omega_{Y_{\ula}}[-2m-4r]
\end{split}
\end{equation*}
where $\vartheta(\ula)$ is the morphism
\eqref{theta for Lambda:eq},
and where $\wedge[-m-2r,-m-2r]$
denotes the morphism induced from $\wedge$
as in \eqref{morphism f[a,b]:eq}.
\end{lem}
\begin{proof}
On the direct summand
\begin{equation*}
(\varepsilon(\usi) \otimes_{\bZ} \Omega_{Y_{\usi}}[-m-2r])
\otimes_{\bC}
(\varepsilon(\uta) \otimes_{\bZ} \Omega_{Y_{\uta}}[m-2s])
\end{equation*}
of $\gr_m^WA_{\bC} \otimes_{\bC} \gr_{-m}^WA_{\bC}$,
the morphism $\gr_m^W\varphi_{\bC,0} \otimes \gr_{-m}^W\varphi_{\bC,0}$
coincides with the morphism
\begin{equation}
\label{morphism varphi0 otimes varphi0:eq}
(-1)^{r+s}\sum_{\mu,\nu}
((e_{\mu}\wedge)^{-1} \otimes \id)
\otimes ((e_{\nu}\wedge)^{-1} \otimes \id),
\end{equation}
where $\mu \in \Lambda^{r+1,\circ}, \nu \in \Lambda^{s+1,\circ}$
with $\umu \subset \usi, \unu \subset \uta$.
Therefore the image of the direct summand
\begin{equation*}
(\varepsilon(\usi) \otimes_{\bZ} \Omega_{Y_{\usi}}[-m-2r])
\otimes_{\bC}
(\varepsilon(\uta) \otimes_{\bZ} \Omega_{Y_{\uta}}[m-2s])
\end{equation*}
by the morphism
$\gr_m^W\varphi_{\bC,0} \otimes \gr_{-m}^W\varphi_{\bC,0}$
is contained in
\begin{equation*}
\bigoplus_{\mu,\nu}
(\varepsilon(\usi \setminus \umu)
\otimes_{\bZ} \Omega_{Y_{\usi}}[-m-2r])
\otimes_{\bC}
(\varepsilon(\uta \setminus \unu)
\otimes_{\bZ} \Omega_{Y_{\uta}}[m-2s])
\end{equation*}
for $\mu \in \Lambda^{r+1,\circ}$, $\nu \in \Lambda^{s+1,\circ}$
with $\umu \subset \usi, \unu \subset \uta$.

On the other hand,
Corollary \ref{dlog t on grWomega}
implies that the morphism
\begin{equation}
\label{Theta[1]dlogtPhi:eq}
\Theta_{\bC,0}(\lambda)
\gr_0^{\delta W}\C(\dlog t \wedge)
\gr_{m,-m}^{\delta W}\Phi_{\bC,0}
\end{equation}
is equal to zero
on the direct summand
\begin{equation*}
(\varepsilon(\usi \setminus \umu)
\otimes_{\bZ} \Omega_{Y_{\usi}}[-m-2r])
\otimes_{\bC}
(\varepsilon(\uta \setminus \unu)
\otimes_{\bZ} \Omega_{Y_{\uta}}[m-2s])
\end{equation*}
of
$\gr_{m+r}^W\omega_{Y_{\umu}}[-r]
\otimes_{\bC} \gr_{-m+s}^W\omega_{Y_{\unu}}[-s]$,
unless the following two conditions are satisfied:
\begin{mylist}
\itemno
\label{second condition for umu and unu}
$\mu=h_r(\lambda)$ and $\nu=t_r(\lambda)$
\itemno
\label{third condition for umu and unu}
$(\usi \setminus \umu) \cup (\uta \setminus \unu) \subset \ula$.
\end{mylist}
By the condition \eqref{second condition for umu and unu}
we have $k=r+s$.
Moreover, the condition \eqref{second condition for umu and unu}
implies $\ula=\umu \cup \unu \subset \usi \cup \uta$
because of the conditions $\umu \subset \usi, \unu \subset \uta$.
By \eqref{second condition for umu and unu}
and \eqref{third condition for umu and unu}
we have $\usi \subset \ula, \uta \subset \ula$.
Therefore $\ula =\usi \cup \uta$.
Now we have the equalities
\begin{equation*}
|\ula|
=|\usi|+|\uta|-|\usi \cap \uta|
=2(k+1)-|\usi \cap \uta|
\end{equation*}
from the equality $k=r+s$,
which imply
$|\usi \cap \uta|=k+1$.
Then $\ula=\usi \cap \uta=\usi \cup \uta$.
Thus we conclude that $\ula=\usi=\uta$, $s=r+m$ and $k=m+2r$.

On the direct summand
\begin{equation}
\label{direct summand of grWK otimes grWK:eq}
(\varepsilon(\ula \setminus \umu)
\otimes_{\bZ} \Omega_{Y_{\ula}}^{p-m-2r})
\otimes_{\bC}
(\varepsilon(\ula \setminus \unu)
\otimes_{\bZ} \Omega_{Y_{\ula}}^{q-m-2r})
\end{equation}
of
$\gr_{m+r}^W\omega_{Y_{\umu}}^{p-r}
\otimes_{\bC} \gr_r^W\omega_{Y_{\unu}}^{q-m-r}$
with the conditions $s=r+m$, $k=m+2r$ and
\eqref{second condition for umu and unu},
the morphism $\gr_{m,-m}^{\delta W}\Phi_{\bC,0}$
coincides with the composite of the exchange isomorphism
\begin{equation}
\label{exchange iso in question:eq}
\begin{split}
(\varepsilon(\ula \setminus \umu)
\otimes_{\bZ} \Omega_{Y_{\ula}}^{p-m-2r})
&\otimes_{\bC}
(\varepsilon(\ula \setminus \unu)
\otimes_{\bZ} \Omega_{Y_{\ula}}^{q-m-2r}) \\
&\simeq
\varepsilon(\ula \setminus \umu)
\otimes_{\bZ} \varepsilon(\ula \setminus \unu)
\otimes_{\bZ} \Omega_{Y_{\ula}}^{p-m-2r}
\otimes_{\bC} \Omega_{Y_{\ula}}^{q-m-2r}
\end{split}
\end{equation}
and the morphism
\begin{equation*}
(-1)^{(p-r)(r+m)+(p-m)|\ula \setminus \unu|}
\chi(\ula \setminus \umu, \ula \setminus \unu) \otimes \wedge
=(-1)^{r+pm}\chi(\ula \setminus \umu, \ula \setminus \unu) \otimes \wedge
\end{equation*}
by using $|\ula \setminus \unu|=r$,
where $\chi(\ula \setminus \umu, \ula \setminus \unu)$
is the morphism \eqref{morphism chi:eq}.
Because $\C(\dlog t \wedge)$ on the direct summand
$\omega_{Y_{\ula}}[-k]$ of $\C(\omega_{Y_{\bullet}})$
is the morphism $(-1)^k(\dlog t \wedge)$,
the morphism
\eqref{Theta[1]dlogtPhi:eq}
is equal to the composite of the isomorphism
\eqref{exchange iso in question:eq}
and the morphism
\begin{equation*}
(-1)^{m+r+pm}
(e_{\lambda} \wedge)^{-1} (e_{\lambda(r)} \wedge)
\chi(\ula \setminus \umu, \ula \setminus \unu)
\otimes \wedge
\end{equation*}
on the direct summand
\eqref{direct summand of grWK otimes grWK:eq}
by Corollary \ref{dlog t on grWomega},
by $k=m+2r$
and by $(\ula \setminus \umu) \cup (\ula \setminus \unu)=\uln{\lambda_r}$.
Here, the equality
\begin{equation*}
\vartheta(\ula)
=(e_{\lambda} \wedge)^{-1} (e_{\lambda(r)} \wedge)
\chi(\ula \setminus \umu,\ula \setminus \unu)
((e_{\mu} \wedge)^{-1} \otimes (e_{\nu}\wedge)^{-1}):
\varepsilon(\ula) \otimes_{\bZ} \varepsilon(\ula) \longrightarrow \bZ
\end{equation*}
can be easily checked.
Thus we obtain the conclusion
by considering \eqref{morphism varphi0 otimes varphi0:eq}
with $r+s=m+2r$
and by the sign convention \eqref{tensor and the sift on complexes:eq}.
\end{proof}

\begin{defn}
From the commutativity of the diagram
\eqref{compatibility of PhiQ and PhiC:eq},
a pair of the morphisms
$(\coh^{p,q}(Y,\Phi_{\bQ}),\coh^{p,q}(Y,\Phi_{\bC}))$
is denoted by
\begin{equation*}
\coh^{p,q}(Y,\Phi):
\coh^p(Y,K) \otimes \coh^q(Y,K)
\longrightarrow
\coh^{p+q}(Y,K)
\end{equation*}
for every $p,q$.
Sometimes,
$\coh^{p,q}(Y,\Phi)(x \otimes y)$
is simply denoted by $x \cup y$
for $x \in \coh^p(Y,K), y \in \coh^q(Y,K)$
if there is no danger of confusion.
\end{defn}

\begin{lem}
\label{associativity of cup}
We have
\begin{equation*}
(x \cup y) \cup z=x \cup (y \cup z)
\end{equation*}
for every $x \in \coh^p(Y,K)$, $y \in \coh^q(Y,K)$
and $z \in \coh^r(Y,K)$.
\end{lem}
\begin{proof}
By Lemma \ref{associativity of the product}.
\end{proof}

\begin{lem}
As for the filtration,
we have
\begin{align}
&W_a\coh^p(Y,K) \cup W_b\coh^q(Y,K)
\subset W_{a+b}\coh^{p+q}(Y,K)
\label{cup product and W:eq} \\
&F^a\coh^p(Y,K) \cup F^b\coh^q(Y,K)
\subset F^{a+b}\coh^{p+q}(Y,K)
\label{cup product and F:eq}
\end{align}
for every $a,b$.
In particular,
we obtain the morphism of mixed Hodge structures
\begin{equation*}
\cup :
(\coh^p(Y,K),W[p],F) \otimes (\coh^q(Y,K),W[q],F)
\longrightarrow
(\coh^{p+q}(Y,K),W[p+q],F)
\end{equation*}
for every $p,q$
if we assume the conditions
\eqref{compactness of Y:eq}
and \eqref{Kahler condition for Y:eq}.
\end{lem}
\begin{proof}
Easy by \eqref{filtrations under products:eq}.
\end{proof}

\begin{lem}
\label{weighted commutativity of cup}
Under the assumptions
\eqref{compactness of Y:eq}
and \eqref{Kahler condition for Y:eq},
we have
\begin{equation*}
y \cup x=(-1)^{pq}x \cup y
\end{equation*}
for $x \in \coh^p(Y,K)$ and for $y \in \coh^q(Y,K)$.
\end{lem}
\begin{proof}
The commutativity of the diagrams
\begin{equation*}
\begin{CD}
\coh^p(Y,K) \otimes \coh^q(Y,K) @>{\coh^{p,q}(Y,\Phi)}>> \coh^{p+q}(Y,K) \\
@V{\coh^p(Y,\pi_{/\ast}) \otimes \coh^q(Y,\pi_{/\ast})}VV
@VV{\coh^{p+q}(Y,\pi_{/\ast})}V \\
\coh^p(Y,\C(\omega_{Y_{\bullet/\ast}})) \otimes \coh^q(Y,\C(\omega_{Y_{\bullet}/\ast}))
@>>{\coh^{p,q}(Y,\Phi_{/\ast})}> \coh^{p+q}(Y,\C(\omega_{Y_{\bullet}/\ast}))
\end{CD}
\end{equation*}
and
\begin{equation*}
\begin{CD}
\coh^p(Y,\omega_{Y/\ast}) \otimes \coh^q(Y,\omega_{Y/\ast})
@>{\wedge}>> \coh^{p+q}(Y,\omega_{Y/\ast}) \\
@V{\coh^p(Y,a_{/\ast}^{\ast}) \otimes \coh^q(Y,a_{/\ast}^{\ast})}VV
@VV{\coh^{p+q}(Y,a_{/\ast}^{\ast})}V \\
\coh^p(Y,\C(\omega_{Y_{\bullet/\ast}})) \otimes \coh^q(Y,\C(\omega_{Y_{\bullet}/\ast}))
@>>{\coh^{p,q}(Y,\Phi_{/\ast})}> \coh^{p+q}(Y,\C(\omega_{Y_{\bullet}/\ast}))
\end{CD}
\end{equation*}
implies the conclusion,
by using the fact
that $\coh^p(Y,\pi_{/\ast})$ and $\coh^p(Y,a_{/\ast}^{\ast})$
are isomorphisms for all $p$.
\end{proof}

\begin{defn}
The wedge product on $\omega_Y$
induces a morphism
\begin{equation*}
\omega_Y^{p+1}/W_r\omega_Y^{p+1}
\otimes_{\bC} W_0\omega_Y^q
\longrightarrow
\omega_Y^{p+q+1}/W_r\omega_Y^{p+q+1}
\end{equation*}
because of the inclusion
$W_r\omega_Y^{p+1} \wedge W_0\omega_Y^q \subset W_r\omega_Y^{p+q+1}$.
We define a morphism
\begin{equation*}
\overline{\Psi}_{\bC}:
A^p_{\bC} \otimes_{\bC} W_0\omega_Y^q
\longrightarrow
A_{\bC}^{p+q}
\end{equation*}
by the direct sum of the morphism above.
It is easy to see that
$\overline{\Psi}_{\bC}$ defines a morphism of complexes
$A_{\bC} \otimes_{\bC} W_0\omega_Y \longrightarrow A_{\bC}$.

On the other hand,
the wedge product
\eqref{wedge product on kos:eq} on $\kos_Y(M_Y)$
induces a morphism
\begin{equation*}
\begin{split}
\kos_Y(M_Y)^{p+1}/W_r\kos_Y(M_Y)^{p+1}
&\otimes_{\bQ} W_0\kos_Y(M_Y)^q \\
&\longrightarrow
\kos_Y(M_Y)^{p+q+1}/W_r\kos_Y(M_Y)^{p+q+1}
\end{split}
\end{equation*}
for every $p,q,r$.
These morphisms define a morphism of complexes
\begin{equation*}
\overline{\Psi}_{\bQ}:
A_{\bQ} \otimes_{\bQ} W_0\kos_Y(M_Y)
\longrightarrow
A_{\bQ}
\end{equation*}
as above.
\end{defn}

\begin{para}
We can easily see that
the diagram
\begin{equation*}
\begin{CD}
A_{\bQ} \otimes_{\bQ} W_0\kos_Y(M_Y) @>{\overline{\Psi}_{\bQ}}>> A_{\bQ} \\
@V{\alpha \otimes \psi}VV @VV{\alpha}V \\
A_{\bC} \otimes_{\bC} W_0\omega_Y
@>>{\overline{\Psi}_{\bC}}> A_{\bC}
\end{CD}
\end{equation*}
is commutative.
As for the filtration,
we easily see
\begin{align*}
&\overline{\Psi}_{\bC}(W_mA_{\bC} \otimes_{\bC} W_0\omega_Y)
\subset
W_mA_{\bC} \\
&\overline{\Psi}_{\bC}(F^pA_{\bC} \otimes F^qW_0\omega_Y)
\subset
F^{p+q}A_{\bC} \\
&\overline{\Psi}_{\bQ}(W_mA_{\bQ} \otimes_{\bQ} W_0\kos_Y(M_Y))
\subset W_mA_{\bQ}
\end{align*}
for every $m,p,q$.
Therefore the morphisms
\begin{align*}
&\gr_m^W\overline{\Psi}_{\bC}:
\gr_m^WA_{\bC} \otimes_{\bC} W_0\omega_Y
\longrightarrow
\gr_m^WA_{\bC} \\
&\gr_m^W\overline{\Psi}_{\bQ}:
\gr_m^WA_{\bQ} \otimes_{\bQ} W_0\kos_Y(M_Y)
\longrightarrow
\gr_m^WA_{\bQ}
\end{align*}
are induced for every $m$.
The following two lemmas are easily proved.
We omit the proofs here.
\end{para}

\begin{lem}
\label{variant of product on gr}
Under the identification \eqref{isomorphism for grWAC:eq},
the morphism $\gr_m^W\overline{\Psi}_{\bC}$
coincides with the morphism
\begin{equation*}
\begin{split}
(\id \otimes \wedge[-m-2r,&0]) \cdot
(\id \otimes a_{\usi}^{\ast}) \\
&:
\varepsilon(\usi) \otimes_{\bZ} \Omega_{Y_{\usi}}[-m-2r]
\otimes_{\bC} W_0\omega_Y
\longrightarrow
\varepsilon(\usi) \otimes_{\bZ} \Omega_{Y_{\usi}}[-m-2r]
\end{split}
\end{equation*}
on the direct summand
$\varepsilon(\usi) \otimes_{\bZ} \Omega_{Y_{\usi}}[-m-2r]
\otimes_{\bC} W_0\omega_Y$,
where $a_{\usi}^{\ast}$
denotes the morphism induced by the inclusion
$a_{\usi}:Y_{\usi} \longrightarrow Y$.

Similarly, under the identification
\eqref{isomorphism for grWAQ:eq},
the restriction of the morphism
$\gr_m^W\overline{\Psi}_{\bQ}$ is identified with
the morphism
\begin{equation*}
\begin{split}
(\id \otimes \wedge[-m-2r,&0]) (\id \otimes a_{\usi}^{-1}) \\
&:\varepsilon(\usi) \otimes_{\bZ} \bQ_{Y_{\usi}}[-m-2r] \otimes_{\bQ} \bQ_Y
\longrightarrow
\varepsilon(\usi) \otimes_{\bZ} \bQ_{Y_{\usi}}[-m-2r]
\end{split}
\end{equation*}
on the direct summand
$\varepsilon(\usi) \otimes_{\bZ} \bQ_{Y_{\usi}}[-m-2r] \otimes_{\bQ} \bQ_Y$
of $\gr_m^WA_{\bQ} \otimes_{\bQ} W_0\kos_Y(M_Y)$,
where $\wedge$ is the morphism
\eqref{wedge for Q:eq}.
\end{lem}

\begin{lem}
\label{commutativity of product and variant of product}
The diagrams
\begin{equation*}
\begin{CD}
A_{\bQ} \otimes_{\bQ} W_0\kos_Y(M_Y) @>{\overline{\Psi}_{\bQ}}>> A_{\bQ} \\
@V{\varphi_{\bQ} \otimes a^{\ast}}VV @VV{\varphi_{\bQ}}V \\
K_{\bQ} \otimes_{\bQ} K_{\bQ} @>>{\Phi_{\bQ}}> K_{\bQ}
\end{CD}
\qquad \qquad
\begin{CD}
A_{\bC} \otimes_{\bC} W_0\omega_Y @>{\overline{\Psi}_{\bC}}>> A_{\bC} \\
@V{\varphi_{\bC} \otimes a^{\ast}}VV @VV{\varphi_{\bC}}V \\
K_{\bC} \otimes_{\bC} K_{\bC} @>>{\Phi_{\bC}}> K_{\bC}
\end{CD}
\end{equation*}
are commutative.
\end{lem}

\section{Trace morphism}
\label{section for trace morphism}

\begin{para}
Let $Y \longrightarrow \ast$ be a log deformation
satisfying the conditions
\eqref{compactness of Y:eq} and
\eqref{Kahler condition for Y:eq}.
In addition,
we assume
\begin{mylist}
\itemno
\label{pure dimensional condition for Y:eq}
all the irreducible components $Y_{\lambda}$
are of dimension $n$
\end{mylist}
in the remainder of this article.
\end{para}

\begin{lem}
\label{conditions for grWcoh(Y,K) non-zero}
The condition
$\gr_m^W\coh^q(Y,K_{\bC})\not=0$
implies the inequalities
$-q \le m \le q$ and $-2n+q \le m \le 2n-q$.
\end{lem}
\begin{proof}
The condition $\gr_m^W\coh^q(Y,K_{\bC})\not=0$ is equivalent to
$\gr_m^W\coh^q(Y,A_{\bC})\not=0$
by Theorem \ref{comparison theorem for A and K}.
The condition $\gr_m^W\coh^q(Y,A_{\bC})\not=0$
implies $E_1^{-m,q+m}(A_{\bC},W)\not=0$.
If this condition is the case,
then the identification
\begin{equation*}
\begin{split}
E_1^{-m,q+m}(A_{\bC},W)
&=\coh^q(Y,\gr_m^WA_{\bC}) \\
&\simeq
\bigoplus_{r \ge \max(0,-m)}\bigoplus_{\usi \in S_{m+2r+1}(\Lambda)}
\varepsilon(\usi) \otimes_{\bZ} \coh^{q-m-2r}(Y_{\usi},\Omega_{Y_{\usi}})
\end{split}
\end{equation*}
induced by \eqref{isomorphism for grWAC:eq}
gives us the inequalities $0 \le q-m-2r \le 2\dim Y_{\usi}=2(n-m-2r)$.
The conclusion can be easily obtained
from these inequalities.
\end{proof}

\begin{cor}
\label{condition for coh(Y,K) non-zero}
The condition
$\coh^q(Y,K)\not=0$ implies $0 \le q \le 2n$.
\end{cor}

\begin{lem}
\label{W on coh^{2n}}
We have
\begin{equation*}
W_{-1}\coh^{2n}(Y,K_{\bC}) =0,
W_0\coh^{2n}(Y,K_{\bC})=\coh^{2n}(Y,K_{\bC})
\end{equation*}
for the weight filtration $W$ on $\coh^{2n}(Y,K_{\bC})$.
On the other hand,
we have
\begin{equation*}
F^n\coh^{2n}(Y,K_{\bC})=\coh^{2n}(Y,K_{\bC}),
F^{n+1}\coh^{2n}(Y,K_{\bC})=0
\end{equation*}
for the Hodge filtration $F$.
\end{lem}
\begin{proof}
Lemma \ref{conditions for grWcoh(Y,K) non-zero}
shows that $\gr_m^W\coh^{2n}(Y,K_{\bC})=0$
for $m\not=0$,
Hence we obtain the conclusion for the filtration $W$.

We have
\begin{equation*}
(E_1^{0,2n}(A_{\bC},W),F)
\simeq
\bigoplus_{r \ge 0}\bigoplus_{\usi \in S_{2r+1}(\Lambda)}
(\varepsilon(\usi) \otimes_{\bZ} \coh^{2n-2r}(Y_{\usi},\Omega_{Y_{\usi}}), F[-r])
\end{equation*}
as in the proof of Lemma \ref{conditions for grWcoh(Y,K) non-zero}.
Since $\dim Y_{\usi}=n-2r$,
$\coh^{2n-2r}(Y_{\usi},\Omega_{Y_{\usi}})=0$ for $r>0$.
Therefore
\begin{equation*}
\gr_F^pE_1^{0,2n}(A_{\bC},W)
=\bigoplus_{\sigma \in \Lambda}
\varepsilon(\sigma) \otimes_{\bZ} \gr_F^p\coh^{2n}(Y_{\sigma},\Omega_{Y_{\sigma}})
\not=0
\end{equation*}
implies $p=n$.
Thus we conclude that
\begin{equation*}
\gr_F^pE_2^{0,2n}(A_{\bC},W) \simeq \gr_F^pE_2^{0,2n}(K_{\bC},W)\not=0
\end{equation*}
implies $p=n$ as desired.
\end{proof}

\begin{cor}
\label{exact sequence for Trace}
We have an exact sequence
\begin{equation}
\label{exact sequence for Trace:eq}
\begin{CD}
E_1^{-1,2n}(K_{\bC},W) @>>> Z_{\bC} @>>> \coh^{2n}(Y,K_{\bC}) @>>> 0
\end{CD}
\end{equation}
by setting
$Z_{\bC}
=\kernel(d_1:E_1^{0,2n}(K_{\bC},W) \longrightarrow E_1^{1,2n}(K_{\bC},W))$.
\end{cor}
\begin{proof}
We have
\begin{equation*}
\coh^{2n}(Y,K_{\bC})
\simeq \gr_0^W\coh^{2n}(Y,K_{\bC})
\simeq E_2^{0,2n}(K_{\bC},W)
\end{equation*}
by Corollary \ref{W on coh^{2n}}
and by $E_2$-degeneration of the spectral sequence
$E_r^{p,q}(K_{\bC},W)$.
\end{proof}

\begin{para}
For $\lambda \in \dprod\Lambda$,
we have the morphism
\begin{equation*}
\int_{Y_{\ula}}:
\coh^{2n-2d(\lambda)}(Y_{\ula},\Omega_{Y_{\ula}})
\longrightarrow
\bC
\end{equation*}
because $\dim Y_{\ula}=n-d(\lambda)$
for $\lambda \in \dprod\Lambda$.
On the other hand,
we have the morphisms
\begin{align*}
&\Theta_{\bC,0}(\lambda)[1]
\gr_0^{\delta W}\C(\dlog t \wedge):
\gr_0^{\delta W}\C(\omega_{Y_{\bullet}}) \longrightarrow
\Omega_{Y_{\ula}}[-2d(\lambda)] \\
&\Theta_{\bC}(\lambda)[1]
\gr_0^W\C(\id \otimes \dlog t \wedge):
\gr_0^WK_{\bC} \longrightarrow
\Omega_{Y_{\ula}}[-2d(\lambda)]
\end{align*}
for every $\lambda \in \dprod\Lambda$.
\end{para}

\begin{defn}
The morphisms
\begin{align*}
&\Theta_{\bC,0}:
E_1^{0,2n}(\C(\omega_{Y_{\bullet}}),\delta W)
=\coh^{2n}(Y, \gr_0^{\delta W}\C(\omega_{Y_{\bullet}}))
\longrightarrow \bC \\
&\Theta_{\bC}:
E_1^{0,2n}(K_{\bC},W)
=\coh^{2n}(Y, \gr_0^WK_{\bC})
\longrightarrow \bC
\end{align*}
are defined by
\begin{align*}
&\Theta_{\bC,0}
=\sum_{\lambda \in \dprod\Lambda}
\epsilon(d(\lambda))(|\ula|!)^{-1}
(2\pi\sqrt{-1})^{d(\lambda)}\int_{Y_{\ula}}
\coh^{2n}(Y,\Theta_{\bC,0}(\lambda)[1]
\gr_0^{\delta W}\C(\dlog t \wedge)) \\
&\Theta_{\bC}
=\sum_{\lambda \in \dprod\Lambda}
\epsilon(d(\lambda))(|\ula|!)^{-1}
(2\pi\sqrt{-1})^{d(\lambda)}\int_{Y_{\ula}}
\coh^{2n}(Y,\Theta_{\bC}(\ula)[1] \gr_0^W\C(\id \otimes \dlog t \wedge))
\end{align*}
respectively, where $\epsilon(a)=(-1)^{a(a-1)/2}$
as in \cite[(3.3)]{Guillen-NavarroAznarCI},
\cite[\RomI-14]{Peters-SteenbrinkMHS}.
\end{defn}

\begin{para}
We can easily check the equality
\begin{equation}
\label{equality for Theta and Theta0:eq}
\Theta_{\bC}
=\Theta_{\bC,0} \coh^{2n}(Y, \gr_0^W\pi_{\bC})
\end{equation}
by direct computation.
\end{para}

\begin{prop}
\label{proposition for the lifting of Theta}
We have $\Theta_{\bC} d_1=0$.
\end{prop}
\begin{proof}
The morphism
$d_1:E_1^{-1,2n}(K_{\bC},W) \longrightarrow E_1^{0,2n}(K_{\bC},W)$
is induced by the Gysin morphism
$\gamma_1(K_{\bC},W): \gr_1^WK_{\bC} \longrightarrow \gr_0^WK_{\bC}[1]$.
Therefore we have
\begin{align*}
\Theta_{\bC} d_1
&=\sum_{\lambda \in \dprod\Lambda}
\epsilon(d(\lambda))(|\ula|!)^{-1}
(2\pi\sqrt{-1})^{d(\lambda)}
\int_{Y_{\ula}} \\
&\qquad \qquad
\coh^{2n-1}(Y,\Theta_{\bC}(\lambda)[2]
\gr_0^W\C(\id \otimes \dlog t \wedge)[1]
\gamma_1(K_{\bC},W)) \\
&=\sum_{\lambda \in \dprod\Lambda}
\epsilon(d(\lambda))(|\ula|!)^{-1}
(2\pi\sqrt{-1})^{d(\lambda)}
\int_{Y_{\ula}} \\
&\qquad \qquad
\coh^{2n-1}(Y,
\Theta_{\bC,0}(\lambda)[2]
\gr_0^{\delta W}\C(\dlog t \wedge)[1] 
\gamma_1(\C(\omega_{Y_{\bullet}}),\delta W)
\gr_1^W\pi_{\bC,0}) \\
&=
\Theta_{\bC,0} d_1
\coh^{2n-1}(Y,\gr_1^W\pi_{\bC,0})
\end{align*}
by Lemma \ref{Theta and gamma for K and somega},
where $d_1$ in the last equality
stands for the morphism of $E_1$-terms
of the spectral sequence $E_r^{p,q}(\C(\omega_{Y_{\bullet}}),\delta W)$.
Therefore the following lemma implies the conclusion.
\end{proof}

\begin{lem}
For the morphism
$d_1:E_1^{-1,2n}(\C(\omega_{Y_{\bullet}}),\delta W)
\longrightarrow E_1^{0,2n}(\C(\omega_{Y_{\bullet}}),\delta W)$,
we have
$\Theta_{\bC,0} d_1=0$.
\end{lem}
\begin{proof}
Because we have
\begin{align*}
\gr_0^{\delta W}\C(\dlog t \wedge)[1] 
\gamma_1(\C(\omega_{Y_{\bullet}}),\delta W)
&=\gamma_2(\C(\omega_{Y_{\bullet}})[1],\delta W)
\gr_1^{\delta W}\C(\dlog t \wedge) \\
&=-\gamma_2(\C(\omega_{Y_{\bullet}}),\delta W)[1]
\gr_1^{\delta W}\C(\dlog t \wedge)
\end{align*}
from the functoriality of the Gysin morphism
and from the equality \eqref{gamma(K[l],W) and gamma(K,W)[l]:eq},
\begin{align*}
\Theta_{\bC,0} d_1
&=-\sum_{\lambda \in \dprod\Lambda}
\epsilon(d(\lambda))(|\ula|!)^{-1}
(2\pi\sqrt{-1})^{d(\lambda)}
\int_{Y_{\ula}} \\
&\qquad \qquad
\coh^{2n}(Y,
\Theta_{\bC,0}(\lambda)[1]
\gamma_2(\C(\omega_{Y_{\bullet}}),\delta W))
\coh^{2n-1}(Y, \gr_1^{\delta W}\C(\dlog t \wedge))
\end{align*}
is obtained.
Then it suffices to prove
that the morphism
\begin{equation}
\label{part of Theta cdot gamma:eq}
\sum_{\lambda \in \dprod\Lambda}
\epsilon(d(\lambda))(|\ula|!)^{-1}
(2\pi\sqrt{-1})^{d(\lambda)}
\int_{Y_{\ula}}
\coh^{2n}(Y,
\Theta_{\bC,0}(\lambda)[1]
\gamma_2(\C(\omega_{Y_{\bullet}}),\delta W))
\end{equation}
from $\coh^{2n}(Y,\gr_2^{\delta W} \C(\omega_{Y_{\bullet}}))$ to $\bC$
is the zero morphism.

For $\lambda \in \Lambda^{k+1,\circ}$,
Lemma \ref{gysin for co-cubical complex}
and Proposition \ref{gamma for a log deformation}
imply that the restriction of the morphism
\begin{equation}
\label{the morphism theta(lambda) cdot gamma:eq}
\Theta_{\bC,0}(\lambda)[1]
\gamma_2(\C(\omega_{Y_{\bullet}}),\delta W):
\gr_2^{\delta W}\C(\omega_{Y_{\bullet}})
\longrightarrow
\Omega_{Y_{\ula}}[-2k]
\end{equation}
is the zero morphism
on the direct summand
\begin{equation}
\varepsilon(\usi) \otimes_{\bZ} \Omega_{Y_{\umu \cup \usi}}[-2-2d(\mu)]
\label{direct summand of gr2omega for mu and usi:eq}
\end{equation}
of $\gr_2^{\delta W}\C(\omega_{Y_{\bullet}})$
for $\mu \in \dprod \Lambda$
and for $\usi \in S_{2+d(\mu)}(\Lambda)$
under the identification
\eqref{isomorphism for grmWsomega:eq},
unless one of the following conditions is satisfied
\begin{mylist}
\itemno
\label{condition lambda=mu:eq}
$\lambda=\mu$ and $\usi=\ula \cup \{\nu\}$
for some $\nu \in \Lambda \setminus \ula$
\itemno
\label{condition usi=ula:eq}
$\mu=\lambda_i$ for some $i=0,1, \dots, k$ and $\usi=\ula$.
\end{mylist}
For the case of \eqref{condition lambda=mu:eq},
the restriction of the morphism
\eqref{the morphism theta(lambda) cdot gamma:eq}
on the direct summand
\eqref{direct summand of gr2omega for mu and usi:eq}
coincides with the morphism
\begin{equation*}
(-1)^k((e_{\lambda} \wedge)^{-1} (e_{\nu} \wedge)^{-1})
\otimes \gamma(Y_{\ula},Y_{\ula \cup \{\nu\}})[-1-2k]
\end{equation*}
by \eqref{gysin for irreducible D:eq},
by Lemma \ref{gysin for co-cubical complex}
and by Proposition \ref{gamma for a log deformation}.
For the case of \eqref{condition usi=ula:eq},
the restriction of the morphism
\eqref{the morphism theta(lambda) cdot gamma:eq}
on the direct summand
\eqref{direct summand of gr2omega for mu and usi:eq}
coincides with the morphism
\begin{equation*}
(-1)^i(e_{\lambda} \wedge)^{-1} \otimes \id
=((e_{\lambda_i} \wedge)^{-1} (e_{\lambda(i)} \wedge)^{-1}) \otimes \id
\end{equation*}
by Lemma \ref{gysin for co-cubical complex}.

Hence, on the direct summand
\begin{equation}
\varepsilon(\ula \cup \{\nu\}) \otimes_{\bZ}
\coh^{2n-2-2k}(Y_{\ula \cup \{\nu\}},\Omega_{Y_{\ula \cup \{\nu\}}})
\end{equation}
for $\lambda \in \Lambda^{k+1,\circ}$
and for some $\nu \in \Lambda \setminus \ula$,
the restriction of the morphism
$\Theta_{\bC,0} d_1$
coincides with the morphism
\begin{equation*}
\begin{split}
\epsilon(k)&((k+1)!)^{-1}(2\pi\sqrt{-1})^k
(-1)^k((e_{\lambda} \wedge)^{-1} (e_{\nu} \wedge)^{-1})
\otimes \int_{Y_{\ula}}
\coh^{2n-2-2k}(Y_{\ula}, \gamma(Y_{\ula},Y_{\ula \cup \{\nu\}})) \\
&\qquad \qquad \qquad
+(k+2)\epsilon(k+1)((k+2)!)^{-1}(2\pi\sqrt{-1})^{k+1}
((e_{\lambda} \wedge)^{-1} (e_{\nu} \wedge)^{-1})
\otimes \int_{Y_{\ula \cup \{\nu\}}} \\
&=\epsilon(k+1)((k+1)!)^{-1}(2\pi\sqrt{-1})^k
((e_{\lambda} \wedge)^{-1} (e_{\nu} \wedge)^{-1}) \\
&\qquad \qquad \qquad
\otimes (\int_{Y_{\ula}}
\coh^{2n-2-2k}(Y_{\ula}, \gamma(Y_{\ula},Y_{\ula \cup \{\nu\}}))+
(2\pi\sqrt{-1})\int_{Y_{\ula \cup \{\nu\}}}) ,
\end{split}
\end{equation*}
which turns out to be zero
because of Proposition \ref{proposition for gysin and integration}.
\end{proof}

\begin{defn}
By Proposition \ref{proposition for the lifting of Theta}
and by Corollary \ref{exact sequence for Trace},
there exists the unique morphism
\begin{equation*}
\tr: \coh^{2n}(Y, K_{\bC}) \longrightarrow \bC ,
\end{equation*}
such that the diagram
\begin{equation*}
\begin{CD}
Z_{\bC} @>>> \coh^{2n}(Y,K_{\bC}) \\
@V{\Theta_{\bC}|_{Z_{\bC}}}VV @VV{\tr}V \\
\bC @= \bC
\end{CD}
\end{equation*}
commutes,
where the top horizontal arrow
stands for the morphism
in the exact sequence \eqref{exact sequence for Trace:eq}.
We call the morphism $\tr$ the trace morphism
for $Y \longrightarrow \ast$.
\end{defn}

\begin{prop}
\label{values of KQ by the trace morphism}
The morphism $\tr$ is defined over $\bQ$,
that is,
we have
\begin{equation*}
\tr(\coh^{2n}(Y,K_{\bQ})) \subset \bQ
\end{equation*}
under the identification
$\coh^{2n}(Y,K_{\bQ}) \otimes_{\bQ} \bC \simeq \coh^{2n}(Y,K_{\bC})$
by $\coh^{2n}(Y,\psi)$.
\end{prop}
\begin{proof}
It suffices to prove the inclusion
\begin{equation}
\label{Theta0 on skos:eq}
\Theta_{\bC,0}
\coh^{2n}(Y,\gr_0^{\delta W}\psi_0)
(E_1^{0,2n}(\C(\kos_{Y_{\bullet}}(M_{Y_{\bullet}})), \delta W))
\subset \bQ
\end{equation}
by the commutative diagram
\eqref{commutaive diagram for grWpi0:eq}
and by the equality
\eqref{equality for Theta and Theta0:eq}.
Since the morphism
$\gr_0^{\delta W}\psi_0$
is identified with the morphism
induced by the inclusion
\begin{equation*}
(2\pi\sqrt{-1})^{-d(\lambda)}\iota: \bQ \longrightarrow \bC
\end{equation*}
on the direct summand
$\varepsilon(\usi) \otimes_{\bZ} \Omega_{Y_{\ula \cup \usi}}[-2d(\lambda)]$
for $\lambda \in \dprod\Lambda$
under the identifications
\eqref{isomorphism for grmWskos:eq}
and \eqref{isomorphism for grmWsomega:eq},
we can easily obtain the inclusion
\eqref{Theta0 on skos:eq} as desired.
\end{proof}

\begin{defn}
We define a pairing
\begin{equation*}
Q_K:
\coh^q(Y,K_{\bC}) \otimes_{\bC} \coh^{2n-q}(Y,K_{\bC})
\longrightarrow \bC
\end{equation*}
by setting
\begin{equation*}
Q_K=\tr \cdot \coh^{q,2n-q}(Y,\Phi_{\bC}) ,
\end{equation*}
that is,
$Q_K(x \otimes y)=\tr(x \cup y)$
for $x \in \coh^q(Y,K)$ and $y \in \coh^{2n-q}(Y,K)$.
\end{defn}

\begin{lem}
We have the property
\begin{equation*}
Q_K(\coh^q(Y,K_{\bQ}) \otimes_{\bQ} \coh^{2n-q}(Y,K_{\bQ})) \subset \bQ
\end{equation*}
for all $q$.
\end{lem}
\begin{proof}
Easy from Proposition \ref{values of KQ by the trace morphism}.
\end{proof}

\begin{lem}
\label{weighted commutativity of Q}
We have
\begin{equation*}
Q_K(y \otimes x)=(-1)^qQ_K(x \otimes y)
\end{equation*}
for $x \in \coh^q(Y,K)$ and for $y \in \coh^{2n-q}(Y,K)$.
\end{lem}
\begin{proof}
Easy by Lemma \ref{weighted commutativity of cup}.
\end{proof}

\begin{lem}
\label{Q and N}
For the morphism of mixed Hodge structures
$N_K$ in \eqref{morphism N_K:eq},
the equality
\begin{equation*}
Q_K(N_K(x) \otimes y)+Q_K(x \otimes N_K(y))=0
\end{equation*}
holds for every $x \in \coh^q(Y,K)$, $y \in \coh^{2n-q}(Y,K)$.
\end{lem}
\begin{proof}
We have
\begin{equation*}
\coh^{p,q}(Y,\Phi)(N_K \otimes \id+\id \otimes N_K)
=N_K \cdot \coh^{p,q}(Y,\Phi)
\end{equation*}
by Lemma \ref{product and NK}.
On the other hand,
$N_K=0$ on $\coh^{2n}(Y,K)$
by Lemma \ref{W on coh^{2n}}.
\end{proof}

\begin{lem}
\label{Q and F}
We have
\begin{equation*}
Q_K(F^p\coh^q(Y,K) \otimes_{\bC} F^{n-p+1}\coh^{2n-q}(Y,K))=0
\end{equation*}
for all $p$.
\end{lem}
\begin{proof}
Easy by \eqref{cup product and F:eq}
and by the conclusion on $F$ in Lemma \ref{W on coh^{2n}}.
\end{proof}

\begin{lem}
\label{QK and W}
We have
\begin{equation*}
Q_K(W_a\coh^q(Y,K) \otimes W_b\coh^{2n-q}(Y,K))=0
\end{equation*}
if $a+b \le -1$.
\end{lem}
\begin{proof}
We easily obtain the conclusion
from the property \eqref{cup product and W:eq}
and from the fact $W_{-1}\coh^{2n}(Y,K)=0$ in Lemma \ref{W on coh^{2n}}.
\end{proof}

\begin{defn}
By the lemma above,
the morphism $Q_K$ induces the morphism
\begin{equation*}
\gr_m^W\coh^q(Y,K) \otimes \gr_{-m}^W\coh^{2n-q}(Y,K)
\longrightarrow
\bC
\end{equation*}
which is denoted by $\gr_{m,-m}^WQ_K$
for $m,q$.
\end{defn}

\begin{lem}
For $x \in \gr_m^W\coh^q(Y,K), y \in \gr_{-m}^W\coh^{2n-q}(Y,K)$,
we have
\begin{equation*}
\gr_{m,-m}^WQ_K(Cx \otimes Cy)=\gr_{m,-m}^WQ_K(x \otimes y),
\end{equation*}
where $C$'s denote the Weil operators on
$\gr_m^W\coh^q(Y,K)$ and $\gr_{-m}^W\coh^{2n-q}(Y,K)$
which are the Hodge structures of weight
$m+q$ and $2n-m-q$ respectively.
\end{lem}
\begin{proof}
The Weil operator on the Hodge structure $\gr_0^W\coh^{2n}(Y,K)$ of weight $n$
coincides with the identity
by Lemma \ref{W on coh^{2n}}.
Then we can easily see the conclusion
from the fact that $\cup$ is a morphism of mixed Hodge structures.
\end{proof}

\section{Main results}
\label{main results}

\begin{para}
\label{final assumption:eq}
Let $Y \longrightarrow \ast$ be a log deformation.
We assume
\begin{mylist}
\itemno
\label{projectivity of Y:eq}
$Y$ is projective
\end{mylist}
together with the condition
\eqref{pure dimensional condition for Y:eq}.
We fix an ample invertible sheaf $\cL$ on $Y$.
\end{para}

\begin{para}
The morphism
\begin{equation*}
\dlog : \cO_Y^{\ast}
\longrightarrow W_0\omega_Y[1]
\end{equation*}
is defined by sending a local section $f \in \cO_Y^{\ast}$
to $df/f \in \Omega_Y^1=W_0\omega_Y^1$.
We note that
the image of the morphism $\dlog$
is contained in $F^1\omega_Y[1]$.

On the other hand,
we have the morphism
\begin{equation*}
\cO_Y^{\ast}
\longrightarrow
\Gamma_{n-1}(\cO_Y) \otimes_{\bQ} \cO_Y^{\ast}
=\kos_Y(\cO_Y^{\ast};n)^1
\end{equation*}
which sends a local section $f \in \cO_Y^{\ast}$ to
$(n-1)!1^{[n-1]} \otimes f
\in \Gamma_{n-1}(\cO_Y) \otimes_{\bQ} \cO_Y^{\ast}$.
Then we obtain a morphism of complexes
\begin{equation*}
\cO_Y^{\ast}
\longrightarrow
\kos_Y(\cO_Y^{\ast})[1]=W_0\kos_Y(M_Y)[1]
\end{equation*}
denoted by $\dlog_{\bQ}$.
The diagram
\begin{equation}
\label{commutative diagram for dlog over Q and C:eq}
\begin{CD}
\cO_Y^{\ast} @>{\dlog_{\bQ}}>> W_0\kos_Y(M_Y)[1] \\
@| @VV{(2\pi\sqrt{-1})\psi_{(Y,M_Y)}[1]}V \\
\cO_Y^{\ast} @>>{\dlog}> W_0\omega_Y[1]
\end{CD}
\end{equation}
is commutative by definition.
\end{para}

\begin{defn}
We set
\begin{align*}
&c_{\bQ}(\cL)=\coh^1(Y,\dlog_{\bQ})([\cL]) \in \coh^2(Y,W_0\kos_Y(M_Y)) \\
&c_{\bC}(\cL)=\coh^1(Y,\dlog)([\cL]) \in \coh^2(Y,W_0\omega_Y)
\end{align*}
where $[\cL]$ denotes the isomorphism class of $\cL$
in $\coh^1(Y,\cO_Y^{\ast})$.
Moreover, we set
\begin{align*}
&c_{K,\bQ}(\cL)=\coh^2(Y,a^{\ast})(c_{\bQ}(\cL)) \in \coh^2(Y,K_{\bQ}) \\
&c_{K,\bC}(\cL)=\coh^2(Y,a^{\ast})(c_{\bC}(\cL)) \in \coh^2(Y,K_{\bC}) ,
\end{align*}
where $a^{\ast}$ denotes the composite of the inclusion
$W_0\kos_Y(M_Y) \longrightarrow \kos_Y(M_Y)$
(resp. $W_0\omega_Y \longrightarrow \omega_Y$)
and the morphism
$a^{\ast}:\kos_Y(M_Y) \longrightarrow K_{\bQ}$
(resp. $a^{\ast}:\omega_Y \longrightarrow K_{\bC}$).
\end{defn}

\begin{lem}
We have the following\,$:$
\begin{mylist}
\itemno
$c_{K,\bQ}(\cL) \in W_0\coh^2(Y,K_{\bQ})$
\itemno
$c_{K,\bC}(\cL) \in W_0\coh^2(Y,K_{\bC}) \cap F^1\coh^2(Y,K_{\bC})$
\itemno
$c_{K,\bC}(\cL)=(2\pi\sqrt{-1})\coh^2(Y,\psi)(c_{K,\bQ}(\cL))$
\end{mylist}
\end{lem}
\begin{proof}
The first two properties are easily seen
by the definition of $c_{K,\bQ}(\cL)$ and $c_{K,\bC}(\cL)$.
The equality
\begin{equation*}
c_{\bC}(\cL)=(2\pi\sqrt{-1})\coh^2(Y,\psi_{(Y,M_Y)})(c_{\bQ}(\cL)),
\end{equation*}
is obtained by the commutative diagram
\eqref{commutative diagram for dlog over Q and C:eq}.
Then the third equality follows
the commutative diagram
in \eqref{commutativity of a* for Q and C:eq}.
\end{proof}

\begin{defn}
For every $q$,
morphisms
\begin{equation*}
\begin{split}
&l_{K,\bQ}: \coh^q(Y,K_{\bQ}) \longrightarrow \coh^{q+2}(Y,K_{\bQ}) \\
&l_{K,\bC}: \coh^q(Y,K_{\bC}) \longrightarrow \coh^{q+2}(Y,K_{\bC}) \\
\end{split}
\end{equation*}
are defined by
\begin{equation*}
\begin{split}
&l_{K,\bQ}(x)=-c_{K,\bQ}(\cL) \cup x
=-\coh^{2,q}(Y,\Phi_{\bQ})(c_{K,\bQ}(\cL) \otimes x) \\
&l_{K,\bC}(y)=-c_{K,\bC}(\cL) \cup y
=-\coh^{2,q}(Y,\Phi_{\bC})(c_{K,\bC}(\cL) \otimes y)
\end{split}
\end{equation*}
for $x \in \coh^q(Y,K_{\bQ})$ and for $y \in \coh^q(Y,K_{\bC})$.
\end{defn}

\begin{lem}
The diagram
\begin{equation*}
\begin{CD}
\coh^q(Y,K_{\bQ}) @>{l_{K,\bQ}}>> \coh^{q+2}(Y,K_{\bQ}) \\
@V{\coh^q(Y,\psi)}VV @VV{(2\pi\sqrt{-1})\coh^{q+2}(Y,\psi)}V \\
\coh^q(Y,K_{\bC}) @>>{l_{K,\bC}}> \coh^{q+2}(Y,K_{\bC})
\end{CD}
\end{equation*}
is commutative.
Moreover we have
\begin{align*}
&l_{K,\bQ}(W_m\coh^q(Y,K_{\bQ})) \subset W_m\coh^{q+2}(Y,K_{\bQ}) \\
&l_{K,\bC}(W_m\coh^q(Y,K_{\bC})) \subset W_m\coh^{q+2}(Y,K_{\bC}) \\
&l_{K,\bC}(F^a\coh^q(Y,K_{\bC})) \subset F^{a+1}\coh^{q+2}(Y,K_{\bC})
\end{align*}
for every $a,m$.
\end{lem}
\begin{proof}
Easy from the properties of $\Phi_{\bQ}$ and $\Phi_{\bC}$.
\end{proof}

\begin{defn}
The lemma above
implies that
the pair of the morphism
\begin{equation*}
l_K=(l_{K,\bQ}, (2\pi\sqrt{-1})^{-1}l_{K,\bC})
\end{equation*}
defines a morphism of mixed Hodge structures
\begin{equation*}
l_K:(\coh^q(Y,K),W[q],F) \longrightarrow (\coh^{q+2}(Y,K),W[q],F[1])
\end{equation*}
for every $q$.
\end{defn}

\begin{lem}
\label{associativity of l}
We have
$x \cup l_Ky=l_Kx \cup y=l_K(x \cup y)$
for $x \in \coh^p(Y,K), y \in \coh^q(Y,K)$.
\end{lem}
\begin{proof}
Easy by Lemma \ref{associativity of cup},
by Lemma \ref{weighted commutativity of cup}
and by the fact $c_{K,\bC}(\cL) \in \coh^2(Y,K)$.
\end{proof}

\begin{lem}
\label{commutativity of N and l}
We have
$l_KN_K=N_Kl_K$
on $\coh^q(Y,K)$ for all $q$.
\end{lem}
\begin{proof}
Lemma \ref{product and NK}
tells us the equality
\begin{equation*}
N_K(c_{K,\bC}(\cL)) \cup x+c_{K,\bC}(\cL) \cup N_K(x)=N_K(c_{K,\bC}(\cL) \cup x)
\end{equation*}
for $x \in \coh^q(Y,K)$.
Because $a^{\ast}:\omega_Y \longrightarrow K_{\bC}$ factors through
the subcomplex $\C(\omega_{Y_{\bullet}})$ by definition,
we have $N_K(c_{K,\bC}(\cL))=0$.
Thus we have
\begin{equation*}
c_{K,\bC}(\cL) \cup N_K(x)=N_K(c_{K,\bC}(\cL) \cup x)
\end{equation*}
as desired.
\end{proof}

\begin{defn}
We set
\begin{equation*}
L_{\bQ}^{i,j}=\gr_{-i}^W\coh^{n+j}(Y,K_{\bQ}),
\qquad
L_{\bC}^{i,j}=\gr_{-i}^W\coh^{n+j}(Y,K_{\bC})
\end{equation*}
and
\begin{equation*}
L^{i,j}=(L_{\bQ}^{i,j},L_{\bC}^{i,j})
\end{equation*}
for every $i,j$.
Note that $L^{i,j}$ is a Hodge structure of weight $n+j-i$.
Then
\begin{equation*}
L_{\bQ}=\bigoplus_{i,j}L_{\bQ}^{i,j},
\qquad
L_{\bC}=\bigoplus_{i,j}L_{\bC}^{i,j}
\end{equation*}
is a pair of a finite dimensional $\bQ$-vector space
and a $\bC$-vector space
such that $L_{\bQ} \otimes_{\bQ} \bC \simeq L_{\bC}$.
Moreover, a morphism
\begin{equation*}
\lp\ \rp:
L \otimes_{\bC} L \longrightarrow \bC
\end{equation*}
is defined by
\begin{equation*}
\lp x \otimes y\rp=
\begin{cases}
\epsilon(j-n)\gr_{i,-i}^WQ_K(x \otimes y)
&\qquad \text{if $x \in L^{-i,-j}$, $y \in L^{i,j}$}\\
0
&\qquad \text{otherwise} ,
\end{cases}
\end{equation*}
which turns out to be a morphism defined over $\bQ$,
that is,
$\lp x \otimes y \rp \in \bQ$
if $x \otimes y \in L_{\bQ} \otimes_{\bQ} L_{\bQ}$.
\end{defn}

\begin{thm}
The data $(L,N_K,l_K,\lp\ \rp)$
is a bigraded polarized Hodge-Lefschetz module
in the sense of Guill\'en-Navarro Aznar
\cite[(4.1)-(4.3)]{Guillen-NavarroAznarCI}.
\end{thm}
\begin{proof}
We have
\begin{equation*}
\gr_{-i}^W\coh^{n+j}(Y,K_{\bC})
=E_2^{i,n+j-i}(K_{\bC},W)
\simeq
E_2^{i,n+j-i}(A_{\bC},W)
\end{equation*}
by Corollary \ref{comparison of E2}.
Then $L$ underlies a bigraded polarized Hodge-Lefschetz module
by \cite[(4.5)Th\'eor\`eme, (5.1)Th\'eor\`eme]{Guillen-NavarroAznarCI}.
Therefore it is sufficient
that our data $N_K,l_K,\lp\ \rp$
coincide with the data
$2\pi\sqrt{-1}N, (2\pi\sqrt{-1})^{-1}l,(2\pi\sqrt{-1})^n\psi$
used in \cite[(5.1)Th\'eor\`eme]{Guillen-NavarroAznarCI}.
(Our definition of the differential of $A$
in \ref{Steenbrink's A:eq} is different from
that in \cite[(2.4)]{Guillen-NavarroAznarCI}.
However, we can apply the results in \cite{Guillen-NavarroAznarCI}
because this difference only affects the sign of the morphism
$d_1:E_1^{p,q}(A,W) \longrightarrow E_1^{p+1,q}(A,W)$.)

The morphism $N_A:E_1^{p,q}(A_{\bC},W) \longrightarrow E_1^{p+2,q-2}(A_{\bC},W)$
coincides with the morphism $(2\pi\sqrt{-1})N$
because these two morphisms are induced by the same morphism
$(2\pi\sqrt{-1})\nu_{\bC}$.
Thus $N_K$ on $L$ coincides with $(2\pi\sqrt{-1})N$
under the identification above.
Lemma \ref{variant of product on gr}
and Lemma \ref{commutativity of product and variant of product}
tell us that
cup product with $c_{K,\bC}(\cL)$ on $L$
is induced by the usual cup product on $Y_{\usi}$
with $a_{\usi}^{\ast}(c_{K,\bC}(\cL))$.
Since $a_{\usi}^{\ast}(c_{K,\bC}(\cL))$
coincides with $c_1'(a_{\usi}^{\ast}\cL)$
in Deligne \cite[(2.2.4.1)]{DeligneII},
the morphism $l_K$ on $L$ coincides with $(2\pi\sqrt{-1})^{-1}l$.

We have
\begin{equation*}
\Theta_{\bC}\coh^{n-j,n+j}(Y,\gr_{i,-i}^W\Psi_{\bC})
=\Theta_{\bC,0}\coh^{n-j,n+j}(Y,\gr_{i,-i}^W\Psi_{\bC,0})
\end{equation*}
by the equality \eqref{Psi and Psi0:eq}.
Therefore we have
\begin{equation*}
\begin{split}
\Theta_{\bC}\coh^{n-j,n+j}&(Y,\gr_{i,-i}^W\Psi_{\bC}) \\
&=\sum_{\lambda \in \dprod\Lambda}
\epsilon(d(\lambda))(|\ula|!)^{-1}
(2\pi\sqrt{-1})^{d(\lambda)}\int_{Y_{\ula}} \\
&\qquad \qquad
\coh^{n-j,n+j}(Y,\Theta_{\bC,0}(\lambda)[1]
\gr_0^{\delta W}\C(\dlog t \wedge)
\gr_{i,-i}^W\Psi_{\bC,0})
\end{split}
\end{equation*}
by definition.
On the direct summand
\begin{equation*}
\begin{split}
(\varepsilon(\ula)
\otimes_{\bZ} \coh^{n-j-i-2r}&(Y_{\ula},\Omega_{Y_{\ula}}))
\otimes_{\bC}
(\varepsilon(\ula)
\otimes_{\bZ} \coh^{n+j-i-2r}(Y_{\ula},\Omega_{Y_{\ula}})) \\
&\simeq
\varepsilon(\ula)
\otimes_{\bZ} \varepsilon(\ula)
\otimes_{\bZ}
\coh^{n-j-i-2r}(Y_{\ula},\Omega_{Y_{\ula}})
\otimes_{\bC} \coh^{n+j-i-2r}(Y_{\ula},\Omega_{Y_{\ula}})
\end{split}
\end{equation*}
of $\coh^{n-j}(Y,\gr_i^WA_{\bC}) \otimes_{\bC} \coh^{n+j}(Y, \gr_{-i}^WA_{\bC})$
for $\ula \in S_{i+2r+1}(\Lambda)$,
we have
\begin{equation*}
\begin{split}
\Theta_{\bC} \coh^{n-j,n+j}&(Y,\gr_{i,-i}^W\Psi_{\bC}) \\
&=(-1)^{r+(n-j)i}|\ula|!\epsilon(|\ula|-1)(|\ula|!)^{-1}(2\pi\sqrt{-1})^{|\ula|-1}
\int_{Y_{\ula}} (\vartheta(\ula) \otimes \cup) \\
&=(-1)^{(n-j)i}\epsilon(i)(2\pi\sqrt{-1})^{i+2r}
\int_{Y_{\ula}} (\vartheta(\ula) \otimes \cup)
\end{split}
\end{equation*}
by Lemma \ref{lemma on grWPsi},
where $\cup$ in the equalities above denotes
the product of the usual de Rham cohomology of $Y_{\ula}$.
On the other direct summands,
we have
\begin{equation*}
\Theta_{\bC} \coh^{n-j,n+j}(Y,\gr_{i,-i}^W\Psi_{\bC})=0
\end{equation*}
by Lemma \ref{lemma on grWPsi} again.
Identifying $\varepsilon(\ula) \otimes_{\bZ} \varepsilon(\ula)$ and $\bZ$
by the canonical isomorphism $\vartheta(\ula)$,
we conclude that
the pairing $\lp\ \rp$ coincides with
$(2\pi\sqrt{-1})^n\psi$
by using the equality
$(-1)^{(n-j)i}\epsilon(j-n)\epsilon(i)=\epsilon(i+j-n)$.
\end{proof}

\begin{cor}
\label{l and N on coh}
For every $i \ge 0$ and for every $q$
\begin{equation*}
N_K^i:(\gr_i^W\coh^q(Y,K),F)
\longrightarrow
(\gr_{-i}^W\coh^q(Y,K),F[-i])
\end{equation*}
are isomorphisms of Hodge structures of weight $i+q$.
Moreover,
\begin{equation*}
l_K^j:(\coh^{n-j}(Y,K),W[n-j],F)
\longrightarrow
(\coh^{n+j}(Y,K),W[n-j],F[j])
\end{equation*}
is an isomorphism of mixed Hodge structures
for every $j \ge 0$.
\end{cor}

\begin{defn}
We set
\begin{equation*}
\coh^j(Y,K)_{\prim}
=\kernel(l_K^{n-j+1}:\coh^j(Y,K) \longrightarrow \coh^{2n-j+2}(Y,K))
\end{equation*}
for $j \le n$.
Then a morphism
$N_K:\coh^j(Y,K)_{\prim} \longrightarrow \coh^j(Y,K)_{\prim}$
is induced by Lemma \ref{commutativity of N and l}.
Moreover, we define a pairing
\begin{equation*}
S_{j,\prim}:\coh^j(Y,K)_{\prim} \otimes \coh^j(Y,K)_{\prim}
\longrightarrow \bC
\end{equation*}
by
\begin{equation*}
S_{j,\prim}(x \otimes y)=\epsilon(j)Q_K(x \otimes l_K^{n-j}y)
\end{equation*}
for $x,y \in \coh^j(Y,K)_{\prim}$.
\end{defn}

\begin{thm}
For $j \le n$, the data
\begin{equation*}
(\coh^j(Y,K)_{\prim},W[j],F,N_K,S_{j,\prim})
\end{equation*}
is a polarized mixed Hodge structure over $\bQ$
in the sense of Cattani-Kaplan-Schmid
\cite[Definition (2.26)]{CKS}.
\end{thm}
\begin{proof}
Lemma \ref{conditions for grWcoh(Y,K) non-zero}
implies $W_{-j-1}\coh^j(Y,K)=0$
and $W_j\coh^j(Y,K)=\coh^j(Y,K)$.
Therefore $N_K^{j+1}=0$.
From \eqref{morphism N_K:eq}, we have
$N_K(F^p\coh^j(Y,K)_{\prim}) \subset F^{p-1}\coh^j(Y,K)_{\prim}$
for all $p$.

Since
\begin{equation*}
l_K^{n-j+1}:(\coh^j(Y,K),W[j],F) \longrightarrow (\coh^{2n-j+2}(Y,K),W[j],F[n-j+1])
\end{equation*}
is a morphism of mixed Hodge structures,
$(\coh^j(Y,K)_{\prim},W[j],F)$ is a mixed Hodge structure.
Moreover, the sequences
\begin{equation*}
\begin{CD}
0 @>>> \gr_m^W\coh^j(Y,K)_{\prim}
  @>>> \gr_m^W\coh^j(Y,K)
  @>{l_K^{n-j+1}}>> \gr_m^W\coh^{2n-j+2}(Y,K)
\end{CD}
\end{equation*}
are exact for all $m$.
Therefore, we have
\begin{equation}
\label{primitive part and L:eq}
\gr_m^W\coh^j(Y,K)_{\prim}
=\kernel(l_K^{n-j+1}:L^{-m,j-n} \longrightarrow L^{-m,n-j+2})
\end{equation}
for all $m$.
The commutativity of $N_K$ and $l_K$ induces the commutative diagram
\begin{equation*}
\begin{CD}
0 @>>> \gr_i^W\coh^j(Y,K)_{\prim}
  @>>> \gr_i^W\coh^j(Y,K)
  @>{l_K^{n-j+1}}>> \gr_i^W\coh^{2n-j+2}(Y,K) \\
@. @V{N_K^i}VV @VV{N_K^i}V @VV{N_K^i}V \\
0 @>>> \gr_{-i}^W\coh^j(Y,K)_{\prim}
  @>>> \gr_{-i}^W\coh^j(Y,K)
  @>>{l_K^{n-j+1}}> \gr_{-i}^W\coh^{2n-j+2}(Y,K) ,
\end{CD}
\end{equation*}
which shows that the morphism
\begin{equation*}
N_K^i:
\gr_i^W\coh^j(Y,K)_{\prim}
\longrightarrow
\gr_{-i}^W\coh^j(Y,K)_{\prim}
\end{equation*}
is isomorphism for $i \ge 0$.
Therefore $W[j]=W(N_K)[j]$ as desired.

Take elements $x, y \in \coh^j(Y,K)_{\prim}$.
We have
\begin{equation*}
\begin{split}
S_{j,\prim}(y \otimes x)
&=\epsilon(j)Q_K(y \otimes l_K^{n-j}x)
=\epsilon(j)Q_K(l_K^{n-j}y \otimes x) \\
&=\epsilon(j)(-1)^jQ_K(x \otimes l_K^{n-j}y)
=(-1)^jS_{j,\prim}(x \otimes y)
\end{split}
\end{equation*}
by Lemma \ref{weighted commutativity of Q}
and by Lemma \ref{associativity of l}.
Moreover, we can easily check
\begin{equation*}
S_{j,\prim}(N_Kx \otimes y)+S_{j,\prim}(x \otimes N_Ky)=0
\end{equation*}
by the commutativity of $N_K$ and $l_K$
and by Lemma \ref{Q and N}.

If $x \in F^p\coh^j(Y,K), y \in F^{j-p+1}\coh^j(Y,K)$,
Lemma \ref{Q and F} implies
$Q_K(x \otimes l_K^{n-j}y)=0$
because $l_K^{n-j}y \in F^{n-p+1}\coh^{2n-p}(Y,K)$.
Thus we obtain
\begin{equation*}
S_{j,\prim}(F^p\coh^j(Y,K)_{\prim} \otimes F^{j-p+1}\coh^j(Y,K)_{\prim})=0
\end{equation*}
for all $p$.

Now we set
\begin{equation*}
P_i=
\kernel(N_K^{i+1}:\gr_i^W\coh^j(Y,K)_{\prim}
\longrightarrow
\gr_{-i-2}^W\coh^j(Y,K)_{\prim})
\end{equation*}
for every $i \ge 0$,
which is a Hodge structure of weight $i+j$.
Then we have
\begin{equation*}
P_i=L^{-i,j-n} \cap \kernel(N_K^{i+1}) \cap \kernel(l_K^{n-j+1})
\end{equation*}
for every $i \ge 0$
by \eqref{primitive part and L:eq}.

By the definition of polarization of bigraded Hodge-Lefschetz module
in \cite[(4.3)]{Guillen-NavarroAznarCI},
we have
\begin{equation*}
\lp x \otimes CN_K^il_K^j\overline{x} \rp > 0
\end{equation*}
for $x \in L^{-i,-j} \cap \kernel(N_K^{i+1}) \cap \kernel(l_K^{j+1})$
with $x \not= 0$,
where $C$ denotes the Weil operator on the Hodge structure $L^{i,j}$.
For $x \in P_i \subset L^{-i,j-n}$ with $x \not= 0$,
we have
\begin{equation*}
\begin{split}
S_{j,\prim}(Cx \otimes N_K^i\overline{x})
&=\epsilon(j)Q_K(Cx \otimes l_K^{n-j}N_K^i\overline{x}) \\
&=\epsilon(j)(-1)^{i+j}Q_K(x \otimes Cl_K^{n-j}N_K^i\overline{x}) \\
&=(-1)^j\epsilon(j)\epsilon(-j)\lp x \otimes CN_K^il_K^{n-j}\overline{x} \rp \\
&=\epsilon(j+1)\epsilon(j+1)\lp x \otimes CN_K^il_K^{n-j}\overline{x} \rp \\
&=\lp x \otimes CN_K^il_K^{n-j}\overline{x} \rp > 0
\end{split}
\end{equation*}
as desired.
\end{proof}

\begin{para}
Now the standard procedure
(e.g. \cite[Example 2.10]{Peters-SteenbrinkMHS})
gives us a polarization of the mixed Hodge structure
$(\coh^q(Y,K),W[q],F)$ as follows.

We have a direct sum decomposition
\begin{align*}
&(\coh^q(Y,K),W[q],F)
=\bigoplus_{j \ge 0}l_K^j(\coh^{q-2j}(Y,K)_{\prim},W[q],F[-j])
\quad \text{for $q \le n$} \\
&(\coh^q(Y,K),W[q],F)
=\bigoplus_{j \ge 0}l_K^{q-n+j}(\coh^{2n-q-2j}(Y,K)_{\prim},W[q],F[n-q-j])
\quad \text{for $q \ge n$}
\end{align*}
by Corollary \ref{l and N on coh}
and by the fact that $l_K$ is a morphism of mixed Hodge structures.

For the case of $q \le n$,
we define
\begin{equation*}
S_q(x \otimes y)
=\sum_{j \ge 0}S_{q-2j,\prim}(x_j \otimes y_j) ,
\end{equation*}
where
$x=\sum_{j \ge 0}l_K^jx_j$
and $y=\sum_{j \ge 0}l_K^jy_j$
for some $x_j, y_j \in \coh^{q-2j}(Y,K)_{\prim}$.

For the case of $q \ge n$,
we set
\begin{equation*}
S_q(x \otimes y)
=\sum_{j \ge 0}S_{2n-q-2j,\prim}(x_j,y_j) ,
\end{equation*}
where
$x=\sum_{j \ge 0}l_K^{q-n+j}x_j$
and $y=\sum_{j \ge 0}l_K^{q-n+j}y_j$
for some $x_j, y_j \in \coh^{2n-q-2j}(Y,K)_{\prim}$.
\end{para}

\begin{thm}
\label{main theorem}
The data
\begin{equation*}
(\coh^q(Y,K),W[q],F,N_K,S_q)
\end{equation*}
is a polarized mixed Hodge structure over $\bQ$
in the sense of Cattani-Kaplan-Schmid \cite[Definition (2.26)]{CKS}.
\end{thm}

\providecommand{\bysame}{\leavevmode\hbox to3em{\hrulefill}\thinspace}
\providecommand{\MR}{\relax\ifhmode\unskip\space\fi MR }
\providecommand{\MRhref}[2]{%
  \href{http://www.ams.org/mathscinet-getitem?mr=#1}{#2}
}
\providecommand{\href}[2]{#2}

\end{document}